\documentclass[12pt]{amsart}
\usepackage{amsmath,eucal,verbatim,amsopn,amsthm,amsfonts,amssymb,mathdots,mathrsfs,esint,mathtools,enumitem,mnsymbol,bbm}
\usepackage{amsmath,amssymb,latexsym}
\usepackage{amscd}
\usepackage[margin=0.9in]{geometry}
\usepackage[hyperfootnotes=false,colorlinks=true,allcolors=blue]{hyperref}

\theoremstyle{plain}
\newtheorem{theorem}{Theorem}
\newtheorem{proposition}[theorem]{Proposition}
\newtheorem{corollary}[theorem]{Corollary}
\newtheorem{lemma}[theorem]{Lemma}

\newtheorem{conjecture}[theorem]{Conjecture}
\newtheorem{remark}[theorem]{Remark}
\newtheorem{theo}{Theorem}[]

\newtheorem*{example}{Example}
\newtheorem{apptheorem}{Theorem}[section]

\newtheorem{appcorollary}[apptheorem]{Corollary}

\numberwithin{equation}{section}

\def\Z{\mathbb{Z}}
\def\R{\mathbb{R}}
\def\A{\mathbb{A}}

\def\C{\mathbb{C}}
\def\Mat{\text{Mat}}

\DeclareMathOperator{\tr}{tr}

\DeclareMathOperator{\Hom}{Hom}

\DeclareMathOperator{\diag}{diag}
\DeclareMathOperator{\Ind}{Ind}
\DeclareMathOperator{\GL}{GL}
\DeclareMathOperator{\SL}{SL}
\DeclareMathOperator{\SO}{SO}

\DeclareMathOperator{\Sp}{Sp}
\DeclareMathOperator{\Real}{Re}

\begin{document}

\title[$L$-functions for symplectic coverings]{Doubling Constructions: the complete $L$-function for coverings of the symplectic group}
\author{Eyal Kaplan}

\address{Department of Mathematics, Bar Ilan University, Ramat Gan 5290002, Israel}
\email{kaplaney@gmail.com}

\thanks{This research was supported by the ISRAEL SCIENCE FOUNDATION (grant No. 421/17).}
\begin{abstract}
We develop the local theory of the generalized doubling method for the $m$-fold central extension $\Sp_{2n}^{(m)}$ of Matsumoto of the symplectic group. We define local $\gamma$-, $L$- and $\epsilon$-factors for pairs of genuine representations of $\Sp_{2n}^{(m)}\times\widetilde{\GL}_k$ and prove their fundamental properties, in the sense of Shahidi. Here $\widetilde{\GL}_k$ is the central extension of $\GL_k$ arising
in the context of the Langlands--Shahidi method for covering groups of $\Sp_{2n}\times\GL_k$.
We then construct the complete $L$-function for cuspidal representations and prove its global functional equation.
Possible applications include classification results and a Shimura type lift of representations from covering groups to general linear groups
(a global lift is sketched here for $m=2$).
\end{abstract}
\maketitle
\addtocontents{toc}{\protect\setcounter{tocdepth}{1}}
\section*{Introduction}\label{Introduction}

One of the major challenges in the study of central extensions of classical (or reductive) groups, in the context of automorphic forms and representation theory, is the definition and characterization of local factors. In the linear setting, a conclusive local theory was provided by Shahidi \cite[Theorem~3.5]{Sh3}, as the culmination of a line of works on his method of local coefficients
(e.g., \cite{Shahidi1985,Sh3}). This method is based on the uniqueness of Whittaker models for quasi-split reductive groups, and as a consequence, the local factors were developed for irreducible representations affording this model, namely generic representations. The local theory and its global counterpart, for globally generic cuspidal representations, has since played a major role in numerous works, in particular in the functoriality results (e.g., \cite{CKPS2,CKPS}). Unfortunately, for covering groups as a rule Whittaker models are not unique.

In the recent work \cite{CFK} a different approach was pursued, using an integral representation to define and study local factors
for arbitrary representations (generic or otherwise) of classical groups, twisted by representations of general linear groups. The construction was based on the extension in \cite{CFGK2,DimaKaplan} of the doubling method of Piatetski-Shapiro and Rallis \cite{PSR} from rank-$1$ twists to arbitrary rank $k$. In the recent work \cite{me12}, the generalized doubling method of \cite{CFGK2} was extended to
central extensions of the symplectic group.

Let $n,m$ and $k$ be positive integers.
Let $G=\Sp_{2n}$ over a local field $F$, and $G^{(m)}$ be the topological central extension of $G$ by $\mu_m$ constructed by
Moore \cite{Moore}, Steinberg \cite{Stein} and Matsumoto \cite{Mats} (this covering is essentially unique). Here $\mu_m$ is the group of $m$-th roots of unity. Restricting $\Sp^{(m)}_{2(n+k)}$ to the preimage of a Levi subgroup
$\GL_k\times\Sp_{2n}$, one obtains a second covering group $\widetilde{\GL}_k$, which is one of the coverings constructed by \cite{BD}
(see also \cite{Savin7,Gao4}, it is not a coverings from \cite{KP}). Also fix a nontrivial additive character $\psi$ of $F$.

In the present work we develop the local theory of the generalized doubling method for $\Sp_{2n}^{(m)}\times\widetilde{\GL}_k$ and for
$\widetilde{\GL}_{n}\times\widetilde{\GL}_k$ arising from \cite{me12}. In the process we introduce a new family of
Rankin--Selberg integrals for $\widetilde{\GL}_{n}\times\widetilde{\GL}_k$, which extend the linear constructions from \cite[Appendix~C]{CFK} and \cite{G7} and replace the Rankin--Selberg $\GL_n\times\GL_k$ factors of \cite{JPSS} within the doubling constructions. Note that if $m>2$ and $k>1$ the results of this work are conditional, and in all cases we assume the field contains $\mu_{2m}$, see below.
\begin{theo}[see Theorem~\ref{theorem:ten commendments}]\label{theo:fundamental}
Let $\pi$ and $\tau$ be a pair of genuine irreducible admissible representations of $G^{(m)}$ and $\widetilde{\GL}_k$,
and assume $\tau$ admits a Whittaker model (usually more than one). There exists a $\gamma$-factor
$\gamma(s,\pi\times\tau,\psi)$ which satisfies the fundamental list of properties of \cite[Theorem~3.5]{Sh3}, \cite{LR} and \cite{CFK}.
These properties characterize this factor uniquely when $|m|=1$ in $F$ (or $F=\C$).
\end{theo}

The existence of the $\gamma$-factor at all places, especially when data are ramified, is based on the recent local uniqueness result of Gourevitch and the author \cite{DimaKaplan} (see \S~\ref{The local integrals} and \eqref{Homspace}).
As a consequence of Theorem~\ref{theorem:ten commendments} we define local $L$- and $\epsilon$- factors at all places, which are the expected factors when data are unramified or over $\C$. The integral representation and local theory imply the following global result:
\begin{theo}[see Theorem~\ref{theorem:complete Sp L function}]
Let $\pi$ and $\tau$ be a pair of genuine cuspidal representations of $G^{(m)}(\A)$ and $\widetilde{\GL}_k(\A)$.
There is a complete $L$-function $L(s,\pi\times\tau)$, defined as an absolutely convergent Euler product for $\Real(s)\gg0$ and by meromorphic continuation to $\C$. The $L$-function satisfies a standard functional equation $L(s,\pi\times\tau)=\epsilon(s,\pi\times\tau)L(1-s,\widetilde{\pi}\times\widetilde{\tau})$.
\end{theo}

If $S$ is a large finite set of places of a number field $F$, $L(s,\pi\times\tau)=L_S(s,\pi_{\nu}\times\tau)L^S(s,\pi\times\tau)$ where $L_S$ denotes the finite product of local $L$-factors over $\nu\in S$, and $L^S(s,\pi\times\tau)$ is the partial $L$-function which can be located in the constant term of a suitable Eisenstein series (there is an implicit non-canonical parameter here when $m\equiv2\,(4)$, see \S~\ref{unr L functions}). This point is important for the understanding of our choice of covering for $\GL_k$. In the linear setting, the Langlands--Shahidi method can be used to study this partial $L$-function by regarding $G\times\GL_k$ as a standard Levi subgroup of $\Sp_{2(n+k)}$. This method was extended to the general class of covering groups of Brylinski and Deligne \cite{BD} in a recent work of Gao \cite{Gao2018}, and indeed in this case the study of
twisted $L$-functions for $G^{(m)}$ must involve the covering $\widetilde{\GL}_k$ as defined by restriction.
Theorem~\ref{theorem:complete Sp L function} refines the results one obtains via the method of the constant term, which does not provide a global functional equation (even in the linear case).

Let $\pi$ be a genuine (irreducible) cuspidal representation of $G^{(2)}(\A)$, and $\Pi$ be an irreducible automorphic representation of $\GL_{2n}(\A)$. We say that $\Pi$ is a Shimura lift of $\pi$ if for almost all finite places $\nu$ where $\pi_{\nu}$ is unramified,
$\Pi_{\nu}$ is the local functorial lift of $\pi_{\nu}$ given by the Satake isomorphism (\cite{Satake63,McNamara}) with respect to $\psi_{\nu}$ and the embedding $\Sp_{2n}(\C)<\GL_{2n}(\C)$ (see \S~\ref{unr L functions} for details),
and at the complex places $\Pi_{\nu}$ is the lift to $\GL_{2n}(\C)$ of the representation $\theta_{\psi_{\nu}}(\pi_{\nu})$ attached to $\pi_{\nu}$ by the theta correspondence (\cite{Howe1989,AdamsBarbasch1995}). This definition extends to any $m\geq1$, but for odd $m$ the lift is to $\GL_{2n+1}(\A)$, the embedding is of $\SO_{2n+1}(\C)$ in $\GL_{2n+1}(\C)$, and at the archimedean places one takes the standard lift (see \eqref{eq:Archimedean property}). When we combine our results for the case $m=2$ with the framework of \cite{CFK} we obtain
a global lift:
\begin{theo}[see Theorem~\ref{theo:globl functorial lift}]
Any genuine cuspidal representation $\pi$ of $G^{(2)}(\A)$ has a Shimura lift $\Pi$ to $\GL_{2n}(\A)$.
\end{theo}
For the uniqueness of the lift for globally generic representations see Corollary~\ref{corollary:globl functorial lift}.

Of course the Shimura lift for $G^{(2)}(\A)$ was already proved by Gan and Ichino \cite{GanIchino2018}, using the theta correspondence
results of Gan and Savin \cite{WGS}, the theta lift of Li \cite{JSLi1997} to $\SO_{2l+1}(\A)$ for $l\gg n$ (which is known to be nonzero) and the work of Arthur \cite{Arthur2013} on the endoscopic classification. In their work Gan and Ichino \cite{GanIchino2018} were able to described the generic part of the automorphic discrete spectrum of $G^{(2)}(\A)$. However, the case of $G^{(2)}$ is an exception due to several reasons, most importantly because the representation theories of $G^{(2)}$ and $\SO_{2n+1}$ are related via the theta correspondence (\cite{Howe1989}, see also \cite{Gan2014} and the references therein), which is not available for $m>2$ (in addition, e.g., Whittaker models are unique for $G^{(2)}$). The method presented here is independent of the trace formula and its prerequisites, but more importantly: It is not limited in the degree of the covering group, and we expect to use it in order to generalize Theorem~\ref{theo:globl functorial lift} to all $m$.

Our proof of Theorem~\ref{theo:globl functorial lift} generalizes the classical work of Shimura. Shimura studied modular forms of half-integral weight and was able to produce a lift of modular forms of weight $k/2$, where $k\geq3$ is an odd integer, to modular forms of weight $k-1$. His result was obtained via the combination of an integral representation generalizing Rankin \cite{Rankin1939}, and the Converse Theorem of Weil \cite{Weil1967}. Here we apply the Converse Theorem of Cogdell and Piatetski-Shapiro (\cite{CPS3})

As mentioned above, for $m=2$ or $k=1$ our results are unconditional. In the general case they depend on several assumptions,
some of which were already present in \cite{me12} (for the local aspects here, stronger assumptions are needed).
These are detailed in \S~\ref{the representation rho_c(tau)} and essentially boil down to the existence of
a Shimura type correspondence for coverings of general linear groups. (Note that even the structure of square-integrable representations of covering groups of $\GL_k$ is at present unknown.) To the best of our knowledge, this is the first definition and study of local factors over any covering group outside the group $G^{(2)}$, as well as the first construction of complete $L$-functions, already for $k=1$.

The assumption $\mu_{2m}\subset F^*$, which is stronger than the condition $\mu_m\subset F^*$ needed for the existence of $G^{(m)}$, is needed because we use (here and in \cite{me12}) results from \cite{BBF2011,McNamara2,Gao5,Gao4}, where it was assumed. This assumption greatly simplifies the formulas and reduces the burden and dependence on some technical details.

The classical doubling method of Piatetski-Shapiro and Rallis \cite{PSR} produced an integral representation for the
standard automorphic $L$-function of a cuspidal representation $\pi$ of a classical group or its rank-$1$ twists. This construction has had numerous applications, e.g., within the theta correspondence \cite{KudlaRallis1994,HKS,WGS,Yamana}, and to cohomological automorphic representations \cite{HarrisLiSkinner2005,HarrisLiSkinner2006,EischenHarrisLiSkinner}.
The local theory was fully developed by Lapid and Rallis \cite{LR}, and extended to $G^{(2)}\times\widetilde{\GL}_1$ by Gan \cite{Gan}. Yamana \cite{Yamana} further developed the local theory of these integrals, for the purpose of characterizing the nonvanishing of the global theta lift by means of $L$-functions and the local theta correspondence.

As mentioned above, Shahidi's method of local coefficients breaks down for covering groups (except for
$G^{(2)}$ again, where we have \cite{Dani,Dani2}). In recent works Gao \textit{et. al.} \cite{GaoShahidiSzpruch2018,GaoShahidiSzpruch2019} and Szpruch \cite{Szpruch2018} replaced the local coefficient with a proportionality matrix,
for generic representations. The determinant of this matrix becomes an invariant of the representation, and for unramified representations (and ramified principal series in low rank cases) they expressed this determinant in terms of Plancherel measures and Tate $\gamma$-factors.
Let us also mention the works
\cite{Mezo2001,Savin5,McNamara,Li,Weissman2,GG,GanFanWeissman2018,Weissman2018a}
on the extension of the Langlands Program to covering groups (among the earlier works see \cite{Flicker2,FK,Savin6}).

\subsection*{Acknowledgments}
We are happy to thank Dani Szpruch for valuable discussions.

\tableofcontents

\section{Preliminaries}\label{Preliminaries}

\subsection{The groups}\label{Groups}
In a local context $F$ is a local field of characteristic $0$. If $F$ is non-archimedean, we let $\mathcal{O}$ be its ring of integers and $\varpi$ be a uniformizer with $|\varpi|=q^{-1}$. In this case when we say that a property is valid outside a discrete subset of $s$, we mean outside a finite set of values of $q^{-s}$. Identify algebraic $F$-groups $G$ with their groups of $F$-points, i.e., $G=G(F)$.
Globally $F$ will denote a number field with a ring of adeles $\A$, and we write $F$-points or $\A$-points explicitly. When both situations are treated simultaneously we simply write $G$.

For the group $\GL_d$, let $B_{\GL_d}=T_{\GL_d}\ltimes N_{\GL_d}$ denote the Borel subgroup of upper triangular invertible matrices,
where $T_{\GL_d}$ is the diagonal torus. For a composition $\beta=(\beta_1,\ldots,\beta_l)$ of $d$, $P_{\beta}=M_{\beta}\ltimes V_{\beta}$ denotes the standard parabolic subgroup with $M_{\beta}=\GL_{\beta_1}\times\ldots\times\GL_{\beta_l}$.
For an integer $c\geq1$, $\beta c=(\beta_1 c,\ldots,\beta_l c)$. Let $w_{\beta}$ be the permutation matrix with the identity blocks $I_{\beta_i}$ on its anti-diagonal, starting with $I_{\beta_1}$ on the top right. In particular set $J_d=w_{(1^d)}$. Let
$W_{\GL_d}$ denote the Weyl group of $\GL_d$, identified with the subgroup of permutation matrices. The abelian group of $d\times d'$ matrices over a ring $R$ is denoted $\Mat_{d\times d'}(R)$ and $\Mat_d(R)=\Mat_{d\times d}(R)$. The trace map is denoted $\tr$ and ${}^tx$ is the transpose of a matrix $x$. For an integer $m$, $R^{*m}=\{a^m:a\in R^*\}$. For $b\in\GL_d$, $b^*=J_d{}^tb^{-1}J_d$.

Let $\Sp_{2l}=\{g\in \SL_{2l}:{}^tg\left(\begin{smallmatrix}&J_l\\-J_l\end{smallmatrix}\right)g=\left(\begin{smallmatrix}&J_l\\-J_l\end{smallmatrix}\right)\}$. Fix $B_{\Sp_{2l}}=\Sp_{2l}\cap B_{\GL_{2l}}=T_{\Sp_{2l}}\ltimes N_{\Sp_{2l}}$ where $N_{\Sp_{2l}}$ is a maximal unipotent subgroup. Denote the Weyl group of $\Sp_{2l}$ by $W_{\Sp_{2l}}$.

For any parabolic subgroup $P$, $\delta_P$ denotes the modulus character. For a unipotent subgroup $V$, the opposite unipotent subgroup is $V^-$. For any group $H$, $C_H$ is the center of $H$, for $x,y\in H$, ${}^xy=xyx^{-1}$ and for $Y<H$, ${}^xY=\{{}^{x}y:y\in Y\}$.

Define a global maximal compact subgroup $K_G=\prod_{\nu}K_{G,\nu}$ as in \cite[\S~I.1.4]{MW2}. In particular for almost all $\nu$,
$K_{G,\nu}=G(\mathcal{O}_{\nu})$. In this work when we consider $K_G$ in a local context, it is always the group $G(\mathcal{O})$.

\subsection{Representations}\label{reps}
In this work all representations act on complex vector spaces. Local representations are assumed to be smooth.
Essentially tempered representations (and supercuspidal representations in particular) are implicitly irreducible.
When the field is archimedean, an admissible representation is understood to be admissible Fr\'{e}chet of moderate growth.
We take generic representations to be admissible by definition.
Induction of representations from parabolic subgroups is always implicitly normalized, and over archimedean fields it is the smooth induction.
Globally, cuspidal representations are always irreducible.
The action of a group by right-translation is denoted $\cdot$.

Consider a representation $\pi$ of a unipotent subgroup $U$ over a local field, acting on a space $V$. Let $\psi$ be a character of $U$. The Jacquet module
$J_{U,\psi}(\pi)$ is the quotient $V(U,\psi)\backslash V$, where over non-archimedean fields $V(U,\psi)\subset V$ is the subspace spanned by all vectors of the form $\pi(u)\xi-\psi(u)\xi$, $u\in U$ and $\xi\in V$, and over archimedean fields $V(U,\psi)$ is the closure of
this subspace. Denote $J_{U,1}(\pi)=J_{U}(\pi)$. We use non-normalized Jacquet functors.

Let $\psi$ be a nontrivial additive character of $F$ (locally) or $F\backslash\A$. For $a\in F^*$ set $\psi_a(x)=\psi(ax)$. Over a local field, the Weil index of $x\mapsto\psi(x^2)$ is denoted $\gamma(\psi)$ and the Weil factor of $\psi$ is $\gamma_{\psi}(a)=\gamma(\psi_a)/\gamma(\psi)$ (see \cite[p.~176]{We}). The global Weil factor is then $\gamma_{\psi}=\prod_{\nu}\gamma_{\psi_{\nu}}$.

For $l\geq1$ and $v\in V_{(c^l)}$, write $v=(v_{i,j})_{1\leq i,j\leq l}$ with $v_{i,j}\in\Mat_c$. Define a character of $V_{(c^l)}$ by
\begin{align}\label{def:psi_l}
\psi_{l}(v)=\psi(\sum_{i=1}^{l-1}\tr(v_{i,i+1})).
\end{align}
When clear from the context, we omit $l$ from the notation and simply write $\psi$.

Let $G$ be either $\GL_d$ or $\Sp_{2d}$, over $\C$. Any irreducible admissible representation $\pi$ of $G$ is the
unique irreducible quotient of a representation $\Ind_{B_{G}}^{G}(\otimes_{i=1}^d|\det|^{a_i}\pi_i)$,
where $\pi_i$ are tempered quasi-characters of $F^*$ and $a_1>\ldots>a_d$.
This inducing data is unique. For any integer $r\geq1$, we let $\pi^r$ be the unique irreducible quotient of
$\Ind_{B_{G}}^{G}(\otimes_{i=1}^d|\det|^{ra_i}\pi_i^r)$.

\subsection{Covering groups}\label{Covering groups}
We briefly describe our conventions for covering groups. The basic reference for this section is \cite{Moore,GanFanWeissman2018}. For more details see \cite[\S~1.2]{me12}.

Let $\mu_m\subset\C^*$ be the cyclic group of $m$-th roots of unity and assume $\mu_m\subset F^*$. Fix a Hilbert
symbol $(\cdot,\cdot)_m$ of order $m$ in $F$, globally this is the product of local symbols (for $m=1$, by definition $(\cdot,\cdot)_m\equiv1$). Let $H$ be a locally compact group. By a topological central extension of $H$ by $\mu_m$ we mean a short exact sequence of topological groups
\begin{align*}
1\rightarrow \mu_m\xrightarrow{i} \widetilde{H}\xrightarrow{p} H\rightarrow 1,
\end{align*}
where $i(\mu_m)$ is closed and belongs to the center of $\widetilde{H}$. Call $\widetilde{H}$ an $m$-fold covering group of $H$.

A $2$-cocycle $\sigma$ of $H$ is a Borel measurable function $\sigma:H\times H\rightarrow\mu_m$ such that for all
$h,h',h''\in H$,
\begin{align}\label{eq:2-cocycle}
\sigma(h,h')\sigma(hh',h'')=\sigma(h,h'h'')\sigma(h',h''),\qquad \sigma(e,h')=\sigma(h,e)=1.
\end{align}
Here $e$ is the identity element of $H$. Let $\mathrm{Z}^2(H,\mu_m)$ denote the group of $2$-cocycles. Granted $\sigma$, we realize $H^{(m)}$ as the set of elements
$\langle h,\epsilon\rangle$, $h\in H$, $\epsilon\in\mu_m$, with the product
\begin{align*}
\langle h,\epsilon\rangle\langle h',\epsilon'\rangle=\langle hh',\epsilon\epsilon'\sigma(h,h')\rangle.
\end{align*}
We use the notation $H[\sigma]$ to emphasize the $2$-cocycle in use, e.g., when $2$ different realizations $H[\sigma]$ and $H[\sigma']$ are considered simultaneously ($\sigma$ and $\sigma'$ may not be cohomologous).

A Borel measurable map $\eta:H\rightarrow\mu_m$ such that $\eta(e)=1$ is called a $1$-cochain, the group of $1$-cochains is denoted $\mathrm{C}^1(H,\mu_m)$. For $\eta\in\mathrm{C}^1(H,\mu_m)$, the function $(h,h')\mapsto \eta(h)\eta(h')/\eta(hh')$ is called a $2$-coboundary and we have the group $\mathrm{B}^2(H,\mu_m)$ of $2$-coboundaries. The $m$-fold coverings $\widetilde{H}$ are parameterized by the $2$-nd cohomology $\mathrm{H}^2(H,\mu_m)=\mathrm{B}^2(G,\mu_m)\backslash \mathrm{Z}^2(H,\mu_m)$.

If $\sigma,\rho\in \mathrm{Z}^2(H,\mu_m)$ are equal in $\mathrm{H}^2(H,\mu_m)$, i.e., cohomologous,
one can find $\eta\in\mathrm{C}^1(H,\mu_m)$ satisfying
\begin{align}\label{eq:cohomologous}
\rho(h,h')=\frac{\eta(h)\eta(h')}{\eta(hh')}\sigma(h,h'),\qquad\forall h,h'\in H.
\end{align}
Then $\langle h,\epsilon\rangle \mapsto \langle h,\epsilon\eta(h)\rangle $ is a topological isomorphism of $\widetilde{H}$, where the domain is realized using
$\rho$ and the image by $\sigma$.

A section of $X<H$ is a continuous map $x\mapsto\langle x,\eta(x)\rangle$ where
$\eta\in\mathrm{C}^1(X,\mu_m)$. This map is a splitting of $X$ if it is also a homomorphism, which means
\begin{align*}
\langle x,\eta(x)\rangle\langle x',\eta(x')\rangle=\langle xx',\eta(xx')\rangle,\qquad\forall x,x'\in X.
\end{align*}
In this case we say that $\widetilde{H}$ splits over $X$. Granted two splittings
$x\mapsto\langle x,\eta(x)\rangle$ and $x\mapsto\langle x,\eta'(x)\rangle$, the abstract map $x\mapsto\eta(x)\eta'(x)^{-1}$ is in
$\Hom(X,\mu_m)$. In particular if $X$ is the $F$-points or $\A$ points of an algebraic unipotent subgroup, or if $X$ is perfect (as an abstract group), $\eta=\eta'$. Moreover since $x\mapsto \langle x,\eta(x)\rangle\langle x,\eta'(x)\rangle^{-1}=\langle e,\eta(x)\eta'(x)^{-1}\rangle$ is continuous, if $i(\mu_m-\{1\})$ is closed in $\widetilde{H}$, $x\mapsto\eta(x)\eta'(x)^{-1}$ is also continuous (throughout, $\widetilde{H}$ will always be Hausdorff).

Since $\widetilde{H}$ is a central extension of $H$, $H$ acts on $\widetilde{H}$ by conjugation (a homeomorphism).
Then ${}^{\langle x,\epsilon'\rangle}\langle y,\epsilon\rangle=\langle{}^xy,\sigma(x,y)\sigma(xy,x^{-1})\epsilon\rangle$ (independent of $\epsilon'$).

When $H=\Sp_{2l}$ or $\SL_l$, let $H^{(m)}$ be the $m$-fold covering group $\widetilde{H}$ of $H$
defined by \cite{Mats} (following \cite{Moore,Stein}) with the Steinberg symbol constructed from $(\cdot,\cdot)_m^{-1}$.
The group $H^{(m)}$ (locally or globally) is locally compact, and over non-archimedean fields it is an $l$-group in the sense of \cite[1.1]{BZ1}.
The group $H^{(m)}$ is split canonically over its (algebraic) unipotent subgroups, see \cite{BLS} and \cite[Appendix~I]{MW2}. In addition $H^{(m)}(\A)$ is split canonically over $H(F)$.
We say that a local field $F$ is unramified, if it is non-archimedean, $|m|=1$, $q$ is odd and $q>3$. In this case
$H^{(m)}$ is split canonically over $K_{H}(=H(\mathcal{O}))$ (\cite{Moore}).

When $X$ is a closed subgroup of $H$ (or plainly a topological subgroup), we can consider the $m$-fold covering $\widetilde{X}$ of $X$ obtained by restriction from $H^{(m)}$. In general this covering depends on the embedding of $X$ in $H$, but this embedding will always be made clear.

Fix a faithful character $\varepsilon:\mu_m\rightarrow\C^*$ to be used throughout. An $\varepsilon$-genuine representation of $\widetilde{H}$ is then a representation where $\mu_m$ acts by $\varepsilon$. The character $\varepsilon$ will usually be omitted,
and the term anti-genuine representation stands for an $\varepsilon^{-1}$-genuine representation. Induction from $\widetilde{P}$ to $\widetilde{G}$ for a parabolic subgroup $P<G$ is implicitly normalized by $\delta_P^{1/2}$ (as in the linear case).

In the rest of this work we assume $\mu_{2m}<F^*$. When $F$ is archimedean, this implies $F=\C$, then $H^{(m)}$ is split over $H$ (canonically since $H$ is perfect), so that the archimedean arguments usually reduce to the linear case.

\subsection{Local coverings}\label{Local coverings}
Let $F$ be a local field.
We collect several properties of the local cocycles we will use. For more details see \cite[\S~1.4, \S~1.7]{me12}. All formulas here
were obtained from \cite{BLS}. For a positive integer $l$, let $\sigma_{\SL_{l+1}}\in \mathrm{Z}^2(\SL_{l+1},\mu_m)$ be the $2$-cocycle of \cite[\S~2]{BLS} which represents $\SL_{l+1}^{(m)}$, and denote by $\sigma_l$ the $2$-cocycle of $\GL_l$ of \textit{loc. cit.} given by
\begin{align*}
\sigma_l(b,b')=(\det b,\det b')_m\sigma_{\SL_{l+1}}(\diag(b,\det b^{-1}),\diag(b',\det {b'}^{-1})),\qquad b,b'\in\GL_l.
\end{align*}
Recall the block-compatibility formula \cite[Theorem~11]{BLS}: for $0<l_0<l$,
$a,a'\in\GL_{l_0}$ and $b,b'\in\GL_{l-l_0}$,
\begin{align}\label{eq:BLS block compatible}
\sigma_l(\diag(a,b),\diag(a',b'))=(\det a,\det b')_m\sigma_{l_0}(a,a')\sigma_{l-l_0}(b,b').
\end{align}
We also have the following properties: for $t,t'\in T_{\GL_l}$, $b,b'\in\GL_l$, $v,v'\in N_{\GL_l}$, $u^-\in N_{\GL_l}^-$, and
if $u^-\mapsto\langle u^-,\varsigma(u^-)\rangle$ is the splitting of $N_{\GL_l}^-$ in the covering of $\GL_l$ defined by $\sigma_l$,
\begin{align}
\label{eq:sigma on torus of GL}
&\sigma_l(t,t')=\prod_{i<j}(t_i,t'_j)_m,\qquad t=\diag(t_1,\ldots,t_l),\quad t=\diag(t'_1,\ldots,t'_l),\\
\label{eq:sigma on h and v}
&\sigma_l(b,v')=\sigma_l(v,b')=1, \\
\label{eq:sigma on vh and h'v'}
&\sigma_l(vb,b'v')=\sigma_l(b,b'),\\
\label{eq:sigma conjugate v by h}
&{}^bv\in N_{\GL_l}\Rightarrow {}^b\langle v,1\rangle=\langle {}^bv,1\rangle,\\
\label{eq:sigma conjugate v- to v by h}
&{}^bu^-\in N_{\GL_l}\Rightarrow{}^b\langle u^-,\varsigma(u^-)\rangle=\langle {}^bu^-,1\rangle.
\end{align}
Let $\mathfrak{W}^+_l<\GL_l$ be the subgroup generated by $W_{\GL_l}$ and $\{t\in T_{\GL_l}:t_i=\mp1,\forall i\}$. Our assumption $\mu_{2m}\subset F^*$ implies that $w\mapsto\langle w,1\rangle$ is a homomorphism of $\mathfrak{W}^+_l$ and also that
\begin{align}\label{eq:conj mathcal t in GLd}
{}^{w}\langle t,1\rangle=\langle {}^wt,\prod_{(i,j)=\alpha>0:w\alpha<0}(t_j,t_i)_m\rangle.
\end{align}
Here we identify the positive roots of $\GL_l$ with the pairs $(i,j)$, $i<j$.

Consider the involution $b\mapsto b^*$ ($b^*=J_l{}^tb^{-1}J_l$) of $\GL_l$. Define
\begin{align*}
\sigma^*_l(b,b')=\sigma_l(b^*,{b'}^*).
\end{align*}
This is again a $2$-cocycle of $\GL_l$ which is cohomologous to $\sigma_l$ by \cite[Proposition~4, Remark~5, Proposition~20]{me12}.
Let $\varsigma_{*,l}\in\mathrm{C}^1(\GL_l,\mu_m)$ be such that
\begin{align}\label{eq:sigma l *}
\sigma^*_l(b,b')=\frac{\varsigma_{*,l}(b)\varsigma_{*,l}(b')}{\varsigma_{*,l}(bb')}\sigma_l(b,b'),\qquad \forall b,b'\in \GL_l.
\end{align}
For an integer $j$, let $\sigma^{*,j}_l\in\mathrm{Z}^2(\GL_l,\mu_m)$ be given by
\begin{align}\label{eq:sigma l * rk}
\sigma^{*,j}_l(b,b')=\left(\frac{\varsigma_{*,l}(b)\varsigma_{*,l}(b')}{\varsigma_{*,l}(bb')}\right)^{j}\sigma^*_l(b,b')
=\left(\frac{\varsigma_{*,l}(b)\varsigma_{*,l}(b')}{\varsigma_{*,l}(bb')}\right)^{j+1}\sigma_l(b,b').
\end{align}
The cocycles $\sigma^{*}_l,\sigma^{*,j}_l$ and $\sigma_l$ are all cohomologous.

Let $d$ be a positive integer. We realize the group $\Sp_{2d}^{(m)}$ using $\sigma_{2d}$. By \eqref{eq:sigma on torus of GL},
\begin{align}\label{eq:BLS $2$-cocycle on torus}
\sigma_{2d}(t,t')=\prod_{i=1}^{d}(t_i,t'_i)_m^{-1},\qquad t=\diag(t_1,\ldots,t_{d},t_{d}^{-1},\ldots,t_{1}^{-1}),t'\in T_{\Sp_{2d}}.
\end{align}
Identities \eqref{eq:sigma on h and v}--\eqref{eq:sigma conjugate v- to v by h} apply in particular to
$b,b'\in\Sp_{2d}$, $v,v'\in N_{\Sp_{2d}}$ and $u^-\in N_{\Sp_{2d}}^-$. When $F$ is unramified, $K_{\Sp_{2d}}$ is perfect and we let
$\eta_{2d}\in\mathrm{C}^1(K_{\Sp_{2d}},\mu_m)$ be the unique $1$-cochain such that
\begin{align}\label{eq:splitting of K H}
\sigma_{2d}(y,y')=\frac{\eta_{2d}(yy')}{\eta_{2d}(y)\eta_{2d}(y')},\qquad\forall y,y'\in K_{\Sp_{2d}}.
\end{align}
    According to \cite[Proposition~0.I.3]{KP} and \cite[(1.3) and p.~183]{Tk} (see \cite[p.~16]{me12}), $\eta_{2d}$ is trivial on
\begin{align*}
N_{\Sp_{2d}}\cap K_{\Sp_{2d}},\qquad T_{\Sp_{2d}}\cap K_{\Sp_{2d}},\qquad \Sp_{2d}\cap \mathfrak{W}^+_{2d}.
\end{align*}

\begin{proposition}\cite[Proposition~2]{me12}\label{proposition:action of W on torus is trivial on Sp}
Let $w\in \Sp_{2d}\cap \mathfrak{W}^+_{2d}$. For any $t\in T_{\Sp_{2d}}$, ${}^w\langle t,1\rangle=\langle {}^wt,1\rangle$.
\end{proposition}

Let $r=m$ when $m$ is odd and $m/2$ otherwise.
Identify $\GL_d$ with the standard Siegel Levi subgroup of $\Sp_{2d}$ by $b\mapsto\diag(b,b^*)$. We obtain
a covering group $\widetilde{\GL}_d$ by restriction from $\Sp_{2d}^{(m)}$, denote it by $\GL_{d}^{(m)}$.
This group was previously studied in \cite{Savin7,Gao4} and was denoted $\GL_{d}^{(m,r)}$ in \cite{me12} to emphasize the
difference between this covering and the coverings of \cite{KP} (here we shall favor a uniform presentation).
By definition $\GL_{d}^{(m)}$ is realized via the $2$-cocycle
\begin{align}\label{eq:sigma square}
\sigma^{\diamondsuit}_{d}(b,b')=\sigma_{2d}(\diag(b,b^*),\diag(b',{b'}^*)),\qquad b,b'\in\GL_d.
\end{align}
We then have
\begin{align}\label{eq:Nice GL $2$-cocycle on torus}
\sigma^{\diamondsuit}_{d}(\diag(t_1,\ldots,t_{d}),\diag(t_1',\ldots,t_{d}'))=\prod_{i=1}^{d}(t_i,t'_i)_m^{-1},
\end{align}
and \eqref{eq:sigma on h and v}--\eqref{eq:sigma conjugate v- to v by h} remain valid for
$\sigma^{\diamondsuit}_{d}$ (instead of $\sigma_{d}$, and for \eqref{eq:sigma conjugate v- to v by h} we choose $\varsigma$ with respect to $\sigma^{\diamondsuit}_{d}$). Also by Proposition~\ref{proposition:action of W on torus is trivial on Sp}, for all
$w\in\mathfrak{W}^+_{d}$ and $t\in T_{\GL_d}$,
\begin{align}\label{eq:Prop 1 for GL}
{}^w\langle t,1\rangle=\langle {}^wt,1\rangle.
\end{align}
The center of $\GL_d^{(m)}$ is $\widetilde{C}_{r,d}$ where $C_{r,d}=\{xI_d:x\in F^{*r}\}$.
We have the ``improved" block-compatibility (cf. \eqref{eq:BLS block compatible}):
Let $\beta=(\beta_1,\ldots,\beta_l)$ be a composition of $d$,
\begin{align}\label{eq:block compatibility on Levi of P}
\sigma^{\diamondsuit}_{d}(b,b')=\prod_{i=1}^l\sigma^{\diamondsuit}_{\beta_i}(b_i,b_i'),\qquad b=\diag(b_1,\ldots,b_l)\in M_{\beta},\quad b'\in M_{\beta}.
\end{align}
In particular, the direct factors of $M_{\beta}$ commute in $\GL_{d}^{(m)}$, and under the embedding $b\mapsto\diag(I_i,b,I_{d-j-i})$ of $\GL_j$ in $\GL_{d}$, $\widetilde{\GL}_{j}=\GL_{j}^{(m)}$. Hence the standard tensor product $\otimes$ is well defined for genuine representations when we identify
\begin{align}\label{eq:M beta as a quotient}
\widetilde{M}_{\beta}=\{(\epsilon_1,\ldots,\epsilon_l)\in\mu_m^l:\prod_{i=1}^l\epsilon_i=1\}\backslash \GL_{\beta_1}^{(m)}\times\ldots\times \GL_{\beta_l}^{(m)}.
\end{align}

Note that \eqref{eq:BLS block compatible} and \eqref{eq:block compatibility on Levi of P} also imply
the direct factors of Levi subgroups of $\Sp_{2d}$ commute in $\Sp_{2d}^{(m)}$, so that $\otimes$ is again well defined for genuine representations.

When $F$ is unramified, put $\eta^{\diamondsuit}_{d}(y)=\eta_{2d}(\diag(y,y^*))$ and fix the splitting of $K_{\GL_d}$ to be
$y\mapsto\langle y,\eta^{\diamondsuit}_{d}(y)\rangle$, which is compatible with our choice for $\Sp_{2d}^{(m)}$.

According to \eqref{eq:BLS block compatible}, $\sigma^{\diamondsuit}_{d}(b^*,{b'}^*)=\sigma^{\diamondsuit}_{d}(b,b')$ in $\mathrm{Z}^2(\GL_d,\mu_m)$, hence the involution $b\mapsto b^*$ lifts to an involution of $\GL_d^{(m)}$ by
\begin{align}\label{eq:involution b*0}
{}^*\langle b,\epsilon\rangle =\langle b^*,\epsilon\rangle.
\end{align}
(Cf. \cite{Kable3}.)
Fixing this lift (this is the unique lift only when $m$ is odd), for a genuine representation $\pi$ of $\GL_d^{(m)}$, $\pi^*$ is defined to be the representation acting on the space of $\pi$ by $\pi^*(\langle b,\epsilon\rangle)=\pi({}^*\langle b,\epsilon\rangle)$.
We also have ${}^*(\langle b,1\rangle^{-1})=({}^*\langle b,1\rangle)^{-1}$, then the definitions imply
\begin{align}\label{eq:involution b*0 and dual}
(\pi^{\vee})^*=(\pi^*)^{\vee}.
\end{align}

\begin{proposition}\cite[Proposition~20]{me12}\label{proposition:sigma * and sigma on GLd}
$\sigma^{\diamondsuit}_{d}(b,b')$ and $\sigma^2_{d}(b,b')(\det b,\det b')_m$ are cohomologous.
\end{proposition}
Thus $\GL_d^{(m)}$ is ``morally" an $r$-fold covering of $\GL_d$ and $r$ will appear throughout.

The following lemma will be applied repeatedly to compute conjugations in integrals.
\begin{lemma}\label{lemma:conjugation commutes}
Let $G$ be $\GL_d$ (resp., $\Sp_{2d}$), $X<G$ and $w\in\mathfrak{W}^+_{d}$ (resp., $w\in\Sp_{2d}\cap\mathfrak{W}^+_{2d}$).
Assume ${}^w(X\cap N_G)<N_G$. Then ${}^w\langle x,1\rangle=\langle {}^wx,1\rangle$ for all $x\in X$.
\end{lemma}
\begin{proof}
For each simple root $\alpha=(i,i+1)$ of $\GL_{2d}$, let
$s_{\alpha}$ denote the simple reflection along $\alpha$ identified with
$w_{\alpha}=\diag(I_{i-1},\left(\begin{smallmatrix}&-1\\1\end{smallmatrix}\right),I_{2d-i-1})$. The set
$\mathfrak{W}_{2d}\subset\SL_{2d}$ was defined in \cite{BLS} as the set of elements $w_{\alpha_1}\cdot\ldots\cdot w_{\alpha_{\ell(w)}}$, where
$w$ varies over the Weyl group of $\GL_{2d}$ and $\ell(w)$ is the length of $w$. By \cite[Theorem~3.7(b)]{BLS},
$\sigma_{2d}(t,w_0)=1$ for any $t\in T_{\GL_{2d}}$ and $w_0\in\mathfrak{W}_{2d}$.

Let $x\in X$ and write $x=utw_xv$ with $u,v\in N_G$, $t\in T_G$ and $w_x\in\mathfrak{W}^+_{d}$ when $G=\GL_d$ or $w_x\in \Sp_{2d}\cap \mathfrak{W}^+_{2d}$ for $G=\Sp_{2d}$. We can write $w_x=t_0w_0$ where $t_0\in T_{\GL_{2d}}$ has entries $\mp1$ on
the diagonal and $w_0\in\mathfrak{W}_{2d}$. By \cite[Theorem~3.7(b)]{BLS} and \eqref{eq:sigma on torus of GL},
\begin{align} \label{eq:t and w_x}
\langle tw_x,1\rangle=\langle tt_0,1\rangle\langle w_0,1\rangle=\langle t,1\rangle\langle t_0,1\rangle\langle w_0,1\rangle=\langle t,1\rangle\langle w_x,1\rangle.
\end{align}
Hence by \eqref{eq:sigma on vh and h'v'},
\begin{align*}
\langle x,1\rangle=\langle u,1\rangle\langle t,1\rangle\langle w_x,1\rangle\langle v,1\rangle.
\end{align*}
Then \eqref{eq:sigma conjugate v by h}, \eqref{eq:Prop 1 for GL} and Proposition~\ref{proposition:action of W on torus is trivial on Sp} (applied to $w$ and $t$) imply
\begin{align*}
{}^w\langle x,1\rangle=\langle {}^wu,1\rangle\langle {}^wt,1\rangle\langle {}^ww_x,1\rangle\langle {}^wv,1\rangle.
\end{align*}
Note that ${}^wu,{}^wv\in N_G$ by our assumption, and we can write ${}^ww_x=t_0'w_0'$ as above, e.g., $w_0'\in\mathfrak{W}_{2d}$.
Applying \eqref{eq:t and w_x} again --- now to ${}^wt$ and ${}^ww_x$, and using \eqref{eq:sigma on vh and h'v'}, the result follows.
\end{proof}
We mention that the lemma generalizes the proof of \cite[(4.20)]{me12} when $\mu_{2m}<F^*$.

\begin{proposition}\label{proposition:irred is admissable}
Assume $F$ is non-archimedean.
Let $M$ be a Levi subgroup of $\GL_d$ or $\Sp_{2d}$, and let $\widetilde{M}$ be
obtained by restriction from an $m$-fold covering group of $\Sp_{2d}$ or $\GL_d$ (this also includes coverings from \cite{KP}). Any genuine irreducible representation of $\widetilde{M}$ is admissible.
\end{proposition}
\begin{proof}
Let $\rho$ be such a representation. By \cite[3.19, 3.13 (d)]{BZ1} we can already assume it is supercuspidal.
Write $M=M_1\times\ldots\times M_l$ ($M_l$ is either $\GL_{d'}$ or $\Sp_{2d'}$). Define
$M_i^0=\{ g\in M_i:|\det g|=1\}$ and $M^0=M_1^0\times\ldots\times M_l^0$. Then $\widetilde{M}^0\cap C_{\widetilde{M}}$ is compact
($C_{\widetilde{M}}$ - the center of $\widetilde{M}$),
$\widetilde{M}^0$ is an open normal subgroup of $\widetilde{M}$, $\widetilde{M}^0\backslash \widetilde{M}=M^0\backslash M$ is abelian and
$(\widetilde{M}^0C_{\widetilde{M}})\backslash \widetilde{M}$ is finite. By a theorem of Harish-Chandra
(\cite[3.21]{BZ1}), $\rho$ restricts to a finite representation of $\widetilde{M}^0$ (finite in the sense of \cite[2.40]{BZ1}). By \cite[3.26]{BZ1}, $\rho|_{\widetilde{M}^0}$ is a finite direct sum of
irreducible representations, and by \cite[2.41]{BZ1} a finitely generated finite representation of $\widetilde{M}^0$ (in fact, of any $l$-group) is
admissible. It follows that $\pi$ is admissible (as a representation of $\widetilde{M}$).
\end{proof}

\subsection{Global coverings}\label{Global coverings}
The bulk part of this work is local, but for global results we still need to describe the global setting.
For more details see \cite[\S~1.5, \S~1.7]{me12}. Let $F$ be a number field.
Recall the local $1$-cochains $\eta_{2d,\nu}$ defined by \eqref{eq:splitting of K H}.
For any $d$, $\Sp_{2d}^{(m)}(\A)$ is realized using a global $2$-cocycle $\rho_{2d}$, which satisfies the local relation
\begin{align}\label{eq:nu and sigma for covering of H}
\rho_{2d,\nu}(h,h')=\frac{\eta_{2d,\nu}(h)\eta_{2d,\nu}(h')}{\eta_{2d,\nu}(hh')}\sigma_{2d,\nu}(h,h'),\qquad \forall h,h'\in\Sp_{2d}(F_{\nu}).
\end{align}
By construction $\rho_{2d,\nu}=\sigma_{2d,\nu}$ in $\mathrm{H}^2(\Sp_{2d}(F_{\nu}),\mu_m)$ at each place $\nu$ of $F$. The global $1$-cochain $\eta_{2d}=\prod_{\nu}\eta_{2d,\nu}$ is defined on $\Sp_{2d}(F)$ (but not on $\Sp_{2d}(\A)$). The splitting of $\Sp_{2d}(F)$ and $N_{\Sp_{2d}}(\A)$ into $\Sp_{2d}^{(m)}(\A)$ is given by $y\mapsto\langle y,\eta_{2d}^{-1}(y)\rangle$.

Let $\rho_d^{\diamondsuit}$ denote the restriction of $\rho_{2d}$ to $\GL_d$ (identified with the standard Siegel Levi subgroup of $\Sp_{2d}$ as in \S~\ref{Local coverings}),
we realize $\GL_{d}^{(m)}(\A)$ using $\rho_d^{\diamondsuit}$, $\rho_{d,\nu}^{\diamondsuit}=\sigma_{d,\nu}^{\diamondsuit}$ in $\mathrm{H}^2(\GL_{d}(F_{\nu}),\mu_m)$. We choose the splitting of $\GL_d(F)$ into $\GL_{d}^{(m)}(\A)$ to be
$y\mapsto\langle y,(\eta_d^{\diamondsuit})^{-1}(y)\rangle$, it is also the splitting of $N_{\GL_d}(\A)$.
Throughout this work automorphic forms on $\GL_{d}^{(m)}(\A)$ are always defined with respect to the above splitting of $\GL_d(F)$.

For a composition $\beta$ of $d$ consisting of $l$ parts, we have a global block-compatible $2$-cocycle $\rho_{\beta}$ of $M_{\beta}(\A)$ such that $\rho_{\beta}(x,x')=\prod_{i=1}^l\rho_{c}^{\diamondsuit}(x_i,x'_i)$ for $x,x'\in M_{\beta}(\A)$.
There is $\eta_{\beta}\in\mathrm{C}^1(M_{\beta}(\A),\mu_m)$ such that
\begin{align}\label{Appendixeq:rho beta and rho square globally}
\rho_{\beta}(x,x')=\frac{\eta_{\beta}(x)\eta_{\beta}(x')}{\eta_{\beta}(xx')}\rho^{\diamondsuit}_{d}(x,x'),\qquad
x,x'\in M_{\beta}(\A).
\end{align}
For each place $\nu$,
\begin{align}\label{eq:eta beta nu}
\eta_{\beta,\nu}(x)=\prod_{i=1}^{l}\eta^{\diamondsuit}_{\beta_i,\nu}(x_i)/\eta^{\diamondsuit}_{d,\nu}(x),\qquad x\in M_{\beta}(F_{\nu}).
\end{align}
We can globalize \eqref{eq:eta beta nu} to $x\in M_{\beta}(F)$, i.e.,
$\eta_{\beta}(x)=\prod_{i=1}^{l}\eta^{\diamondsuit}_{\beta_i}(x_i)/\eta^{\diamondsuit}_{d}(x)$ where
$\eta^{\diamondsuit}_{\beta_i}=\prod_{\nu}\eta^{\diamondsuit}_{\beta_i,\nu}$ and similarly for $\eta^{\diamondsuit}_{d}$ (see the explanation following \cite[(1.67)]{me12}).

Let $\GL_c^{\triangle}$ denote the diagonal embedding of $\GL_c$ in $M_{(c^{rk})}$. In the following
results we compute $\rho^{\diamondsuit}_{rkc}$ on $\GL_c^{\triangle}(\A)$.
Let $\sigma_{2c}^{\mathrm{Rao}}$ denote the Rao $2$-cocycle of $\Sp_{2c}(\A)$ (\cite{Rao}).

\begin{proposition}\label{proposition:rho and Rao}
If $m\nmid rk$, $\rho_{2c}^{rk}=\sigma_{2c}^{\mathrm{Rao}}$ in $\mathrm{H}^2(\Sp_{2c}(\A),\mu_2)$.
\end{proposition}
\begin{proof}
Consider a place $\nu$ and let $\sigma_{\nu}^{\mathrm{Kubota}}$ be the Kubota $2$-cocycle of $\SL_2(F_{\nu})$ constructed
with $(,)_{2,\nu}$ (\cite{Kubota}).
Since $\rho_{2c,\nu}^{rk}=\sigma_{2c,\nu}^{rk}$ in $\mathrm{H}^2(\Sp_{2c}(F_{\nu}),\mu_2)$ and the restriction of $\sigma_{2c,\nu}^{rk}$ to $\diag(I_{c-1},\SL_2,I_{c-1})$ coincides with $\sigma_{\nu}^{\mathrm{Kubota}}$ (\cite[Lemma~2.5, Corollary~8]{BLS}),
$\rho_{2c,\nu}^{rk}$ is the nontrivial element of $\mathrm{H}^2(\Sp_{2c}(F_{\nu}),\mu_2)$. Hence
$\rho_{2c,\nu}^{rk}=\sigma_{2c,\nu}^{\mathrm{Rao}}$ in $\mathrm{H}^2(\Sp_{2c}(F_{\nu}),\mu_2)$ for all $\nu$ and therefore
$\rho_{2c}^{rk}=\sigma_{2c}^{\mathrm{Rao}}$ in $\mathrm{H}^2(\Sp_{2c}(\A),\mu_2)$ (the product of local $2$-coboundaries will be defined on $\Sp_{2c}(\A)$ because for almost all $\nu$, the local $2$-coboundary restricts to an element of the trivial space $\Hom(\Sp_{2c}(\mathcal{O}_{\nu}),\mu_2)$).
\end{proof}

If $m\nmid rk$, by Proposition~\ref{proposition:rho and Rao} we can fix $\eta_{2c}^{(rk)}\in\mathrm{C}^1(\Sp_{2c}(\A),\mu_2)$ such that
\begin{align}\label{eq:RS rho top rk}
\rho_{2c}^{rk}(h,h')=\frac{\eta_{2c}^{(rk)}(hh')}{\eta_{2c}^{(rk)}(h)\eta_{2c}^{(rk)}(h)}\sigma_{2c}^{\mathrm{Rao}}(h,h').
\end{align}
For $b\in\GL_c(\A)$ put $\eta_c^{\diamondsuit,(rk)}(b)=\eta_{2c}^{(rk)}(\diag(b,b^*))$.
Since $\sigma_{2c}^{\mathrm{Rao}}$ restricts to $(\det ,\det )_2$ on $\GL_c(\A)$
(e.g., \cite[(2.16)]{DihuaSoudry2007}),
\begin{align}\label{eq:RS rho top rk on Levi}
(\rho_{c}^{\diamondsuit})^{rk}(b,b')=\frac{\eta_c^{\diamondsuit,(rk)}(bb')}{\eta_c^{\diamondsuit,(rk)}(b)\eta_c^{\diamondsuit,(rk)}(b')}(\det b,\det b')_2,\qquad \forall\, b,b'\in\GL_c(\A).
\end{align}
When $m|rk$, $\eta_{2c}^{(rk)}$ and thereby $\eta_{c}^{\diamondsuit,(rk)}$ are taken to be the trivial function $1$.

\begin{lemma}\label{lemma:eta 2c rk and eta 2c (rk)}
We have $\eta_{2c}^{rk}\cdot\eta_{2c}^{(rk)}=1$ on $\Sp_{2c}(F)$ and in particular
$(\eta_{c}^{\diamondsuit})^{rk}\cdot\eta_{c}^{\diamondsuit,(rk)}=1$ on $\GL_{c}(F)$.
\end{lemma}
\begin{proof}
Since $h\mapsto\langle h,1\rangle$ is the splitting of $\Sp_{2c}(F)$ into $\Sp_{2c}(\A)[\sigma_{2c}^{\mathrm{Rao}}]$
(see e.g., \cite[(2.15)]{DihuaSoudry2007}), looking at \eqref{eq:RS rho top rk} and using the facts that $h\mapsto\langle h,\eta_{2c}^{rk}(h)^{-1}\rangle$ is the splitting of $\Sp_{2c}(F)$ into
$\Sp_{2c}(\A)[\rho_{2c}^{rk}]$ and
$\Hom(\Sp_{2c}(F),\mu_2)$ is trivial, we deduce $\eta_{2c}^{rk}\cdot\eta_{2c}^{(rk)}=1$ on $\Sp_{2c}(F)$.
\end{proof}
Observe that by \eqref{eq:eta beta nu} for $\beta=(c^{rk})$ and because $\eta^{\diamondsuit}_{d}$ is defined on
$\GL_d(F)$ (for all $d$), $\eta_{(c^{rk})}(b^{\triangle})=(\eta^{\diamondsuit}_{c})^{rk}(b)/\eta^{\diamondsuit}_{rkc}(b^{\triangle})$ for
$b\in\GL_c(F)$. Combining this with Lemma~\ref{lemma:eta 2c rk and eta 2c (rk)} we deduce
\begin{align}\label{eq:eta eta rk m divides or not eta diamond triangle rat}
\eta_{(c^{rk})}(b^{\triangle})\eta_c^{\diamondsuit,(rk)}(b)=(\eta^{\diamondsuit}_{rkc})^{-1}(b^{\triangle})
,\qquad \forall\, b\in \GL_c(F).
\end{align}
\begin{proposition}\label{proposition:rho rkc on diagonal GLc}
For all $b,b'\in\GL_c(\A)$,
\begin{align}\label{eq:RS rho rkc rk on Levi}
\rho^{\diamondsuit}_{rkc}(b^{\triangle},{b'}^{\triangle})=
\frac{\eta_{(c^{rk})}((bb')^{\triangle})\eta_c^{\diamondsuit,(rk)}(bb')}{\eta_{(c^{rk})}(b^{\triangle})
\eta_c^{\diamondsuit,(rk)}(b)\eta_{(c^{rk})}({b'}^{\triangle})\eta_c^{\diamondsuit,(rk)}(b')}
(\det b,\det b')_{m/r}^k.
\end{align}
Thus $\GL_{rkc}^{(m)}(\A)$ is split over $\GL_c^{\triangle}(\A)$ ($\mu_{2m}<F^*$) and
when $m|rk$, $(\eta_{rkc}^{\diamondsuit})^{-1}=\eta_{(c^{rk})}$ on $\GL_c^{\triangle}(\A)$.
\end{proposition}
\begin{proof}
By \eqref{Appendixeq:rho beta and rho square globally} for $\beta=(c^{rk})$,
\begin{align}\label{Appendixeq:rho c rk and rho square globally}
&(\rho_c^{\diamondsuit})^{rk}(b,b')=\frac{\eta_{(c^{rk})}(b^{\triangle})\eta_{(c^{rk})}({b'}^{\triangle})}
{\eta_{(c^{rk})}((bb')^{\triangle})}\rho^{\diamondsuit}_{rkc}(b^{\triangle},{b'}^{\triangle}).
\end{align}
Then \eqref{eq:RS rho rkc rk on Levi} follows from \eqref{Appendixeq:rho c rk and rho square globally} and \eqref{eq:RS rho top rk on Levi}, and the assertion on the splitting is clear (when $m|rk$, $(\rho_c^{\diamondsuit})^{rk}$ is trivial, otherwise we use the assumption $\mu_{2m}<F^*$).
Moreover when $m|rk$, by \eqref{eq:eta beta nu} for $\beta=(c^{rk})$ and because $(\eta^{\diamondsuit}_{c,\nu})^{rk}$ is trivial, $\eta_{(c^{rk}),\nu}=(\eta^{\diamondsuit}_{rkc,\nu})^{-1}$ on $\GL_c^{\triangle}(F_{\nu})$, and since $\eta_{(c^{rk})}$ is defined on $\GL_c^{\triangle}(\A)$, so is $\eta^{\diamondsuit}_{rkc}$ whence $\eta_{(c^{rk})}(b^{\triangle})=(\eta^{\diamondsuit}_{rkc})^{-1}(b^{\triangle})$ for $b\in\GL_c(\A)$.
\end{proof}

\subsection{Unramified $L$-functions}\label{unr L functions}
Assume $F$ is unramified. We briefly describe the definition of unramified $L$-functions, complete details appeared in \cite[\S~1.8]{me12}.
For general results see \cite{Savin5,McNamara,Weissman2,GG,Weissman2018a}. A genuine irreducible unramified representation $\pi$ of $\Sp_{2d}^{(m)}$
(resp., $\GL_{d}^{(m)}$)
can be identified with a constituent of a genuine unramified principal series representation, which by the Stone--von Neumann Theory
(\cite[\S~0.3]{KP}, \cite[\S~13.5]{McNamara}) can be constructed non-canonically by choosing an unramified character of $C_{\widetilde{T}_{\Sp_{2d}}}$ (resp., $C_{\widetilde{T}_{\GL_{d}}}$). Such characters have the form $\otimes_{i=1}^d\varepsilon\otimes\vartheta\mu_i$, where $\vartheta=1$ unless $m\equiv2\,(4)$ in which case $\vartheta=\gamma_{\psi'}$, and $\mu_i$ is an unramified quasi-character of $F^*$. The $L$-group of $\Sp_{2d}^{(m)}$ is $\SO_{2d+1}(\C)$ when $m$ is odd and $\Sp_{2d}(\C)$ otherwise, and for $\GL_d^{(m)}$ it is
$\GL_d(\C)$ (\cite[\S~2.3]{Gao4}, \cite[\S~5.1]{WWLi2017}). We put
\begin{align*}
t_{\pi,\vartheta}=
\begin{cases}
\diag(\mu_1(\varpi^r),\ldots,\mu_d(\varpi^r),1,\mu_d^{-1}(\varpi^r),\ldots,\mu_1^{-1}(\varpi^r))&\Sp_{2d}^{(m)},r=m,\\
\diag(\mu_1(\varpi^r),\ldots,\mu_d(\varpi^r),\mu_d^{-1}(\varpi^r),\ldots,\mu_1^{-1}(\varpi^r))&\Sp_{2d}^{(m)},r=m/2,\\
\diag(\mu_1(\varpi^r),\ldots,\mu_d(\varpi^r))&\GL_{d}^{(m)}.
\end{cases}
\end{align*}
Now for any finite-dimensional complex representation $\sigma$ of the corresponding dual group,
\begin{align*}
L_{\vartheta}(s,\pi,\sigma)=\det(1-\sigma(t_{\pi,\vartheta})q^{-s})^{-1}.
\end{align*}
We can now define $L_{\nu}(s,\pi\times\tau)$ for pairs of irreducible unramified representations $\pi\times\tau$ using
$t_{\pi,\vartheta_{\pi}}\otimes t_{\tau,\vartheta_{\tau}}$, similarly to the linear case.
For $m\equiv2\,(4)$, since $\gamma_{\psi'}\gamma_{\psi''}^{-1}$ is a quadratic character of $F^*$ (whether $F$ is unramified or not),
$t_{\pi,\gamma_{\psi'}}=t_{\gamma_{\psi'}\gamma_{\psi''}^{-1}\pi,\gamma_{\psi''}}$ (replacing the parametrization amounts to a quadratic twist of the linear data).

A priori, the $L$-function is independent of $\vartheta$ unless $m\equiv2\,(4)$, in which case it does depend on $\vartheta$.
However, the class of functions we will study here when $m\equiv2\,(4)$ will always involve 2 representations, and since we will always
use the same parameter $\vartheta$ for both, $t_{\pi,\gamma_{\psi'}}\otimes t_{\tau,\gamma_{\psi'}}=t_{\pi,\gamma_{\psi''}}\otimes t_{\tau,\gamma_{\psi''}}$, which means $L_{\vartheta}(s,\pi\times\tau)$ is independent of $\vartheta$. This remark also applies
to $L_{\vartheta}(s,\tau,\wedge^2)$ and $L_{\vartheta}(s,\tau,\vee^2)$, i.e., they are independent of $\vartheta$.
In addition $\gamma_{\psi'}^{-1}=\gamma_{\psi'}$ (because $\mu_{2m}<F^*$), and if $\tau$ is a representation of $\GL_d^{(m)}$,
$t_{\tau^*,\gamma_{\psi'}}=t_{\tau^{\vee},\gamma_{\psi'}}=t_{\tau^{\vee},\gamma_{\psi'}^{-1}}$ (see \cite[Proposition~25]{me12}). Note that $\tau^*$ is genuine and $\tau^{\vee}$ is anti-genuine.

We need the following result on the Satake parameters of tempered unramified representations (formulated for representations with Iwahori fixed vectors). Let $G$ be either $\Sp_{2d}$ or $\GL_d$, and $\widetilde{G}=\Sp_{2d}^{(m)}$ or $\GL_{d}^{(m)}$. Let $I$ be the Iwahori subgroup of $G$, determined by
our choice of $B_G$. Identify $I$ with its image in $\widetilde{G}$, determined by the fixed splitting of $K_G$.
\begin{lemma}\label{lemma:unr unitary Satake}
Let $\pi$ be a genuine tempered representation of $\widetilde{G}$, which has a nonzero vector fixed by $I$. Assume
$\pi$ is associated with
$\otimes_{i=1}^d\varepsilon\otimes\vartheta\mu_i$ as above. Then $|\mu_i|=1$ for all $i$.
\end{lemma}
\begin{proof}
By the results of Savin \cite{Savin6} (see \cite[Corollary~5]{McNamara}),
the Iwahori--Hecke algebra of $\widetilde{G}$, i.e., the algebra of anti-genuine $I$-bi-invariant locally constant and compactly supported
functions on $\widetilde{G}$ (see \cite[\S~13.12]{McNamara}), is isomorphic to the Iwahori--Hecke algebra of $\SO_{2d+1}$ when $G=\Sp_{2d}$
and $m$ is odd, $\Sp_{2d}$ when $m$ is even, and $\GL_d$ if $G=\GL_d$ (in the latter case this was explicitly proved in \cite{Savin7}).
According to the arguments in \cite[\S~16]{FK} (see \cite[Corollary~17.2]{FK} in particular), this isomorphism implies an isomorphism
between certain representations of $\widetilde{G}$ and $G$, taking $\pi$ into a tempered representation
of $G$ with inducing data $\otimes_i\mu_i^r$, and this representation admits a nonzero vector fixed by the corresponding Iwahori subgroup. It is known (see e.g., \cite[\S~10]{LR}) that
for such representations the inducing data is unitary. (\cite[\S~16]{FK} was independent of the trace formula, as explicitly noted there.)
\end{proof}

In a global context we choose the parameter $\vartheta$ globally: if $m\equiv2\,(4)$, we take
a character $\psi'$ of $F\backslash\A$ and set $\vartheta=\prod_{\nu}\vartheta_{\nu}=\prod_{\nu}\gamma_{\psi'_{\nu}}$, otherwise $\vartheta=1$. Note that $\vartheta$ is defined in all places (not only the unramified ones). If $\pi$ and $\tau$ are a pair of genuine cuspidal representations of $\Sp_{2d}^{(m)}(\A)$ and $\GL_{d}^{(m)}(\A)$ and $S$ is a finite set outside which all data are unramified,
$L_{\vartheta}^S(s,\pi\times\tau)=\prod_{\nu\notin S}L_{\vartheta}(s,\pi_{\nu}\times\tau_{\nu})$. Again since we use the same parameter $\vartheta$, the partial $L$-function is independent of $\vartheta$ and we already write $L^S(s,\pi\times\tau)$.

\subsection{The covering $\GL_d^{(2)}$}\label{covering GL r=1}
When $r=1$, $\sigma^2_{d}\equiv1$ hence by Proposition~\ref{proposition:sigma * and sigma on GLd} we can
assume the $2$-cocycle of $\GL_d^{(2)}$ is given by $(\det b,\det b')_2$ (locally or globally). In this case we are assuming $\mu_4<F^*$,
hence $\GL_d^{(2)}$ is split over $\GL_d$. But even without this assumption, the representation theories of $\GL_d^{(2)}$ and $\GL_d$ are directly related (see e.g., \cite[\S~2.4, \S~7.2]{Gan}). There is a bijection between representations of $\GL_d$ and $\GL_d^{(2)}$, given by
$\tau\mapsto\gamma_{\psi'}\otimes\tau$ where $\gamma_{\psi'}\otimes\tau(\langle b,\epsilon\rangle)=\epsilon\gamma_{\psi'}(\det b)\tau(b)$.
Under this bijection $\Ind_{P_{\beta}}^{\GL_k}(\otimes_{i=1}^d\tau_i)$ is mapped to $\gamma_{\psi'}\Ind_{P_{\beta}}^{\GL_k}(\otimes_{i=1}^d\gamma_{\psi'}\tau_i)$. Among the similarities, the theory of Speh representations (\cite{Jac4}) is applicable; and the Whittaker model of $\tau$ is mapped to the Whittaker model of $\gamma_{\psi'}\otimes\tau$ by $W\mapsto\gamma_{\psi'}W$. The Rankin--Selberg integrals of \cite{JS2,JS1,JPSS} for
pairs of representations $\tau_1\times\tau_2$ coincide with the similar integrals for $\gamma_{\psi'}\otimes\tau_1$ and $\gamma_{\psi'}\otimes\tau_2$.
In turn we can define local and complete $L$-functions for $\gamma_{\psi'}\otimes\tau_1$ and
$\gamma_{\psi'}\otimes\tau_2$ using the definitions for $\tau_1\times\tau_2$. Then $L(s,(\gamma_{\psi'}\otimes\tau)\times(\gamma_{\psi'}\otimes\tau)^*)=L(s,\tau\times\tau^{\vee})$.

\subsection{Dual representations}\label{dual reps}
For any smooth representation $\pi$ of an $l$-group, the contragredient representation
$\pi^{\vee}$ is by definition the smooth part of the algebraic dual of $\pi$ (e.g., \cite[2.13]{BZ1}). This definition applies in particular to the covering groups here, but of course when $\pi$ is genuine, $\pi^{\vee}$ is anti-genuine, and while it is certainly possible to define $L$-functions
formally for (unramified) representations ignoring this fact, it makes more sense to consider only genuine representations. Indeed as we shall see below (e.g., \eqref{eq:RS functional equation}), the local functional equations relate between $\pi$ and a genuine representation, in particular not $\pi^{\vee}$ (save perhaps double coverings where genuine is the same as anti-genuine). Thus we require a notion different from $\pi^{\vee}$.

It is reasonable to demand the following properties from a plausible replacement $\widetilde{\pi}$ of $\pi^{\vee}$.
\begin{enumerate}[leftmargin=*]
  \item If $\pi$ is genuine, so is $\widetilde{\pi}$.
  \item For trivial or $2$-fold coverings and irreducible admissible representations $\pi$, $\widetilde{\pi}\cong\pi^{\vee}$.
  \item\label{it:props preserving irred admissible L2} Being admissible, irreducible, supercuspidal, square-integrable, tempered or unramified for $\pi$ and $\widetilde{\pi}$ should be equivalent.
  \item\label{it:props preserving center} Assume $\pi$ admits a central character, then so does $\widetilde{\pi}$. Whenever the center of the underlying linear group admits a subgroup $C$ whose covering $\widetilde{C}$ is trivial and belongs to the center of the covering group, we can identify the restriction of the central character of $\pi$ to $C$ with a linear representation (non-canonically), say $\mu_{\pi}$. Then essentially $\mu_{\widetilde{\pi}}=\mu_{\pi}^{-1}$.
  \item\label{it:props preserving Satake} When data are unramified, $t_{\widetilde{\pi},\vartheta}=t_{\pi^{\vee},\vartheta}$ (see \S~\ref{unr L functions}).
\end{enumerate}

Define for $\GL_d^{(m)}$, $\widetilde{\pi}=\pi^*$ and for $\Sp_{2d}^{(m)}$, $\widetilde{\pi}=\pi$. Indeed when
$\pi$ is irreducible and admissible, for $\GL_d$ we have $\pi^*\cong\pi^{\vee}$, for $\GL_d^{(2)}$ we can write
$\pi(\langle b,\epsilon\rangle)=\varepsilon(\epsilon)\gamma_{\psi'}(b)\pi_0(b)$ for a representation $\pi_0$ of $\GL_d$
then again $\pi^*\cong\pi^{\vee}$ (recall $\gamma_{\psi'}^{-1}=\gamma_{\psi'}$), and for $\Sp_{2d}^{(m)}$ with $m\leq2$, $\pi\cong\pi^{\vee}$
(\cite[Chapitre~4]{MVW} and \cite{Binyong2011}). The rest is trivial for $\Sp_{2d}^{(m)}$, and observe that
$\widetilde{\pi}=\pi$ is reasonable for this group mainly because in the unramified case
$t_{\pi,\vartheta}=t_{\pi^{\vee},\vartheta}$ and we require \eqref{it:props preserving Satake}.
For $\GL_d^{(m)}$, $\pi$ and $\pi^*$ are simultaneously square-integrable or tempered, this follows from the characterization
of these properties using cuspidal exponents, then \eqref{it:props preserving irred admissible L2} is clear.
Regarding \eqref{it:props preserving center} for $C=C_{r,d}$, if
$\pi(\langle x^rI_d,\epsilon\rangle)=\varepsilon(\epsilon)\vartheta(x^r)\eta(x^r)$ for a quasi-character $\eta$ of $F^*$,
$\pi^*(\langle x^rI_d,\epsilon\rangle)=\pi(\langle x^{-r}I_d,\epsilon\rangle)=
\varepsilon(\epsilon)\vartheta(x^r)^{-1}\eta(x^r)^{-1}$ ($\vartheta=1$ unless $m\equiv2\,(4)$, then $\vartheta=\gamma_{\psi'}$). For \eqref{it:props preserving Satake} see \S~\ref{unr L functions}.

The relation between the contragredient representation and its $L$-parameter was studied by Adams and Vogan \cite{AdamsVogan2016} in the context of reductive groups. Following their work Weissman \cite{Weissman2018a} indicated the possibility of replacing $\pi^{\vee}$ by a genuine representation.
Another reason for the present definition is evident when considering intertwining operators. The image of the intertwining
operator acting on a representation of $\Sp_{2d}^{(m)}$ parabolically induced from the Siegel parabolic subgroup and $|\det|^s\tau$, is a
representation parabolically induced from $|\det|^{-s}\tau^*$ (by Lemma~\ref{lemma:conjugation commutes}, see also \cite[(4.16), p.~85]{me12}). Oftentimes functional equations follow from the theory of intertwining operators, and these take genuine representations into genuine ones.

\section{Representations of type $(rk,c)$}\label{Representations of type (k,c)}
\subsection{Definition and basic properties}\label{Definition and generalities}
Let $F$ be a local field.
Let $m$, $k$ and $c$ be positive integers, then we have the covering group $\GL_{rkc}^{(m)}$ defined in \S~\ref{Local coverings}.
Recall the character $\psi=\psi_{rk}$ of $V_{(c^{rk})}$ given by \eqref{def:psi_l}. In general, the unipotent orbits of $\GL_{l}$ are in bijection with the partitions of $l$. For a partition $\beta$ of $l$, let $V(\beta)<N_{\GL_{l}}$ be the corresponding unipotent subgroup and $\widehat{V}(\beta)_{\mathrm{gen}}$ be the set of its generic characters. Write $\beta\succsim \beta'$ if $\beta$ is greater than or not comparable with the partition $\beta'$, according to the partial ordering defined on the partitions. Refer to \cite[\S~2]{G2}
or to \cite[\S~1.4]{DimaKaplan} and the references therein, for these definitions.

In particular for $\beta=((rk)^c)$, $V(\beta)=V_{(c^{rk})}$ and $\psi\in \widehat{V}(\beta)_{\mathrm{gen}}$. The stabilizer of
$\psi$ in $M_{(c^{rk})}$ is then the diagonal embedding $\GL_c^{\triangle}$ of $\GL_c$ in $M_{(c^{rk})}$.

Let $\rho$ be a genuine admissible finite-length representation of $\GL_{rkc}^{(m)}$. We say that $\rho$ is an $(rk,c)$ representation if it satisfies the folllowing properties:
\begin{enumerate}[leftmargin=*]
\item $\Hom_{V(\beta')}(\rho,\psi')=0$ for any $\beta'\succsim((rk)^c)$ and $\psi'\in \widehat{V}(\beta')_{\mathrm{gen}}$.
\item $\dim\Hom_{V_{(c^{rk})}}(\rho,\psi)=1$.
\end{enumerate}
Assume this is the case.
An $(rk,c)$ functional on $\rho$ is an element $0\ne\lambda\in \Hom_{V_{(c^{rk})}}(\rho,\psi)$. The $(rk,c)$ model (which is by definition unique) is denoted $W_{\psi}(\rho)$, it is the space spanned by the mappings $g\mapsto\lambda(\rho(g)\xi)$ where $g\in\GL_{rkc}^{(m)}$ and $\xi$ varies in the space of $\rho$.

Let $F$ be a local non-archimedean field.
An alternative definition can be given in terms of the theory of derivatives of Bernstein and Zelevinsky \cite{BZ1,BZ2}. Recall the
functors $\Phi^{\mp}$ and $\Psi^{\mp}$ defined in \cite[5.11]{BZ1}. These can be defined for covering groups as well
(see \cite{Kable} for coverings of \cite{KP}). Let $Y_{l}<\GL_{l}$ denote the subgroup of matrices with the last row
$(0,\ldots,0,1)$.
For an integer $j$, let $\rho^{(j)}=\Psi^-(\Phi^-)^{j-1}(\rho|_{\widetilde{Y}_{rkc}})$ denote the $j$-th derivative of $\rho$ in the sense of \cite{BZ1,BZ2} (we use the non-normalized version).
For a partition $\beta'$, let $[\beta']$ denote the composition defined
by taking the integers of $\beta'$ in decreasing order. If $[\beta']=(\beta_1,\ldots,\beta_l)$, define inductively $\rho^{[\beta']}=(\rho^{(\beta_1)})^{(\beta_2,\ldots,\beta_{l})}$. With this notation
$\rho$ is $(rk,c)$ if for any partition $\beta'\succsim((rk)^c)$, $\rho^{[\beta']}=0$, and
$\dim\rho^{((rk)^c)}=1$. This follows using the local ``exchange of roots" technique of Ginzburg \textit{et al.} \cite{GRS5},
which is also applicable to covering groups.
Over archimedean fields according to \cite{AGS2015a,AGS2015b,GGS} the same characterization is valid, except that the derivative here is the pre-derivative of \cite{GGS}. See \cite[Theorems~E, F]{GGS} which explicated the relation between degenerate Whittaker models and derivatives over any local field, and also \cite{MW3,Cai2}.
\begin{proposition}\label{prop:rodier non}
Assume $F$ is non-archimedean. For $1\leq i\leq d$, let $\rho_i$ be an $(rk_i,c)$ representation. Then
$\rho=\Ind_{\widetilde{P}_{\beta rc}}^{\GL_{rkc}^{(m)}}(\otimes_{i=1}^d\rho_i)$ is $(rk,c)$, where $\beta=(k_1,\ldots,k_d)$ and $k=\sum_{i=1}^dk_i$.
\end{proposition}
\begin{proof}
Since direct factors of Levi subgroups of $\GL_{rkc}$ do commute in $\GL_{rkc}^{(m)}$, the ``Leibniz rule" for the derivative \cite[4.5]{BZ2} applies, and the
result follows from the definition of $(rk,c)$ representations using repeated derivatives. See also
\cite[Proposition~3]{CFK}.
\end{proof}

With the aid of this result we can realize $(rk,c)$ representations by virtue of a Jacquet type integral. With the notation of Proposition~\ref{prop:rodier non}, we have a representation $\rho=\Ind_{\widetilde{P}_{\beta rc}}^{\GL_{rkc}^{(m)}}(\otimes_{i=1}^d\rho_i)$ which is of type $(rk,c)$, and further assume we have an $(rk,c)$ representation $\rho_0$ which is a quotient of $\rho$. Since Jacquet functors are exact, this means that $\rho_0$ is the unique $(rk,c)$ constituent of $\rho$. Denote $\beta'=(\beta_d,\ldots,\beta_1)$. The following Jacquet integral formally defines an $(rk,c)$ functional on the space of $\rho$:
\begin{align}\label{eq:(k,c)  functional using an integral}
\int\limits_{V_{\beta'rc}}\xi(\langle w_{\beta rc}v,1\rangle)\psi^{-1}(v)\,dv.
\end{align}
To ensure convergence one must twist the inducing data using auxiliary complex parameters. Over non-archimedean fields,
the uniqueness of the $(rk,c)$ functional on the twisted space of $\rho$ implies \eqref{eq:(k,c)  functional using an integral} admits analytic continuation, which becomes a (nonzero) $(rk,c)$ functional on $\rho$ itself and factors through $\rho_0$ by our assumptions on the latter. This provides a realization of $W_{\psi}(\rho_0)$. For more details and the archimedean case see \cite[Proposition~3]{CFK} and \cite[\S~2.3]{CFK}.

We describe a second realization of $(rk,c)$ models, based on decompositions of $c$. Fix $0<l<c$. Assume we have an $(rk,l)$ (resp., $(rk,c-l)$)
representation $\rho_l$ (resp., $\rho_{c-l}$) and an $(rk,c)$ representation $\rho_0$. Further assume
\begin{align}\label{rep:composition lemma output space before W}
\rho_0\subset \Ind_{\widetilde{P}_{(rkl,rk(c-l))}}^{\GL_{rkc}^{(m)}}((
W_{\psi}(\rho_{l})\otimes W_{\psi}(\rho_{c-l}))\delta_{P_{(rkl,rk(c-l))}}^{-1/(2rk)}).
\end{align}
Set
\begin{align*}
&\kappa=\kappa_{l,c-l}=\left(\begin{array}{cccccccc}I_l\\0&0&I_l\\0&0&0&0&I_l&\ddots\\&&&&&&I_l&0\\0&I_{c-l}\\0&0&0&I_{c-l}&&\ddots\\&&&&&&&I_{c-l}\end{array}\right)\in\GL_{rkc}.
\end{align*}
(Here $I_l,I_{c-l}$ appear $rk$ times.) For $v=(v_{i,j})_{1\leq i,j\leq rk}\in V_{(c^{rk})}$, write each $v_{i,j}\in\Mat_c$ in the form
\begin{align*}
\left(\begin{smallmatrix}v_{i,j}^1&v_{i,j}^2\\v_{i,j}^3&v_{i,j}^4\end{smallmatrix}\right),\qquad
v_{i,j}^1\in\Mat_{l},\quad v_{i,j}^4\in\Mat_{c-l}.
\end{align*}
Let $V_{(c^{rk})}^t<V_{(c^{rk})}$ be the subgroup obtained by deleting the blocks $v_{i,j}^{t'}$ for all $i<j$ and $t'\ne t$,
where $1\leq t\leq 4$. Put $V=V^3$. Then we have the following integral
\begin{align}\label{eq:mnk functional using w_{n,m,k}}
\int\limits_{V}\xi(\langle \kappa,1\rangle\langle v,1\rangle)\,dv,
\end{align}
where $\xi$ belongs to the space of \eqref{rep:composition lemma output space before W}.

As explained in \cite[\S~2.6]{me12} and
\cite[Proposition~62]{me12} (see also \cite[\S~2.4]{CFK}), \eqref{eq:mnk functional using w_{n,m,k}} is absolutely convergent and nonzero on any summand
of \eqref{rep:composition lemma output space before W}, and is an $(rk,c)$ functional.
In particular we again obtain a realization of $W_{\psi}(\rho_0)$.

Let $\rho$ be an $(rk,c)$ representation.
By definition, the group $\widetilde{\GL}_c^{\triangle}$ acts by a genuine character on $J_{V_{(c^{rk})},\psi}(\rho)$.
By Proposition~\ref{proposition:sigma * and sigma on GLd} there is $\varsigma_{\triangle,c}\in\mathrm{C}^1(\GL_c,\mu_m)$ such that
\begin{align}\label{eq:sigma rkc diamondsuit on G triangle}
\sigma_{rkc}^{\diamondsuit}(b^{\triangle},{b'}^{\triangle})=
(\sigma_{c}^{\diamondsuit}(b,b'))^{rk}=
\left(\frac{\varsigma_{\triangle,c}(b)\varsigma_{\triangle,c}(b')}{\varsigma_{\triangle,c}(bb')}\right)^{rk}(\det b,\det b')_{m/r}^{k}.
\end{align}
(If $m|rk$, $(\det b,\det b')^k_{m/r}\equiv1$, otherwise $(\det b,\det b')^k_{m/r}=(\det b,\det b')_2$.)
Hence a genuine character $\eta_{\triangle}$ of $\widetilde{\GL}_c^{\triangle}=\GL_c[\sigma_{rkc}^{\diamondsuit}]$
is of the form
$\eta_{\triangle}(\langle b^{\triangle},\varsigma_{\triangle,c}(b)^{-rk}\epsilon\rangle)=\varepsilon(\epsilon)[\gamma_{\psi'}(\det b)]\eta'(\det b)$,
where $\psi'$ is a nontrivial additive character of $F$, $\gamma_{\psi'}$ is omitted if $m|rk$, and $\eta'$ is a quasi-character of $F^*$. Changing $\psi'$ amounts to multiplying $\eta'$ by a square-trivial character of $F^*$. By \eqref{eq:sigma rkc diamondsuit on G triangle} there is a genuine character $\eta_{\triangle}$ such that
\begin{align}\label{eq:def of eta on all of GLc triangle}
\lambda(\rho(b)\xi)=\eta_{\triangle}(b)\lambda(\xi),\qquad\forall b\in \widetilde{\GL}_c^{\triangle}.
\end{align}
Note that $\eta_{\triangle}$ is independent of $\lambda$.
By \cite[Proposition~36]{me12}, $b^{\triangle}\mapsto\langle b^{\triangle},\varsigma_{*,c}^{-rk}(b)\rangle$ is the unique splitting of $\SL_c^{\triangle}$ in $\GL_{rkc}^{(m)}$, therefore $\varsigma_{\triangle,c}^{rk}(b)=\varsigma_{*,c}^{rk}(b)$ for all
$b\in\SL_c$. If $\rho$ is unramified and there is $\lambda$ which is nonzero on an unramified vector (this might be difficult to show),
so is $\eta_{\triangle}$.

For $a,a'\in F^*$, $\sigma_{rkc}^{\diamondsuit}(aI_{rkc},a'I_{rkc})=(a,a')_{m/r}^{kc}$. Thus we can also fix a genuine character $\eta$ of the preimage of $(C_{\GL_c})^{\triangle}\cong F^*$ in $\GL_{rkc}^{(m)}$, which depends on $\gamma_{\psi'}$ if $m\nmid rkc$, such that
\begin{align}\label{eq:def of eta}
\lambda(\rho(\langle aI_{rkc},1\rangle)\xi)=\eta(\langle a,1\rangle)\lambda(\xi),\qquad\forall a\in F^*.
\end{align}
Of course \eqref{eq:def of eta} and \eqref{eq:def of eta on all of GLc triangle} are by definition compatible, i.e.,
$\eta(\langle a,1\rangle)=\eta_{\triangle}(\langle aI_c,1\rangle)$.
If $\rho$ is unramified and $\lambda$ can be chosen as above, $\eta$ is unramified, i.e., trivial on $\{\langle y,1\rangle:y\in\mathcal{O^*}\}$. This is because $\eta^{\diamondsuit}_{rkc}$ is trivial on $T_{\GL_{rkc}}\cap K_{\GL_{rkc}}$. Also by the definition if we replace $\rho$ with $\rho^*$ in \eqref{eq:def of eta} we obtain $\eta^{-1}$ (use ${}^*\langle aI_{rkc},1\rangle=\langle a^{-1}I_{rkc},1\rangle$). Moreover since $a^rI_{rkc}\in C_{r,rkc}$, if $\rho$ admits a central character, $\eta(\langle a,1\rangle)^r=\eta(\langle a,1\rangle^r)=\eta(\langle a^r,1\rangle)=\rho(\langle a^rI_{rkc},1\rangle)$.

\subsection{The representation $\rho_c(\tau)$}\label{the representation rho_c(tau)}

Let $F$ be a local field. We describe a map $\rho_c$ from genuine irreducible generic representations of
$\GL_k^{(m)}$ to $(rk,c)$ representations. We recall that a generic representation of $\GL_{rkc}^{(m)}$ is a representation which affords a Whittaker model, which
is in general not unique even if the representation is irreducible.
All conjectures and results in this section are known for $m=1$ (see \cite[\S~2.2]{CFK} and the references therein) and proved here
if $r=1$ or $k=1$. We stress that all conjectures in this section will henceforth be assumed to hold.

Let $\tau$ be a genuine essentially tempered generic representation of $\GL_k^{(m)}$ ($\tau$ is in particular irreducible, see \S~\ref{reps}). According to the Langlands Quotient Theorem, proved
for covering groups over non-archimedean fields by Ban and Jantzen \cite{BJ,BanJantzen2016}, the representation
\begin{align}\label{eq:rho c tempered tau}
\Ind_{\widetilde{P}_{(k^{rc})}}^{\GL_{rkc}^{(m)}}((\tau\otimes \ldots \otimes \tau)\delta_{P_{(k^{rc})}}^{1/(2rk)})
\end{align}
has a unique irreducible quotient which we denote by $\rho_c(\tau)$. Equivalently, $\rho_c(\tau)$ is the unique irreducible summand of
\begin{align}\label{eq:rho c tempered tau sub}
\Ind_{\widetilde{P}_{(k^{rc})}}^{\GL_{rkc}^{(m)}}((\tau\otimes \ldots \otimes \tau)\delta_{P_{(k^{rc})}}^{-1/(2rk)}).
\end{align}
In particular $\rho_1(\tau)$ is the unique irreducible quotient of
\begin{align*}
\Ind_{\widetilde{P}_{(k^{r})}}^{\GL_{rk}^{(m)}}((\tau\otimes \ldots \otimes \tau)\delta_{P_{(k^{r})}}^{1/(2rk)}),
\end{align*}
and by transitivity of induction $\rho_c(\tau)$ is the unique irreducible quotient of
\begin{align}\label{eq:rho c tempered tau quotient using rho 1 and transitivity}
\Ind_{\widetilde{P}_{((rk)^{c})}}^{\GL_{rkc}^{(m)}}((\rho_1(\tau)\otimes \ldots \otimes \rho_1(\tau))\delta_{P_{((rk)^{c})}}^{1/(2rk)}).
\end{align}

\begin{corollary}\label{corollary:tempered}
Assume $\tau$ is genuine essentially tempered generic and for $0<l<c$, $\rho_{l}(\tau)$ and $\rho_{c-l}(\tau)$ are
$(rk,l)$ and $(rk,c-l)$, respectively. Then $\rho_c(\tau)$ is embedded in the r.h.s.~ (right hand side) of \eqref{rep:composition lemma output space before W} and
$W_{\psi}(\rho_c(\tau))$ can be realized using \eqref{eq:mnk functional using w_{n,m,k}}.
\end{corollary}
\begin{proof}
By definition $\rho_{l}(\tau)$ is embedded in
the corresponding space \eqref{eq:rho c tempered tau sub}, and being irreducible, it is isomorphic to
$W_{\psi}(\rho_{l}(\tau))$. Similarly for $\rho_{c-l}(\tau)$.
Now use transitivity of induction.
\end{proof}
\begin{conjecture}\label{conjecture:derivatives of rho c tau}
The representation $\rho_1(\tau)$ is $(rk,1)$. Moreover, if $F$ is non-archimedean, for all $c>1$ the highest derivative of $\rho_c(\tau)$ is $\rho_c^{(rk)}(\tau)$ and equals $|\det|^{(r-1)/2}\rho_{c-1}(\tau)$.
\end{conjecture}
\begin{proposition}
Conjecture~\ref{conjecture:derivatives of rho c tau} holds for $r=1$ or $k=1$.
\end{proposition}
\begin{proof}
The result for $r=1$ is clear from the linear case, see \S~\ref{covering GL r=1} and for the assertions on the derivative see
\cite[\S~6.1]{Tadic1986}.

When $F=\C$ and $k=1$, the result follows from the fact that
$\Ind_{B_{\GL_r}}^{\GL_{r}}(\delta_{B_{\GL_r}}^{1/(2r)})$ is already irreducible and generic (see e.g., \cite{Vog86}).

Assume $F$ is non-archimedean, $r>1$ and $k=1$.
In this case $\tau$ is an irreducible representation of $\GL_1^{(m)}$, constructed using a quasi-character $\chi$
of $F^*$ restricted to $F^{*r}$ (e.g., \cite[\S~13.5]{McNamara}). The definition implies $\rho_c(\tau)=\chi\Theta_{rc,m,r,\vartheta}$,
where $\Theta_{rc,m,r,\vartheta}$ is the exceptional representation of $\GL_{rc}^{(m)}$ of Gao \cite{Gao5}, who extended the construction of \cite{KP}. It is the unique irreducible quotient of a genuine principal series
induced from $\delta_{B_{\GL_{rc}}}^{1/(2r)}$ with a parameter $\vartheta$ (see \S~\ref{unr L functions} and \cite[\S~1.11]{me12}), which we
briefly denote $\mathrm{I}(\delta_{B_{\GL_{rc}}}^{1/(2r)})$. It is also the image of a standard intertwining operator
$M(J_{rc}):\mathrm{I}(\delta_{B_{\GL_{rc}}}^{1/(2r)})\rightarrow \mathrm{I}(\delta_{B_{\GL_{rc}}}^{-1/(2r)})$, given by an absolutely convergent integral because $\delta_{B_{\GL_{rc}}}^{1/(2r)}$ belongs to the positive Weyl chamber.
Also for any $0\leq l\leq rc$,
\begin{align}\label{eq:Jacquet of exceptional}
J_{V_{(rc-l,l)}}(\Theta_{rc,m,r,\vartheta})=\delta_{V_{(rc-l,l)}}^{(r-1)/(2r)}(\Theta_{rc-l,m,r,\vartheta}\otimes\Theta_{l,m,r,\vartheta}).
\end{align}
This follows as in \cite[Theorem~5.1 (1)]{Kable} for double coverings of \cite{KP} (the proof is simpler with the usual tensor). Recall the Jacquet functor here is not normalized. Note that the $l$-th derivative
$\Theta_{rc,m,r,\vartheta}^{(l)}$ factors through $J_{V_{(rc-l,l)}}(\Theta_{rc,m,r,\vartheta})$.

Assume data are unramified. Then $\Theta_{l,m,r,\vartheta}$ affords a unique Whittaker model for $l=r$ and no Whittaker model
when $l>r$, by \cite[Proposition~3.5]{Gao4}. Hence $\dim\Theta_{r,m,r,\vartheta}^{(r)}=1$ which means
$\rho_1(\tau)$ is $(r,1)$, and $\Theta_{l,m,r,\vartheta}^{(l)}=0$ for $l>r$ so that \eqref{eq:Jacquet of exceptional} implies $\Theta_{rc,m,r,\vartheta}^{(l)}=0$. In addition \eqref{eq:Jacquet of exceptional} with $l=r$ shows $\rho_c^{(r)}(\tau)=|\det|^{(r-1)/2}\rho_{c-1}(\tau)$. This completes the proof for this case.

Assume data are ramified. Then $\rho_1(\tau)$ is $(r,1)$ according to the proof of \cite[Proposition~5.7]{Gao2018b} (which used the globalization argument of \cite[Theorem~II.2.5]{KP}, note that $\widetilde{C}_{\GL_r}<\GL_r^{(m)}$ is abelian). Using \eqref{eq:Jacquet of exceptional} it remains to show $\Theta_{l,m,r,\vartheta}^{(l)}=0$ for $l>r$. Arguing as in \cite[Theorem~2.7]{me11} (i.e., applying \cite[Theorem~5.2]{BZ2}) or by \cite{Banks1998},
the case $l>r$ follows from $l=r+1$.

We can assume $\Theta_{r+1,m,r,\vartheta}$ is a quotient of $\mathrm{I}=\Ind_{\widetilde{P}_{(r,1)}}^{\GL_{r+1}^{(m)}}((\Theta_{r,m,r,\vartheta}\otimes\Theta_{1,m,r,\vartheta})\delta_{P_{(r,1)}}^{1/(2r)})$ (dualize \cite[Theorem~5.1 (6)]{Kable}). Let $\widetilde{A}$ be a maximal abelian subgroup of $\GL_1^{(m)}$ with $A<F^*$, and
$\mathcal{G}$ be a set of representatives for $A\backslash F^*$. Set $d=|\mathcal{G}|$. For a representation $\pi$ of $\GL_{r+1}^{(m)}$, let
$W(\pi,\psi)$ denote the space of $\psi$-Whittaker functionals on $\pi$ where $\psi$ is the $(r+1,1)$ character (i.e., the standard generic character). Since $\Theta_{1,m,r,\vartheta}$ is $d$-dimensional (\cite[\S~0.3]{KP}, \cite[\S~13.5]{McNamara}) and
$\Theta_{r,m,r,\vartheta}$ is $(r,1)$, $\dim W(\mathrm{I},\psi)=d$ (use \cite[Theorem~5.2]{BZ2} or argue as in \cite{Banks1998}).

Because $\Theta_{r+1,m,r,\vartheta}$ is a quotient of $\mathrm{I}$, $W(\Theta_{r+1,m,r,\vartheta},\psi)\subset W(\mathrm{I},\psi)$. In addition $W(\Theta_{r+1,m,r,\vartheta},\psi)$ is a (finite-dimensional) representation of $\widetilde{C}_{\GL_{r+1}}$. By \eqref{eq:Nice GL $2$-cocycle on torus}, for $a,a'\in F^*$,
$\sigma^{\diamondsuit}_{r+1}(aI_{r+1},a'I_{r+1})=(a,a')_{m/r}(a,a')_m^{-1}$.
Thus either $\GL_1^{(m)}\cong \widetilde{C}_{\GL_{r+1}}$ when $m$ is odd, or there is a bijection between the genuine representations of
$\GL_1^{(m)}$ and $\widetilde{C}_{\GL_{r+1}}$, given by $\pi\mapsto\pi_{\psi'}$ where $\pi_{\psi'}(\langle aI_{r+1},\epsilon\rangle)=
\gamma_{\psi'}(a)\pi(\langle a,\epsilon\rangle)$. In both cases the genuine irreducible representations of
$\widetilde{C}_{\GL_{r+1}}$ are also $d$-dimensional, hence either $\dim W(\Theta_{r+1,m,r,\vartheta},\psi)=0$ and we are done,
or $\dim W(\Theta_{r+1,m,r,\vartheta},\psi)=d$ but then $W(\mathrm{I},\psi)=W(\Theta_{r+1,m,r,\vartheta},\psi)$. Suppose the latter.

Consider $\lambda\in W(\mathrm{I},\psi)$ given by the Jacquet integral
$\lambda(f)=\int_{V_{(1,r)}}f(\langle w_{(r,1)},1\rangle\langle v,1\rangle)\psi(v_1)dv$, where $v$ is identified with a row in $F^{r}$.
This integral is absolutely convergent on $\mathrm{I}$. Indeed let $\lambda'(f)$ be the integral similar to $\lambda(f)$, but with $\psi$ omitted, and
note that the inducing data $\Theta_{r,m,r,\vartheta}$ in $\mathrm{I}$ is the image of $M(J_{r})$.
Using $J_{r+1}=\diag(J_r,1)w_{(r,1)}$ we obtain $M(J_{r+1})=\lambda'\circ M(J_r)$ ($\lambda'$ is an outer integral for
$M(J_{r+1})$). Since $M(J_{r+1})$ is absolutely convergent on the space of
$\mathrm{I}(\delta_{B_{\GL_{r+1}}}^{1/(2r)})$, so is $\lambda'(f)$ and thereby $\lambda(f)$.

For $b\in F^*$ let $\ell(b)f$ denote the left-translation of $f$ by $\langle\diag(I_{r},b),1\rangle$. Then
we have the functionals $\lambda_b\in W(\mathrm{I},\psi)$ defined by $\lambda_b(f)=\lambda(\ell(b)f)$. ($\{\lambda_b\}_{b\in\mathcal{G}}$ is a basis for $W(\mathrm{I},\psi)$, see e.g., \cite[Lemma~I.3.1]{KP}.) Since $W(\mathrm{I},\psi)=W(\Theta_{r+1,m,r,\vartheta},\psi)$, each
$\lambda_b$ factors through $\Theta_{r+1,m,r,\vartheta}$. Observe that $\Theta_{r+1,m,r,\vartheta}$ is a proper quotient of $\mathrm{I}$,
e.g., because by \eqref{eq:Jacquet of exceptional}, $J_{V_{(r,1)}}(\Theta_{r+1,m,r,\vartheta})$ is irreducible but
$J_{V_{(r,1)}}(\mathrm{I})$ is not. Thus $\Theta_{r+1,m,r,\vartheta}=\mathbb{V}\backslash\mathrm{I}$ for a representation $\mathbb{V}\ne0$
of $\GL_{r+1}^{(m)}$ and for all $b\in F^*$, $\lambda_b$ vanishes on $\mathbb{V}$. Therefore $\lambda_b(f)=0$ and equivalently
\begin{align*}
\int\limits_{V_{(1,r)}}f(\langle w_{(r,1)},1\rangle\langle v,1\rangle)\psi(bv_1)\,dv=0,\qquad\forall f\in\mathbb{V},\quad b\in F^*.
\end{align*}
We proceed using the method of \cite[p.~118]{JS}.
Fixing $v_2,\ldots,v_r$ for a moment, one can regard $f$ as a function $\xi_f(v_1)$ of $v_1\in F$, then the last identity becomes $\widehat{\xi}_f(b)=0$ for all $b\in F^*$ ($\widehat{\xi}_f$ - the Fourier transform of $\xi_f$). Since the integral is absolutely convergent, the Fourier inversion formula implies $\xi_f(0)=\int_F\widehat{\xi}_f(b)db=0$, i.e.,
\begin{align*}
\int\limits_{v\in V_{(1,r)},\,v_1=0}f(\langle w_{(r,1)},1\rangle\langle v,1\rangle)\,dv=0.
\end{align*}
Then as in \textit{loc. cit.} we eliminate the remaining coordinates and obtain $f(\langle w_{(r,1)},1\rangle)=0$ for all $f\in\mathbb{V}$
(see \cite[Proposition~62]{me12} or \S~\ref{RS The functional equation} here), which is a contradiction.
\end{proof}
We mention that Yamana \cite[Proposition~2.8]{Yamana2} proved the exceptional representations
of a double covering of $\GL_d$ of \cite{KP} do not support any Whittaker functional for $d>2$,
also using \cite{JS}. Our case is different from
\cite{Yamana2}: Here the representation $\mathrm{I}$ induced from $\Theta_{r,m,r,\vartheta}\otimes\Theta_{1,m,r,\vartheta}$ does not have a unique Whittaker model and
$\widetilde{C}_{\GL_{r+1}}$ is not abelian (as opposed to $\widetilde{C}_{\GL_{2r+1}}$ for the double coverings of \cite{KP}).

Now write a genuine irreducible generic representation $\tau$ of $\GL_k^{(m)}$ as the unique irreducible quotient of
$\Ind_{\widetilde{P}_{\beta}}^{\GL_k^{(m)}}(\otimes_{i=1}^d|\det|^{a_i}\tau_i)$ where $\tau_i$ are tempered and $a_1>\ldots>a_d$ (\cite{BJ}). Define
\begin{align}\label{rep:rho c tau in general}
\rho_c(\tau)=\Ind_{\widetilde{P}_{\beta rc}}^{\GL_{rkc}^{(m)}}(\otimes_{i=1}^d|\det|^{a_i}\rho_c(\tau_i)).
\end{align}
\begin{theorem}\label{theorem rho c tau unique}
For a genuine irreducible generic representation $\tau$ of $\GL_{k}^{(m)}$, $\rho_c(\tau)$ is $(rk,c)$.
Furthermore, when $F=\C$ and the inducing data of $\tau$ as a constituent of a principal series is in general position,
$\rho_c(\tau)$ is the $(rk,c)$ representation of \cite[\S~2.2]{CFK} constructed from $\rho_1(\tau)$.
\end{theorem}
\begin{proof}
Consider a non-archimedean field.
The essentially tempered case follows as in the proof of \cite[Theorem~5]{CFK}: Using Conjecture~\ref{conjecture:derivatives of rho c tau}, we
take the highest derivative and obtain $\rho_{c-1}(\tau)$, up to a power of $|\det|$. Repeating this $c$ times gives the genuine one-dimensional representation of $\mu_m$
because $\rho_1(\tau)$ was assumed to be $(rk,1)$.
The general case follows from the definition and Proposition~\ref{prop:rodier non}.

For $F=\C$, when $\tau$ is tempered (then $k=1$), $\rho_1(\tau)$ is unitary and
$\rho_c(\tau)$ is the irreducible $(r,c)$ representation defined in \cite[\S~2.2]{CFK} for $\rho_1(\tau)$ (see \eqref{eq:rho c tempered tau quotient using rho 1 and transitivity} and \cite[(2.2)]{CFK}). In fact
$\rho_c(\tau)=\Ind_{P_{(c^r)}}^{\GL_{rc}}(\otimes_{i=1}^{r}|\det|^{(r-2i+1)/(2r)}\tau\det_{\GL_c})$ (see \cite[p.~13, (2.5)]{CFK}).
Now by \eqref{rep:rho c tau in general}, in the general case
$\rho_c(\tau)=\Ind_{P_{(c^{rk})}}^{\GL_{rkc}}(\otimes_{j=1}^{k}\otimes_{i=1}^{r}|\det|^{a_j+(r-2i+1)/(2r)}\tau_j\det_{\GL_c})$, which is $(rk,c)$ by
\cite[Proposition~3]{CFK} (an application of \cite{AGS2015a,AGS2015b,GGS}) and is the $(rk,c)$ representation of \cite[\S~2.2]{CFK} up to a permutation of
the inducing data ($i$ and $j$).
\end{proof}

\begin{conjecture}\label{conjecture:rho c tau inductive for tempered tau}
If $\tau$ is a quotient of $\Ind_{\widetilde{P}_{\beta}}^{\GL_k^{(m)}}(\otimes_{i=1}^d\tau_i)$, where each
$\tau_i$ is genuine and supercuspidal (unitary or not), and moreover $\tau$ is essentially tempered and generic,
then for all $c\geq1$, $\rho_c(\tau)$ is a quotient of $\Ind_{\widetilde{P}_{\beta rc}}^{\GL_{rkc}^{(m)}}(\otimes_{i}\rho_c(\tau_i))$.
\end{conjecture}
By \S~\ref{covering GL r=1}, the case $r=1$ is implied by $m=1$. For the assertion concerning $\rho_c(\tau)$ when $m=1$ see
\cite[(2.5) and Lemma~8]{CFK}.

\begin{example}
For example when $\tau$ is genuine tempered generic and unramified, by Lemma~\ref{lemma:unr unitary Satake} and using the same notation $|\mu_i|=1$ for all $i$. Let $\tau_i$ be the representation of $\GL_1^{(m)}$ associated with $\mu_i$. If $\rho_c(\tau)$ is also unramified, the argument from \cite[Proposition~60]{me12} implies $\rho_c(\tau^{\vee})$ is a summand of $\Ind_{P_{\beta rc}}^{\GL_{rkc}^{(m)}}(\otimes_{i}\rho_c(\tau_i^{\vee}))$
(in \textit{loc. cit.} we further assumed $\mu_i\ne\mu_j$ for all $i\ne j$, this can be relaxed by regularizing the intertwining operator).
Since the definitions imply $\rho_c(\tau^{\vee})\cong\rho_c(\tau)^{\vee}$ and
$\rho_c(\tau_i^{\vee})\cong\rho_c(\tau_i)^{\vee}$, we deduce $\rho_c(\tau)$
is a quotient of $\Ind_{P_{\beta rc}}^{\GL_{rkc}^{(m)}}(\otimes_{i}\rho_c(\tau_i))$.
\end{example}

The central character of $\rho_c(\tau)$ is given by
\begin{align}\label{eq:central character rho c tau}
\rho_c(\tau)(\langle z^rI_{rkc},1\rangle)=\rho_1(\tau)(\langle z^rI_{rk},1\rangle)^c=\tau(\langle z^rI_{k},1\rangle)^{rc}.
\end{align}

\begin{lemma}\label{lemma:rho c dual}
We have $\rho_c(\tau)^*=\rho_c(\tau^*)$.
\end{lemma}
\begin{proof}
In general if $\beta=(\beta_1,\ldots,\beta_d)$ is a composition of $l$ and $\beta^*=(\beta_d,\ldots,\beta_1)$,
\begin{align}\label{eq:* on induced}
(\Ind_{\widetilde{P}_{\beta}}^{\GL_l^{(m)}}(\otimes_{i=1}^d\varrho_{i}))^*=\Ind_{\widetilde{P}_{\beta^*}}^{\GL_l^{(m)}}(\otimes_{i=1}^d\varrho_{d-i+1}^*).
\end{align}
If $\tau$ is tempered, $\rho_c(\tau)$ is the unique irreducible quotient of the induced representation $\mathrm{I}(\tau)$ defined by \eqref{eq:rho c tempered tau}, whence $\rho_c(\tau)^{\vee}$ is the unique irreducible subrepresentation of $\mathrm{I}(\tau)^{\vee}$.
Since for any genuine representations $\pi_1,\pi_2$ of $\GL_d^{(m)}$, $\Hom_{\GL_d^{(m)}}(\pi_1,\pi_2)=\Hom_{\GL_d^{(m)}}(\pi_1^*,\pi_2^*)$, and using
\eqref{eq:involution b*0 and dual}, it follows that
$(\rho_c(\tau)^*)^{\vee}$ is the unique irreducible subrepresentation of $(\mathrm{I}(\tau)^*)^{\vee}$. Hence
$\rho_c(\tau)^*$ is the unique irreducible quotient of $\mathrm{I}(\tau)^*=\mathrm{I}(\tau^*)$ (by \eqref{eq:* on induced})
and so is $\rho_c(\tau^*)$.
The general case follows from the definition, the tempered case and \eqref{eq:* on induced}.
\end{proof}
Recall the definitions \eqref{eq:def of eta on all of GLc triangle} and \eqref{eq:def of eta} of $\eta_{\triangle}$ and $\eta$, for an $(rk,c)$ representation $\rho$. When $\rho=\rho_c(\tau)$, we re-denote $\eta_{\triangle}$ by $\eta_{\tau,\triangle,c}$ and $\eta$ by $\eta_{\tau,c}$, and set $\eta_{\tau}=\eta_{\tau,1}$. We have $\eta_{\tau}(\langle a^rI_{rk},1\rangle)=\tau(\langle a^rI_{k},1\rangle)^r$. By Lemma~\ref{lemma:rho c dual}, $\eta_{\tau^*}(\langle a^rI_{rk},\epsilon\rangle)=\eta_{\tau}(\langle(a^rI_{k})^*,\epsilon\rangle)
=\epsilon\eta_{\tau}^{-1}(\langle a^rI_{k},1\rangle)$.

\begin{lemma}\label{lemma:eta tau c 1}
$\eta_{\tau,c}=\eta_{\tau}^c$.
\end{lemma}
\begin{proof}
First assume $\tau$ is tempered. We prove the result using induction on $c$, the case $c=1$ is clear.
For $c>1$, by Corollary~\ref{corollary:tempered} we can assume the $(rk,c)$ functional is given by \eqref{eq:mnk functional using w_{n,m,k}}. Then the result follows by replacing $\xi$ in \eqref{eq:mnk functional using w_{n,m,k}} with its right-translate by $\langle aI_{rkc},1\rangle$, noting that $\langle aI_{rkc},1\rangle=\langle aI_{rkl},1\rangle\langle aI_{rk(c-l)},1\rangle$ and $\gamma_{\psi'}^l(a)\gamma_{\psi'}^{c-l}(a)=\gamma_{\psi'}^c(a)$ (because $(a,a)_m=1$),
and applying the induction hypothesis to $\rho_l(\tau)$ and $\rho_{c-l}(\tau)$. Since we have uniqueness, it suffices to check one functional.
The general case follows from \eqref{rep:rho c tau in general} using \eqref{eq:(k,c)  functional using an integral} and the tempered case.
\end{proof}

Let $\tau$ be a genuine unitary irreducible generic representation of $\GL_k^{(m)}$.
For an integer $l\geq1$ and $\zeta\in\C^{l}$, consider the standard intertwining operator
\begin{align*}
M(\zeta,w_{(k^{l})}):\Ind_{\widetilde{P}_{(k^{l})}}^{\GL_{lk}^{(m)}}(\otimes_{i=1}^{l}|\det|^{\zeta_i}\tau)\rightarrow\Ind_{\widetilde{P}_{(k^{l})}}^{\GL_{lk}^{(m)}}(\otimes_{i=1}^{l}|\det|^{\zeta_{l-i+1}}\tau),
\end{align*}
defined by an absolutely convergent integral for $\Real(\zeta)$ in a suitable cone, and in general by meromorphic continuation. Denote \begin{align}\label{eq:zeta_0}
\zeta^{(l)}=((l-1)/(2r),(l-3)/(2r),\ldots, (1-l)/(2r))\in\C^{l}.
\end{align}
\begin{conjecture}\label{conjecture:intertwining op Speh}
For all $1<l\leq r$, $M(\zeta,w_{(k^{l})})$ is well defined at $\zeta=\zeta^{(l)}$ and its image is irreducible.
For all $c\geq1$, $M(\zeta,w_{(k^{rc})})$ is well defined at $\zeta=\zeta^{(rc)}$, its image is irreducible and isomorphic
to the representation $\rho_c(\tau)$. In particular it is $(rk,c)$.
\end{conjecture}
For $r=1$ this follows from the case $m=1$ in \cite[Proposition~2.2]{Jac4} (see also \cite[Proposition~I.10]{MW4}); for $k=1$ it holds because for all $l$, the image of $M(\zeta,w_{(1^{l})})$ at $\zeta^{(l)}$ is an exceptional representation (\cite{KP,Gao5}).

We quickly recall the global construction of $(rk,c)$ representations from \cite[\S~2.4]{me12}.
Let $\tau$ be a genuine cuspidal representation of $\GL_k^{(m)}(\A)$ (see \S~\ref{Global coverings}). Consider the representation
\begin{align}\label{ind1}
\Ind_{\widetilde{P}_{(k^{rc})}({\A})}^{\GL_{rkc}^{(m)}({\A})}
(|\det|^{\zeta_1}\tau\otimes\ldots\otimes |\det|^{\zeta_{rc}}\tau).
\end{align}
For a standard $\widetilde{K}_{\GL_{rkc}}$-finite vector $\xi$ in the space of \eqref{ind1},
we have the Eisenstein series
\begin{align}\label{eq:Eisenstein series on GL}
E(g;\zeta,\xi)=\sum\limits_{y\in P_{(c^{rk})}(F)\backslash \GL_{rkc}(F)}\xi(\zeta,\langle y,(\eta_{rkc}^{\diamondsuit})^{-1}(y)\rangle g).
\end{align}
The function $g\mapsto E(g;\zeta,\xi)$ is an automorphic form on $\GL_{rkc}^{(m)}(\A)$. Let
\begin{align}\label{eq:limit}
E_{-1}(g;\xi)=\lim\limits_{\zeta\to\zeta^{(rc)}}\prod_{i=1}^{rc-1}(r(\zeta_i-\zeta_{i+1})-1)E(g;\zeta,\xi)
\end{align}
and let $\mathcal{E}_{\tau}$ denote the corresponding residual representation.

\begin{conjecture}\label{Shimura conjecture}
The partial $L$-function $L^S(s,\tau\times\widetilde{\tau})$ has a simple pole at $s=1$ and is holomorphic and nonzero for $\Real(s)>1$.
(As explained in \S~\ref{unr L functions} and \S~\ref{dual reps}, $L^S(s,\tau\times\widetilde{\tau})=L^S(s,\tau\times\tau^{\vee})$.)
\end{conjecture}
When $r=1$ or $k=1$, this conjecture follows from the case $m=1$ proved in \cite{JS2,JS1}.

By \cite[Theorem~53]{me12}, $\mathcal{E}_{\tau}$ belongs to the discrete spectrum of the space of $L^2$ automorphic forms on $\GL_{rkc}^{(m)}({\A})$, which transform by $\varrho_{\tau}^{rc}$ under the center of $\GL_{rkc}^{(m)}({\A})$,
where $\varrho_{\tau}$ is the central character of $\tau$. Combining
\cite[Theorem~53]{me12} with Conjecture~\ref{conjecture:intertwining op Speh} we deduce that for each $\nu$, $(\mathcal{E}_{\tau})_{\nu}=\rho_c(\tau_{\nu})$ which is $(rk,c)$ by Theorem~\ref{theorem rho c tau unique}. By \cite[Theorem~54]{me12}, $\mathcal{E}_{\tau}$ admits an $(rk,c)$ functional (which is then unique up to scaling) $\Lambda$ in the form of a Fourier coefficient.

We can therefore write for a factorizable automorphic form $\xi$ in the space of $\mathcal{E}_{\tau}$,
\begin{align}\label{eq:factorizable rk c functional}
\Lambda(\mathcal{E}_{\tau}(h)\xi)=\prod_{\nu}\lambda_{\nu}(\rho_c(\tau_{\nu})(h_{\nu})\xi_{\nu}),\qquad h\in\GL_{rkc}^{(m)}(\A).
\end{align}
Here $\lambda_{\nu}$ are local $(rk,c)$ functionals and for almost all $\nu$, $\lambda_{\nu}$ is scaled to be $1$ on a choice of an unramified vector in the space of $\rho_c(\tau_{\nu})$. At these places both $\eta_{\tau_{\nu},\triangle,c}$ and $\eta_{\tau_{\nu},c}$ are unramified (see \S~\ref{Definition and generalities}).
Hence $\eta_{\tau,\triangle,c}=\prod_{\nu}\eta_{\tau_{\nu},\triangle,c}$ is a well defined genuine character and
\begin{align}\label{eq:global eta triangle}
\Lambda(\mathcal{E}_{\tau}(b)\xi)=
\eta_{\tau,\triangle,c}(b)\Lambda(\xi),\qquad b\in\widetilde{\GL}_{c}^{\triangle}(\A).
\end{align}
The character $\eta_{\tau,\triangle,c}$ is automorphic because for
$b=\langle y^{\triangle},(\eta_{rkc}^{\diamondsuit})^{-1}(y^{\triangle})\rangle$ with $y\in\GL_{c}(F)$,
the l.h.s.~ of \eqref{eq:global eta triangle} equals $\Lambda(\xi)$.
Thus if we choose a nontrivial additive character $\psi'$ of $F\backslash\A$, and quasi-characters $\eta_{\tau_{\nu}}'$ of
$F_{\nu}^*$ such that for all $b_{\nu}\in\GL_c(F_{\nu})$,
\begin{align*}
\eta_{\tau_{\nu},\triangle,c}(\langle b_{\nu}^{\triangle},\eta_{(c^{rk})}(b_{\nu}^{\triangle})\eta_c^{\diamondsuit,(rk)}(b_{\nu})\rangle)=[\gamma_{\psi_{\nu}'}(\det b_{\nu})]\eta_{\tau_{\nu}}'(\det b_{\nu})
\end{align*}
($\gamma_{\psi_{\nu}'}$ is omitted if $m|rk$), which is possible by the local version of \eqref{eq:RS rho rkc rk on Levi}, then $\eta_{\tau}'=\prod_{\nu}\eta_{\tau_{\nu}}'$ is defined. The character
$\eta_{\tau}'\circ\det$ of $\A^*$ is automorphic by \eqref{eq:global eta triangle}, \eqref{eq:eta eta rk m divides or not eta diamond triangle rat} and because $\gamma_{\psi'}$ is trivial on $F^*$. Then we have
\begin{align}\label{eq:aut character eta identity}
\eta_{\tau,\triangle,c}(\langle b^{\triangle},\eta_{(c^{rk})}(b^{\triangle})\eta_c^{\diamondsuit,(rk)}(b)\rangle)=[\gamma_{\psi'}(\det b)]\eta_{\tau}'(\det b),\qquad\forall\,b\in\GL_c(\A),
\end{align}
which makes sense by \eqref{eq:RS rho rkc rk on Levi}. Also, since
$\mathcal{E}_{\tau}$ admits a central character, \eqref{eq:global eta triangle}--\eqref{eq:aut character eta identity} imply
\begin{align}\label{eq:central aut character eta identity}
\mathcal{E}_{\tau}(\langle b^{\triangle},\eta_{(c^{rk})}(b^{\triangle})\eta_c^{\diamondsuit,(rk)}(b)\rangle)\xi=[\gamma_{\psi'}(\det b)]\eta_{\tau}'(\det b),\qquad\forall\,b\in C_{\GL_c}(\A).
\end{align}

We mention that in \cite[\S~2.4]{me12} the residual representation was denoted $\mathcal{L}_{\tau,c}$, and $\mathcal{E}_{\tau}$ was chosen to be an irreducible summand which admits a nonzero $(rk,c)$ functional. Here (as opposed to \cite[Conjecture~50]{me12}) we already assume the images of the local intertwining operators
$M(\zeta,w_{(k^{rc})})$ are irreducible at $\zeta^{(rc)}$, so that $\mathcal{L}_{\tau,c}$ is irreducible and
$\mathcal{L}_{\tau,c}=\mathcal{E}_{\tau}$.

\section{Rankin--Selberg integrals}\label{RS integrals}
\subsection{The integrals and $\gamma$-factor}\label{ints and RS integrals}
In \cite[Appendix~C]{CFK} we introduced a new family of $\GL_c\times\GL_k$ integrals (independently constructed
in \cite{G8,G7}, see also \cite{LapidMao2018}), whose structure resembles doubling type integrals but produce Rankin--Selberg type factors. We generalize \cite[Appendix~C]{CFK} to $\GL_c^{(m)}\times \GL_k^{(m)}$ and use it to define and study local factors and
a global complete $L$-function satisfying a functional equation. (The case of
$\GL_1^{(m)}\times \GL_k^{(m)}$ already appeared in \cite[\S~2.7]{me12} where we computed the integral with unramified data.)

Let $F$ be a local field.
Let $\psi$ be a nontrivial additive character of $F$, and fix the Haar measure $dx=d_{\psi}x$ of $F$ which is self-dual with respect to $\psi$.
For $b\in\GL_c$, put $e_1(b)=\diag(b,I_{(rk-1)c})$ and
$e_2(b)=\diag(I_c,b,\ldots,b)$, $e_1(b)e_2(b)$ is the diagonal embedding of $\GL_c$ in $M_{(c^{rk})}$ ($r$ was defined in \S~\ref{Local coverings}).
Identity~\eqref{eq:block compatibility on Levi of P} implies that $e_1(\GL_c)$ and $e_2(\GL_c)$ commute in $\GL_{rkc}^{(m)}$.
For the covering of $e_2(\GL_c)$ observe that by
\eqref{eq:block compatibility on Levi of P} and \eqref{eq:sigma rkc diamondsuit on G triangle},
\begin{align*}
\sigma_{rkc}^{\diamondsuit}(e_2(b),e_2(b'))=
(\sigma_{c}^{\diamondsuit}(b,b'))^{rk-1}=
\left(\frac{\varsigma_{\triangle,c}(b)\varsigma_{\triangle,c}(b')}{\varsigma_{\triangle,c}(bb')}\right)^{rk}(\det b,\det b')_{m/r}^{k}
\sigma_{c}^{\diamondsuit}(b,b')^{-1}.
\end{align*}
Recall $\GL_c^{(m)}=\GL_c[\sigma_c^{\diamondsuit}]$ by definition. It follows that the mappings
\begin{align}\label{eq:RS embeddings}
\langle b,\epsilon\rangle\mapsto\langle \mathfrak{e}_1(b),\epsilon\rangle,\qquad
\langle b,\epsilon\rangle\mapsto\langle \mathfrak{e}_2(b),\varsigma_{\triangle,c}(b)^{-rk}\epsilon^{-1}\rangle
\end{align}
lift of the embedding $\GL_c\times \GL_c<\GL_{rkc}$ to an embedding
\begin{align}\label{eq:RS lift of embedding}
\{(\epsilon_1,\epsilon_2)\in\mu_m^2:\epsilon_1=\epsilon_2\}\backslash \GL_c^{(m)}\times \GL_c[\sigma_{c}^{\diamondsuit}\cdot(\det ,\det )_{m/r}^{k}] \rightarrow \GL_{rkc}^{(m)}\qquad (\mu_m^2=\mu_m\times\mu_m).
\end{align}
If $m\nmid rk$, for any genuine representation $\pi_0$ of $\GL_c^{(m)}$, $(\pi_0)_{\psi'}(\langle b,\epsilon\rangle)=\gamma_{\psi'}^{-1}(\det b)\pi_0(\langle b,\epsilon\rangle)$ is a genuine representation of $\GL_c[\sigma_{c}^{\diamondsuit}\cdot(\det ,\det )_{m/r}^{k}]$ (acting on the same space of $\pi$). For uniformity when $m|rk$, denote $(\pi_0)_{\psi'}=\pi_0$.

Let $\tau$ be a genuine irreducible generic representation and $\rho_c(\tau)$ be the corresponding $(rk,c)$ representation,
realized in its unique $(rk,c)$ model $W_{\psi}(\rho_c(\tau))$. If $m\nmid rk$, we take the character $\psi'$ above such
that for all $b\in\GL_c$, $\eta_{\tau,\triangle,c}(\langle b^{\triangle},\varsigma_{\triangle,c}(b)^{-rk}\rangle)=\gamma_{\psi'}(\det b)\eta'(\det b)$ for a quasi-character $\eta'$ of $F^*$ (see \S~\ref{Definition and generalities}). Let $\pi$ be a genuine irreducible admissible representation of $\GL_c^{(m)}$. Throughout this section assume $rk>1$, the case $rk=1$ is quickly explained at the end of this section, and in addition $s'=r^{-1}(s-1/2)+1/2$.

Consider the following integral, for a matrix coefficient $\omega$ of $\pi^{\vee}$ and $W\in W_{\psi}(\rho_c(\tau))$:
\begin{align}\label{eq:Z integral GL 1 GL rk}
Z(s,\omega,W)=\int\limits_{\GL_c}W(\langle e_1(a),1\rangle)\omega(\langle a,1\rangle)|\det a|^{s'-(rkc-2c+1)/2}\,da.
\end{align}
It is formally well defined by \eqref{eq:RS embeddings} and \eqref{eq:RS lift of embedding} and because $W$ is genuine and $\omega$ is anti-genuine. The integral is absolutely convergent for $\Real(s)\gg0$ independent of $W$ and $\omega$.

Integral~\eqref{eq:Z integral GL 1 GL rk} also satisfies an equivariance property with respect to $e_2(\GL_c)$.
Write $\omega(\langle a,1\rangle)=\pi^{\vee}(\langle a,1\rangle)\varphi^{\vee}(\varphi)$ for $\varphi$ (resp., $\varphi^{\vee}$) in the space of $\pi$ (resp.,\ $\pi^{\vee}$).
According to \eqref{eq:RS embeddings} and \eqref{eq:RS lift of embedding} and since both $\pi$ and $\rho_c(\tau)$ are genuine, the function
\begin{align*}
b\mapsto W(\langle e_1(a),1\rangle\langle e_2(b),\varsigma_{\triangle,c}(b)^{-rk}\rangle)\pi^{\vee}(\langle a,1\rangle)\varphi^{\vee}(\pi_{\psi'}(\langle b,1\rangle)\varphi)
\end{align*}
is well defined as a function on $\GL_c$, i.e., it does not depend on the choice of section $\GL_c\rightarrow\GL_c[\sigma_{c}^{\diamondsuit}\cdot(\det ,\det )_{m/r}^{k}]$.

Changing variables $a\mapsto ba$ in \eqref{eq:Z integral GL 1 GL rk}, using
\eqref{eq:def of eta on all of GLc triangle} which implies
\begin{align*}
W(\langle e_1(b),1\rangle\langle e_2(b),\varsigma_{\triangle,c}(b)^{-rk}\rangle)=
W(\langle b^{\triangle},\varsigma_{\triangle,c}(b)^{-rk}\rangle)=[\gamma_{\psi'}(\det b)]\eta'(\det b)W(\langle I_{rkc},1\rangle)
\end{align*}
and combining it with the definition of $\pi_{\psi'}$ ($\gamma_{\psi'}$ is omitted precisely when $\pi_{\psi'}=\pi$), we deduce
\begin{align}\label{eq:RS def equivariance on e2(G)}
Z(s,\pi^{\vee}(\cdot)\varphi^{\vee}(\pi_{\psi'}(\langle b,1\rangle)\varphi),e_2(b)\cdot W)=|\det b|^{s'-(rkc-2c+1)/2}\eta'(\det b)Z(s,\pi^{\vee}(\cdot)\varphi^{\vee}(\varphi),W).
\end{align}

When all data are unramified, both $\pi$ and $\tau$ are parameterized using $\vartheta$ (see \S~\ref{unr L functions}) and $\omega$ and $W$ are also normalized,
\begin{align}\label{RS unr integral}
Z(s,\omega,W)=L_{\vartheta}(s,\pi^{\vee}\times\tau)=L(s,\pi^{\vee}\times\tau)=L(s,\widetilde{\pi}\times\tau).
\end{align}
In the case $c=1$ this follows from \cite[(2.54)]{me12}. In the general case it was proved in \cite{G7} (for coverings of \cite{KP} but the argument extends and is even simpler for the coverings here), by expressing the integral as a product of $c$ unramified integrals
for $\GL_1^{(m)}\times \GL_k^{(m)}$.

The definition of the $\gamma$-factor involves an additional integral. Fix the splitting
$v\mapsto\langle v,\varsigma^-(v) \rangle$ of $N_{\GL_{rkc}}^-$ in $\GL_{rkc}^{(m)}$.
Set $w=\left(\begin{smallmatrix}&I_{(rk-1)c}\\I_c&\end{smallmatrix}\right)$ and for $v\in V_{((rk-2)c,c)}^-$,
$[v]=\diag(I_c,v)$. Define
\begin{align}\label{eq:second Z integral GL 1 GL rk}
&Z^*(s,\omega,W)=\int\limits_{\GL_c}\int\limits_{V_{((rk-2)c,c)}^-}
W(\langle{}^we_1(a),1\rangle\langle[v],\varsigma^-([v])\rangle\langle w,1\rangle)\omega(\langle a,1\rangle)|\det a|^{s'-1+(rkc-2c+1)/2}\,dv\,da.
\end{align}
Here $dv$ is the product measure defined using $dx$.
This integral is absolutely convergent for $\Real(s)\ll0$ independent of the data.
When data are unramified,
\begin{align}\label{second RS unr integral}
&Z^*(s,\omega,W)=L(1-s,\pi\times\widetilde{\tau}).
\end{align}
This is proved by showing that the integrand vanishes unless the coordinates of $v$ belong in $\mathcal{O}$, then applying
\eqref{RS unr integral}. See the proof of \cite[Lemma~85]{me12} for $c=1$ and \cite[p.~117]{CFK}.

The local theory has a global counterpart recently introduced in the linear setting in \cite{G7} (the case $c=1$ for coverings of \cite{KP} was constructed in \cite{BF2}). Our focus here is local, but the global integral will be presented in
\S~\ref{RS Crude functional equation} below, where we will also prove it is Eulerian.

If $F=\C$, $\GL_{rkc}^{(m)}$ is split over $\GL_{rkc}$ and both integrals we defined are the $\GL_c\times\GL_{rk}$ integrals of \cite[Appendix~C]{CFK}. Their theory including convergence, meromorphic continuation, continuity properties and the existence of a $\gamma$-factor were proved in \textit{loc. cit.}. We now focus mainly on the non-archimedean case.

Let $V_{(c^{rk})}^{\circ}=\{v\in V_{(c^{rk})}:v_{1,2}=0\}$ ($v=(v_{i,j})$, $v_{i,j}\in\Mat_c$).
The $(rk,c)$ character $\psi$ restricts to a character of $V_{(c^{rk})}^{\circ}$. Denote by $\mathcal{T}(s,\tau,\pi)$ the space of trilinear forms $\mathcal{T}$ on the space of $\rho_c(\tau)\times\pi^{\vee}\times\pi$ such that for all
$(\xi,\varphi^{\vee},\varphi)$, $a,b\in \GL_c$ and $v\in V_{(c^{rk})}^{\circ}$,
\begin{align}\label{eq:equivariance for bilinear RS}
&\mathcal{T}(\rho_c(\tau)(\langle ve_1(a)e_2(b),\varsigma_{\triangle,c}(b)^{-rk}\rangle)\xi,\pi^{\vee}(\langle a,1\rangle)\varphi^{\vee},
\eta'(\det b)^{-1}\pi_{\psi'}(\langle b,1\rangle)\varphi)
\\&\qquad=\psi(v)|\det ab^{-1}|^{-s+(rkc-2c+1)/2}\mathcal{T}(\xi,\varphi^{\vee},\varphi).\nonumber
\end{align}
The following is the key result underlying the functional equation \eqref{eq:gamma RS} below.
\begin{theorem}\label{theorem:uniqueness for bilinear RS}
Outside a discrete subset of $s$, $\dim \mathcal{T}(s,\tau,\pi)\leq 1$.
\end{theorem}
\begin{proof}
The proof of \cite[Theorem~C.1]{CFK} (the linear case) extends to the coverings here, one simply replaces every occurrence of $k$ in
\textit{loc. cit.} with $rk$. Note that the uniqueness of the Whittaker models is mentioned in \textit{loc. cit.} but never actually used.
\end{proof}

A quick computation implies both $Z(s,\omega,W)$ and $Z^*(s,\omega,W)$ belong to
$\mathcal{T}(s,\tau,\pi)$ for all $s$ in their domain of absolute convergence. Indeed this is clear for the former with the aid of \eqref{eq:RS def equivariance on e2(G)}; for the latter note that for $z\in V_{(c^{rk})}^{\circ}$, ${}^wz=z'[v]$ where $z'\in\diag(V_{(c^{rk-1})},I_c)$, then by \eqref{eq:sigma on h and v},
\eqref{eq:sigma conjugate v by h} and \eqref{eq:sigma conjugate v- to v by h},
${}^w\langle z,1\rangle = \langle z',1\rangle\langle [v],\varsigma^-([v])\rangle$.
Additionally, as in the linear case the integrals can be made constant ($\GL_{rkc}^{(m)}$ is split over a sufficiently small compact open subgroup).
\begin{corollary}\label{corollary:Z and Z^* are meromorphic}
Assume $F$ is non-archimedean. The integrals $Z(s,\omega,W)$ and $Z^*(s,\omega,W)$ admit meromorphic continuation to rational functions
in $q^{-s}$. Moreover, the set of possible poles in $q^{-s}$ can be taken to be finite and independent of
$\omega$ or $W$.
\end{corollary}
\begin{proof}
This follows from Theorem~\ref{theorem:uniqueness for bilinear RS}, the above remarks and Bernstein's continuation principle (in \cite{Banks}).
\end{proof}
Put $\mathfrak{i}_{\GL_c}=\langle -I_c,1\rangle\in \widetilde{C}_{r,c}$ ($\mu_{2m}\subset F^*$)
and
\begin{align*}
\vartheta(s,\widetilde{\pi},\tau)=\pi(\langle r^rI_c,1\rangle)^k \widetilde{\tau}(\langle r^rI_k,1\rangle)^c|r|^{-kc(s-1/2)}.
\end{align*}
Theorem~\ref{theorem:uniqueness for bilinear RS} also implies the existence of a $\gamma$-factor $\gamma(s,\widetilde{\pi}\times\tau,\psi)$ such that for all $\omega$ and $W$,
\begin{align}\label{eq:gamma RS}
\gamma(s,\widetilde{\pi}\times\tau,\psi)Z(s,\omega,W)=
\widetilde{\pi}(\mathfrak{i}_{\GL_c})^{rk-1}\vartheta(s,\widetilde{\pi},\tau)Z^*(s,\omega,W).
\end{align}
By the above discussion this factor is well defined, not identically zero and belongs to $\C(q^{-s})$.
Since $\dim\Hom_{\GL_1^{(m)}}(\pi,\pi)=1$ and $\rho_c(\tau)$ is $(rk,c)$, $\gamma(s,\widetilde{\pi}\times\tau,\psi)$ is independent of the concrete realizations of $\pi$ and $\tau$.

Over $\C$ define $\gamma(s,\widetilde{\pi}\times\tau,\psi)$ again by \eqref{eq:gamma RS} (\cite[Theorem~C.1]{CFK} is valid in the archimedean case as well).
Recall the representation $\pi^r$ defined in \S~\ref{reps}.

The following is our main local result on the Rankin--Selberg integrals.
\begin{theorem}\label{theorem:RS ten commendments}
The factor $\gamma(s,\widetilde{\pi}\times\tau,\psi)$ satisfies the following properties.
\begin{itemize}[leftmargin=*]
\item Unramified twisting:
\begin{align}\label{eq:RS Unramified twisting}
\gamma(s,|\det|^{-s_1}\widetilde{\pi}\times|\det|^{s_0}\tau,\psi)=\gamma(s+rs_0-rs_1,\widetilde{\pi}\times\tau,\psi).
\end{align}
\item Multiplicativity I: If $\widetilde{\pi}\subset\Ind_{\widetilde{P}_{(l,c-l)}}^{\GL^{(m)}_c}(\sigma)$ where $\sigma=\widetilde{\pi}_1\otimes\widetilde{\pi}_2$ is irreducible and admissible,
    \begin{align}\label{eq:RS mult I}
     \gamma(s,\widetilde{\pi}\times\tau,\psi)=\gamma(s,\widetilde{\pi}_1\times\tau,\psi)\gamma(s,\widetilde{\pi}_2\times\tau,\psi).
    \end{align}
\item Unramified factors: when data are unramified and $\pi,\tau$ are parameterized with one parameter,
    \begin{align}\label{eq:RS unramified factors}
    \gamma(s,\widetilde{\pi}\times\tau,\psi)=\frac{L(1-s,\pi\times\widetilde{\tau})}{L(s,\widetilde{\pi}\times\tau)}.
    \end{align}
\item Dependence on $\psi$: for any $b\in F^*$,
    \begin{align}\label{eq:RS dependence on psi}
    \gamma(s,\widetilde{\pi}\times\tau,\psi_b)=\widetilde{\pi}(\langle b^rI_c,1\rangle)^{k}\eta_{\tau}(\langle b,1\rangle)^c|b|^{kc(s-1/2)}\gamma(s,\widetilde{\pi}\times\tau,\psi).
    \end{align}
\item Functional equation:
    \begin{align}\label{eq:RS functional equation}
    \gamma(s,\widetilde{\pi}\times\tau,\psi)\gamma(1-s,\pi\times\widetilde{\tau},\psi^{-1})=1.
    \end{align}
\item Over $F=\C$ let $\varphi:\C^*\rightarrow\GL_{c}(\C)\times\GL_{k}(\C)$ be the homomorphism attached to
$(\pi^{\vee})^r\otimes\tau^r$ and $\epsilon(s,\mathrm{st}\circ\varphi,\psi)$ and $L(s,\mathrm{st}\circ\varphi)$ be Artin's local factors attached to $\mathrm{st}\circ\varphi$ by Langlands' correspondence (\cite{Bo,La3}), where $\mathrm{st}$ is the standard representation. Then
\begin{align}\label{eq:RS Archimedean property}
\gamma(s,\widetilde{\pi}\times\tau,\psi)=\epsilon(s,\mathrm{st}\circ\varphi,\psi)\frac{L(1-s,\mathrm{st}^{\vee}\circ\varphi)}{L(s,\mathrm{st}\circ\varphi)}.
\end{align}
\item Crude functional equation: let $F$ be a number field with a ring of adeles $\A$, $\psi$ be a nontrivial character of $F\backslash\A$, and assume $\pi$ and $\tau$ are genuine cuspidal representations of $\GL_c^{(m)}(\A)$ and $\GL_{k}^{(m)}(\A)$. With the global parametrization and the set $S$ as in \S~\ref{unr L functions},
    \begin{align}\label{eq:RS global property}
    L^S(s,\widetilde{\pi}\times\tau)=\prod_{\nu\in S}\gamma(s,\widetilde{\pi}_{\nu}\times\tau_{\nu},\psi_{\nu})L^S(1-s,\pi\times\widetilde{\tau}).
    \end{align}
\end{itemize}
\end{theorem}
We show that the factor appearing on the r.h.s. of \eqref{eq:RS dependence on psi} makes sense in the global context of \eqref{eq:RS global property}, i.e.,
$\prod_{\nu}\widetilde{\pi}_{\nu}(\langle b_{\nu}^rI_c,1\rangle)=1$ and $\prod_{\nu}\eta_{\tau_{\nu}}(\langle b_{\nu},1\rangle)=1$ for $b\in\GL_c(F)$.
First note that $\widetilde{\pi}_{\nu}(\langle b_{\nu}^rI_c,1\rangle)=\pi_{\nu}(\langle b_{\nu}^rI_c,1\rangle)^{-1}$, and we show
$\prod_{\nu}\pi_{\nu}(\langle b_{\nu}^rI_c,1\rangle)=1$.
In the local setting $\pi$ is a representation of $\GL_c[\sigma_c^{\diamondsuit}]$, but as a component of a cuspidal $\pi$,
$\pi_{\nu}$ is a representation of $\GL_c(F_{\nu})[\rho_c^{\diamondsuit}]$. To regard it as a representation of $\GL_c(F_{\nu})[\sigma_c^{\diamondsuit}]$, say, $\pi_{\nu}'$, define
$\pi_{\nu}'(\langle b_{\nu},\epsilon\rangle)=\pi_{\nu}(\langle b_{\nu},(\eta_{c,\nu}^{\diamondsuit})^{-1}(b_{\nu})\epsilon\rangle)$
(see \S~\ref{Local coverings}, \S~\ref{Global coverings}). With this observation in mind, we prove $\prod_{\nu}\pi_{\nu}'(\langle b_{\nu}^rI_c,1\rangle)=1$:
Because $\eta_c^{\diamondsuit}=\prod_{\nu}\eta_{c,\nu}^{\diamondsuit}$ is defined on $\GL_c(F)$,
$\prod_{\nu}\pi_{\nu}'(\langle b_{\nu}^rI_c,1\rangle)=\pi(\langle b^r,(\eta_c^{\diamondsuit})^{-1}(b^r)\rangle)=1$ since $\pi$ is automorphic.
Similarly by \eqref{eq:def of eta} (with $c=1$) and \eqref{eq:factorizable rk c functional},
\begin{align*}
\prod_{\nu}\eta_{\tau_{\nu}'}(\langle b_{\nu},1\rangle)\lambda_{\nu}(\xi_{\nu})
=\prod_{\nu}\lambda_{\nu}(\rho_1(\tau_{\nu}')(\langle b_{\nu}I_{rk},(\eta_{c,\nu}^{\diamondsuit})^{-1}(b_{\nu})\rangle)\xi_{\nu})
=\lambda(\mathcal{E}_{\tau}(\langle bI_{rk},(\eta_{c}^{\diamondsuit})^{-1}(b)\rangle)\xi)=
\lambda(\xi),
\end{align*}
whence $\prod_{\nu}\eta_{\tau_{\nu}'}(\langle b_{\nu},1\rangle)=1$.

Note that \eqref{eq:RS dependence on psi} is not symmetric with respect to $\pi$ and $\tau$, unless one proves
$\eta_{\tau}(\langle b,1\rangle)=\tau(\langle b^rI_k,1\rangle)$. We conjecture this is the case but will not assume it.

We expect the analog of \eqref{eq:RS mult I} to apply to $\tau$, i.e., if $\tau$ is a quotient of $\Ind_{\widetilde{P}_{\beta }}^{{\GL_{k}}^{(m)}}(\otimes_{i=1}^d\tau_i)$,
\begin{align}\label{eq:tau mult}
\gamma(s,\widetilde{\pi}\times\tau,\psi)=\prod_{i=1}^d\gamma(s,\widetilde{\pi}\times\tau_i,\psi).
\end{align}
We leave this to a future work because Theorem~\ref{theorem:RS ten commendments} will be sufficient for our purpose
here (see \S~\ref{The local integrals}). At present it is not clear how to prove \eqref{eq:tau mult} even in the linear case, and
at any rate the case $k>1$ for $r>1$ is conjectural. Henceforth \eqref{eq:tau mult} will not be used except in
\S~\ref{RS complete L factors}.

For $rk=1$ let $\gamma(s,\widetilde{\pi}\times\tau,\psi)$ be the factor $\gamma(s,(\gamma_{\psi'}\otimes\tau)\times(\gamma_{\psi'}\otimes\pi^*),\psi)$ defined in \cite{JPSS,JS3}.
This definition is independent of $\psi'$. Then $L(s,\pi\times\tau)$ is defined and
Theorem~\ref{theorem:uniqueness for bilinear RS} is known. See \S~\ref{covering GL r=1}.

\subsection{Proof of Theorem~\ref{theorem:RS ten commendments}}\label{proof:RS ten commendments}
\subsubsection{Unramified twisting}
The integral for $|\det|^{-s_1}\pi^*\times\tau$ involves a matrix coefficient $\omega_{s_1}$ of $|\det|^{-s_1}\pi^{\vee}$
($|\det|^{-s_1}\pi^*=(|\det|^{s_1}\pi)^*$),
then $\omega_{s_1}(\langle a,1\rangle)=|\det a|^{-s_1}\omega_0(\langle a,1\rangle)$. Also $\rho_c(|\det|^{s_0}\tau)=|\det|^{s_0}\rho_c(\tau)$
and $\vartheta(s,|\det|^{-s_1}\pi^*,|\det|^{s_0}\tau)=
\vartheta(s+rs_0-rs_1,\pi^*,\tau)$.
These observations imply \eqref{eq:RS Unramified twisting}.

\subsubsection{Multiplicativity I}
The proof of \eqref{eq:RS mult I} in \cite[Theorem~C.2]{CFK} extends to the coverings here. One needs to replace
all occurrences of $k$ with $rk$, and use \eqref{eq:sigma on h and v}--\eqref{eq:sigma conjugate v- to v by h}.
See \S~\ref{RS The functional equation} for the justifications using convolutions of $W$ against Schwartz functions (this technique
was used repeatedly in the proof of \cite[Theorem~C.2]{CFK}).

\subsubsection{Unramified factors}
Because $\eta^{\diamondsuit}_{d}$ is trivial on $T_{\GL_d}\cap K_{\GL_d}$ (see \S~\ref{Local coverings}),
$\vartheta(s,\widetilde{\pi},\tau)=1$. Identity \eqref{eq:RS unramified factors} then
follows from \eqref{RS unr integral} and \eqref{second RS unr integral}. In fact one can deduce this
directly from the case $c=1$, \cite[(2.54)]{me12} and \cite[Lemma~85]{me12} (thereby avoiding \cite{G7}) using \eqref{eq:RS mult I}.

\subsubsection{Dependence on $\psi$} First note that by \eqref{eq:Nice GL $2$-cocycle on torus}
and because $(b,b)_m=1$, $\sigma_{rkc}^{\diamondsuit}$ is trivial on any $t,t'\in T_{\GL_{rkc}}$ having only integer powers of
$b$ on the diagonal, so that $\langle t,1\rangle\langle t',1\rangle=\langle tt',1\rangle$ and in particular $t$ and $t'$ commute
in $\GL_{rkc}^{(m)}$. Put $t_b=\diag(I_c,b^{-1}I_c,\ldots,b^{1-rk}I_c)$. The functions $W_b(h)=W(\langle t_b,1\rangle h)$ ($h\in\GL_{rkc}^{(m)}$) span $W_{\psi_b}(\rho_c(\tau))$. Since
$t_b$ and $\diag(a,I_{(rk-1)c})$ commute in $\GL_{rkc}^{(m)}$ (by \eqref{eq:Nice GL $2$-cocycle on torus}),
$Z(s,\omega,\langle t_b,1\rangle\cdot W)=Z(s,\omega,W_b)$.

For $Z^*(s,\omega,\langle t_b,1\rangle\cdot W)$ we use the following observations.
\begin{enumerate}[leftmargin=*]
\item By \eqref{eq:sigma conjugate v by h} (with $N_{\GL_{rkc}}^-$ instead of $N_{\GL_{rkc}}$), if $[v_b]={}^{{}^wt_b}[v]$, ${}^{{}^wt_b}\langle[v],\varsigma^-([v])\rangle=\langle[v_b],\varsigma^-([v_b])\rangle$.
\item The topological module of $[v]\mapsto{}^{{}^wt_b}[v]$ (see e.g., \cite[p.~444]{BZ2}) is $|b|^{((rk(rk-1)/2)-1)c^2}$.
\item The elements ${}^wt_b$ and $\diag(I_{(rk-1)c},a)$ commute in $\GL_{rkc}^{(m)}$.
\item ${}^wt_b\diag(I_{(rk-1)c},b^{-rk}I_c)=(b^{-1}I_{rkc})t_b$
and by Proposition~\ref{proposition:action of W on torus is trivial on Sp}, ${}^w\langle t_b,1\rangle=\langle{}^wt_b,1\rangle$.
\item\label{it:last step} When we change $a\mapsto (b^{-rk}I_c)a$, $\omega(\langle a,1\rangle)\mapsto
\sigma^{\diamondsuit}_{c}(b^{-rk}I_c,a)\pi^{\vee}(\langle b^rI_c,1\rangle)^{-k}\omega(\langle a,1\rangle)$, and the root
$\sigma^{\diamondsuit}_{c}(b^{-rk}I_c,a)$ is cancelled
by the similar change in $W$.
\item $d_{\psi_b}x=|b|^{1/2}d_{\psi}x$ ($dx=d_{\psi}x$, see \S~\ref{ints and RS integrals}),
hence the measure of $V_{((rk-2)c,c)}^-$ appearing in $Z^*(s,\omega,W_b)$ equals $|b|^{(rk-2)c^2/2}dv$.
\end{enumerate}
Note that for \eqref{it:last step} it is crucial that $b^{rk}I_c\in C_{r,c}$. Also $s'rkc=kc(s-1/2)+rkc/2$.
We obtain
\begin{align*}
&Z^*(s,\omega,\langle t_b,1\rangle\cdot W)=\pi(\langle b^rI_c,1\rangle)^{k}\eta_{\tau,c}(\langle b^{-1},1\rangle)|b|^{-kc(s-1/2)}Z^*(s,\omega,W_b).
\end{align*}
This computation together with Lemma~\ref{lemma:eta tau c 1} imply \eqref{eq:RS dependence on psi}.

\subsubsection{The functional equation \eqref{eq:RS functional equation}}\label{RS The functional equation}
By \cite[Proposition~77]{me12}, $\eta_{\tau}\eta_{\tau^*}=1$. Then since
\begin{align*}
\pi(\langle b^rI_c,1\rangle)^{-k}\eta_{\tau}(\langle b,1\rangle)^c|b|^{kc(s-1/2)}
\pi^*(\langle b^rI_c,1\rangle)^{-k}\eta_{\tau^*}(\langle b,1\rangle)^c|b|^{kc(1-s-1/2)}=1,
\end{align*}
equality \eqref{eq:RS functional equation} is independent of the choice of $\psi$. Hence we assume the conductor of $\psi$ is $0$.
Let $\mathcal{S}(\Mat_c)$ be the space of Schwartz functions on $\Mat_c$,
and define the Fourier transform of $\phi\in\mathcal{S}(\Mat_c)$ by $\widehat{\phi}(y)=\int_{\Mat_c}\phi(x)\psi^{-1}(\tr xy)dx$ where
$dx=\prod_{i,j}d_{\psi}x$. Let
\begin{align*}Y=\left\{\jmath(y)=\left(\begin{smallmatrix}
                I_c  \\
                y & I_{(rk-2)c} \\
                0 & 0 & I_c
              \end{smallmatrix}\right)\right\}<N_{\GL_{rkc}}^-.
\end{align*}
For each $1\leq i\leq rk-2$, embed $\Mat_{ic\times c}$ in $\Mat_{((rk-2)c)\times c}$ by $y\mapsto \left(\begin{smallmatrix}y\\0\end{smallmatrix}\right)$ ($0\in\Mat_{((rk-2-i)c)\times c}$), thereby
$\Mat_{ci\times i}<Y$. For $x\in \Mat_c$, put
\begin{align*}
\ell_i(x)=\left(\begin{smallmatrix}
                I_c & 0& -x \\
                 & I_{ic} \\
                 &  & I_c\\
                 &  & & I_{(rk-2-i)c}
              \end{smallmatrix}\right).
\end{align*}
For a given $W\in W_{\psi}(\rho_c(\tau))$, let $l>0$ be such that $W$ is right-invariant under the subgroup
\begin{align*}
\{\langle \jmath(y),\varsigma^-(\jmath(y))\rangle:y\in Y,\,||y||_{\infty}=\max_{i,j}|y_{i,j}|\leq q^{-l}\}.
\end{align*}
Such an $l$ always exists (see \cite[p.~321]{BJ}, a different $1$-cochain will still coincide with $\varsigma^-$ for $l\gg0$).
Let $\phi_i\in\mathcal{S}(\Mat_c)$ be such that $\widehat{\phi}_i$ is the characteristic function of $\Mat_c(\varpi^l\mathcal{O})$ and set
\begin{align*}
\phi_i W(h)=\int\limits_{\Mat_c}W(h\langle \ell_i(x),1\rangle)\phi_i(x)\,dx,\qquad h\in\GL_{rkc}^{(m)}.
\end{align*}
Let $W_0\in W_{\psi}(\rho_c(\tau))$ be such that $Z(s,\omega,W_0)$ is not identically $0$ and define for $i=1,\ldots,rk-2$,
$W_i=\phi_iW_{i-1}$, i.e., for each $i$ there is a suitable $l_i>0$, $\widehat{\phi}_i$ is the characteristic function of
$\Mat_c(\varpi^{l_i}\mathcal{O})$, etc. Put $W=\langle w,1\rangle\cdot W_{rk-2}$. Using the fact that
if $[v]^*=\jmath(y)$, ${}^*\langle[v],\varsigma^-([v])\rangle=
\langle\jmath(y),\varsigma^-(\jmath(y))\rangle$ (see e.g., \cite[(1.4)]{me12}), applying \eqref{eq:involution b*0} and changing $a\mapsto a^*$,
\begin{align*}
&Z^*(1-s,\omega^*,W^*)
\\&=\int\limits_{\GL_c}\int\limits_{\Mat_{((rk-2)c)\times c}}
W_{rk-2}(\langle e_1(a),1\rangle\langle\jmath(y),\varsigma^-(\jmath(y))\rangle)\omega(\langle a,1\rangle)|\det a|^{s'-(rkc-2c+1)/2}\,dy\,da.
\end{align*}
Consider $1\leq i\leq rk-2$. For $y\in\Mat_{ic\times c}$, since $\ell_i(x)\in N_{\GL_{rkc}}$, by \eqref{eq:sigma on h and v} we have
\begin{align}\label{eq:conjugations W_i using ell_i}
{}^{\ell_i(x)^{-1}}\langle\jmath(y),\varsigma^-(\jmath(y))\rangle
=\langle\left(\begin{smallmatrix}I_c\\y&I_{ic}&-yx\\&&I_c\\&&&I_{(rk-i-2)c}\end{smallmatrix}\right),\varsigma^-(\jmath(y))\rangle
=\langle\left(\begin{smallmatrix}I_c\\&I_{ic}&-yx\\&&I_c\\&&&I_{(rk-i-2)c}\end{smallmatrix}\right),1\rangle
\langle\jmath(y),\varsigma^-(\jmath(y))\rangle.
\end{align}
The $(rk,c)$ character takes the value $\psi^{-1}(\tr(xy_i))$ on the upper triangular matrix on the r.h.s., where $y_i$ is the bottom $c\times c$ block of $y$. Thus
\begin{align*}
&\int\limits_{\Mat_{ic\times c}}
W_{i}(\langle e_1(a),1\rangle\langle\jmath(y),\varsigma^-(\jmath(y))\rangle)\,dy
\\&=\int\limits_{\Mat_{ic\times c}}\int\limits_{\Mat_c}
W_{i-1}(\langle e_1(a),1\rangle\langle\jmath(y),\varsigma^-(\jmath(y))\rangle\langle\ell_i(x),1\rangle)\phi_i(x)\,dx\,dy
\\&=\int\limits_{\Mat_{ic\times c}}
W_{i-1}(\langle e_1(a),1\rangle\langle\jmath(y),\varsigma^-(\jmath(y))\rangle)\widehat{\phi}_i(y_i)\,dy
\\&=q^{-l_ic^2}\int\limits_{\Mat_{((i-1)c)\times c}}
W_{i-1}(\langle e_1(a),1\rangle\langle\jmath(y),\varsigma^-(\jmath(y))\rangle)\,dy.
\end{align*}
Here $q^{-l_i}$ is the volume of $\varpi^{l_i}\mathcal{O}$. Applying this repeatedly for $i=rk-2,\ldots,1$, we obtain
\begin{align}\label{eq: Z^* with W^*}
&Z^*(1-s,\omega^*,W^*)\\&=q^{-c^2\sum_{i=1}^{rk-2}l_i}\int\limits_{\GL_c}
W_{0}(\langle e_1(a),1\rangle)\omega(\langle a,1\rangle)|\det a|^{s'-(rkc-2c+1)/2}\,da=q^{-c^2\sum_{i=1}^{rk-2}l_i}Z(s,\omega,W_0).\nonumber
\end{align}
Next we compute $Z(1-s,\omega^*,W^*)$. First we have
\begin{align}\label{eq:Z 1-s W^*}
&Z(1-s,\omega^*,W^*)\\&=\int\limits_{\GL_c}\nonumber
W_{rk-2}^*(\langle e_1(a),1\rangle\,{}^*\langle w,1\rangle)\,\omega^*(\langle a,1\rangle)|\det a|^{(1-s)'-(rkc-2c+1)/2}\,da
\\&=\int\limits_{\GL_c}
W_{rk-2}(\langle{}^we_1(a),1\rangle\langle w,1\rangle)\omega(\langle a,1\rangle)|\det a|^{s'-1+(rkc-2c+1)/2}\,da.\nonumber
\end{align}
Consider a fixed $a$ and let $1\leq i \leq rk-2$.
For $x\in\Mat_c$, ${}^w\ell_i(x)=[(0_{c\times((i-1)c)},-x,0_{c\times((rk-2-i)c)})]$ whence
\begin{align*}
{}^w\langle\ell_i(x),1\rangle=\langle[(0_{c\times((i-1)c)},-x,0_{c\times((rk-2-i)c)})],\varsigma^-([(0_{c\times((i-1)c)},-x,0_{c\times((rk-2-i)c)})])\rangle.
\end{align*}
(See \cite[(1.4)]{me12}.)
Let $y^i=[(0_{c\times ic},y_{i+1},\ldots,y_{rk-2})]$ where $y_{i+1},\ldots,y_{rk-2}\in\Mat_c$ and note that $y^{rk-2}$ is trivial and $y^0$ is a general element of $\Mat_{c\times((rk-2)c)}$.
We then have
\begin{align}\label{eq:W i to W i-1}
&W_{i}(\langle{}^we_1(a),1\rangle\langle y^i,\varsigma^-(y^i)\rangle\langle w,1\rangle)
\\&=\int\limits_{\Mat_c}W_{i-1}(\langle{}^we_1(a),1\rangle\langle y^{i-1},\varsigma^-(y^{i-1})\rangle\langle w,1\rangle)\phi_i(y_i)\,dy_i.\nonumber
\end{align}
Take $z\in \Mat_{((rk-2)c)\times c}$ whose $i$-th $c\times c$ block $z_i$ satisfies $||z_i||_{\infty}=q^{-l_i}$ and all other blocks are $0$. A computation using \eqref{eq:sigma on h and v} similar to \eqref{eq:conjugations W_i using ell_i} shows
\begin{align*}
&W_{i-1}(\langle{}^we_1(a),1\rangle\langle [y^{i-1}],\varsigma^-([y^{i-1}])\rangle\langle w,1\rangle)
\\&=W_{i-1}(\langle{}^we_1(a),1\rangle\langle [y^{i-1}],\varsigma^-([y^{i-1}])\rangle\langle w,1\rangle\langle \jmath(z),\varsigma^-(\jmath(z))\rangle)
\\&=\psi^{-1}(\tr z_iy_i)W_{i-1}(\langle{}^we_1(a),1\rangle\langle [y^{i-1}],\varsigma^-([y^{i-1}])\rangle\langle w,1\rangle).
\end{align*}
According to our assumption on $\psi$, the l.h.s.~ vanishes unless $||y_i||_{\infty}\leq q^{l_i}$ (if $||y_i||_{\infty}>q^{l_i}$, there is $z_i$ such that $\psi(\tr z_iy_i)\ne1$), and $\phi_i$ is the characteristic function of $\Mat_c(\varpi^{-l_i}\mathcal{O})$ multiplied by $q^{-l_ic^2}$. Thus we can replace $\phi_i(y_i)$ with $q^{-l_ic^2}$ on the r.h.s.~ of \eqref{eq:W i to W i-1}. Applying \eqref{eq:W i to W i-1} repeatedly for $i=rk-2,\ldots,1$ we obtain
\begin{align}\label{eq:removing unipotent int}
W_{rk-2}(\langle{}^we_1(a),1\rangle\langle w,1\rangle)
=q^{-c^2\sum_{i=1}^{rk-2}l_i}\int\limits_{V_{((rk-2)c,c)}^-}W_{0}(\langle{}^we_1(a),1\rangle\langle[x],\varsigma^-([x])\rangle\langle w,1\rangle)\,dx.
\end{align}
Here the measure $dx$ on the r.h.s.~ coincides with the measure $dv$ of $Z^*(s,\omega,W_0)$, by our choices of measures
for $\phi_i W$ for each $i$. Plugging this into \eqref{eq:Z 1-s W^*} we deduce
\begin{align}\label{eq: Z with W^*}
Z(1-s,\omega^*,W^*)=q^{-c^2\sum_{i=1}^{rk-2}l_i}Z^*(s,\omega,W_0).
\end{align}

By our choice of $W_0$, $Z(s,\omega,W_0)\not\equiv0$, and because $\gamma(s,\widetilde{\pi}\times\tau,\psi)$ is not identically zero, \eqref{eq:gamma RS} implies $Z^*(s,\omega,W_0)\not\equiv0$ whence by
\eqref{eq: Z with W^*}, $Z(1-s,\omega^*,W^*)\not\equiv0$. Now combining \eqref{eq: Z^* with W^*} and \eqref{eq: Z with W^*}, and
since $\pi(\mathfrak{i}_{\GL_c})^{rk-1}=\pi^*(\mathfrak{i}_{\GL_c})^{rk-1}$ and $\vartheta(1-s,\pi,\widetilde{\tau})=\vartheta(s,\widetilde{\pi},\tau)^{-1}$,
we deduce
\begin{align*}
\gamma(1-s,\pi\times\widetilde{\tau},\psi^{-1})=
\beta\frac{Z^*(1-s,\omega^*,W^*)}{Z(1-s,\omega^*,W^*)}=\beta\frac{Z(s,\omega,W_0)}{Z^*(s,\omega,W_0)}=\gamma(s,\widetilde{\pi}\times\tau,\psi)^{-1}.
\end{align*}
Here $\beta=\pi(\mathfrak{i}_{\GL_c})^{rk-1}\vartheta(1-s,\pi,\widetilde{\tau})$.
Note that indeed $W^*\in W_{\psi^{-1}}(\rho_c(\tau^*))$. This proves \eqref{eq:RS functional equation}.

\subsubsection{$F=\C$}\label{RS Archimedean property}
We can write $\pi\subset\Ind_{B_{\GL_c}}^{\GL_c}(\otimes_{j=1}^c\pi_j)$ and $\tau=\Ind_{B_{\GL_k}}^{\GL_k}(\otimes_{i=1}^k\tau_i)$,
for quasi-characters $\pi_j$ and $\tau_i$ of $F^*$. We prove
\begin{align}\label{eq:RS tilde of rho tau related to Artin gamma factors wish result}
\gamma(s,\widetilde{\pi}\times\tau,\psi)=\prod_{i,j}\gamma^{\mathrm{Tate}}(s,\tau_i^r\pi_j^{-r},\psi),
\end{align}
where $\gamma^{\mathrm{Tate}}$ is Tate's $\gamma$-factor (\cite{Tate}).
Granted that, the definitions of the archimedean $L$- and $\epsilon$-factors (see e.g., \cite[\S~3]{Shahidi1985}) imply \eqref{eq:RS Archimedean property}. By \eqref{eq:RS mult I} and the multiplicativity of $\gamma^{\mathrm{Tate}}$ it is enough to prove \eqref{eq:RS tilde of rho tau related to Artin gamma factors wish result} for $c=1$, henceforth assume this is the case. Now the integrals $Z$ and $Z^*$ are the
Rankin--Selberg $\GL_{rk}\times\GL_1$ integrals (\cite{JS3}) for $\rho_1(\tau)\times\pi^{\vee}$. Using \eqref{eq:gamma RS} and
\eqref{rep:rho c tau in general},
\begin{align}\label{eq:RS tilde of rho tau}
\gamma(s,\widetilde{\pi}\times\tau,\psi)&=\vartheta(s,\widetilde{\pi},\tau)\gamma^{\mathrm{RS}}(s',\rho_1(\tau)\times\pi^{-1},\psi)\\&=\nonumber
\prod_{i=1}^k(\pi\tau_i^{-1})(r^r)|r|^{-(s-1/2)}\gamma^{\mathrm{RS}}(s',\rho_1(\tau_i)\times\pi^{-1},\psi),
\end{align}
where $\gamma^{\mathrm{RS}}(s',\rho_1(\tau)\times\pi^{-1},\psi)$ is the Rankin--Selberg $\gamma$-factor of Jacquet and Shalika \cite{JS3}
for $\rho_1(\tau)\times\pi^{-1}$.
For each $i$, $\rho_1(\tau_i)=\tau_i\Ind_{B_{\GL_r}}^{\GL_r}(\delta_{B_{\GL_r}}^{1/(2r)})$ (the r.h.s.~ is irreducible) and
\begin{align}\label{eq:RS of rho tau}
\gamma^{\mathrm{RS}}(s',\rho_1(\tau_i)\times\pi^{-1},\psi)=\prod_{i=1}^r\gamma^{\mathrm{Tate}}(s'+(r-2i+1)/(2r),\tau_i\pi^{-1},\psi).
\end{align}
Over $\C$, $|\cdot|$ is the square of the ordinary absolute value $|\cdot|_{\R}$.
Identify $\tau_i$ with the pair $(l_{\tau_i},t_{\tau_i})\in\Z\times\C$ where $\tau_i(z)=(z/|z|_{\R})^{l_{\tau_i}}|z|^{t_{\tau_i}}$ and similarly identify $\pi^{-1}$ with $(l_{\pi^{-1}},t_{\pi^{-1}})\in\Z\times\C$. For brevity set $l_i=l_{\tau_i}+l_{\pi^{-1}}$ and
$t_i=t_{\tau_i}+t_{\pi^{-1}}$. By Langlands' definition of the $L$- and $\epsilon$-functions
(see \cite{Knapp1994}), assuming $\psi(z)=e^{2\pi i(z+\overline{z})}$
(which is possible since we already proved \eqref{eq:RS dependence on psi}),
\begin{align*}
L(s',\tau_i\pi^{-1})=2(2\pi)^{-(s'+t_{i}+|l_{i}|_{\R}/2)}\Gamma(s'+t_{i}+|l_{i}|_{\R}/2),\quad
\epsilon(s',\tau_i\pi^{-1},\psi)=i^{|l_{i}|_{\R}}.
\end{align*}
According to Tate's computation (\cite{Tate}),
\begin{align*}
\gamma^{\mathrm{Tate}}(s',\tau_i\pi^{-1},\psi)=\frac{\epsilon(s',\tau_i\pi^{-1},\psi)L(1-s',\tau_i^{-1}\pi)}{L(s',\tau_i\pi^{-1})}.
\end{align*}
Note that $l_{\tau_i^{-1}}=-l_{\tau_i}$ whence $|l_{\tau_i^{-1}}|_{\R}=|l_{\tau_i}|_{\R}$.
Using the multiplicativity formula
\begin{align*}
r^{r\beta-1/2}\prod_{i=0}^{r-1}\Gamma(\beta+i/r)=(2\pi)^{(r-1)/2}\Gamma(r\beta)
\end{align*}
with $\beta=s'+t_i+|l_i|_{\R}/2+(1-r)/(2r)$, the r.h.s.~ of \eqref{eq:RS of rho tau} equals
\begin{align*}
\frac{2(2r\pi)^{-(1-s-rt_i+|rl_i|_{\R}/2)}\Gamma(1-s-rt_i+|rl_i|_{\R}/2)}
{2(2r\pi)^{-(s+rt_i+|rl_i|_{\R}/2)}\Gamma(s+rt_i+|rl_i|_{\R}/2)}=|r|^{s-1/2}(\pi^{-1}\tau_i)(r^r)\gamma^{\mathrm{Tate}}(s,\tau_i^r\pi^{-r},\psi).
\end{align*}
This together with \eqref{eq:RS tilde of rho tau} and \eqref{eq:RS of rho tau} imply \eqref{eq:RS tilde of rho tau related to Artin gamma factors wish result} and we conclude \eqref{eq:RS Archimedean property}.
\subsubsection{Crude functional equation}\label{RS Crude functional equation}
First we adapt the global construction from \cite{G7} (the linear case) to define the global integral. We then present the global functional equation (which was not discussed in \cite{G7}). We use the results of \S~\ref{Global coverings} on the global covering.
Let $\xi$ be an automorphic form in the space of $\mathcal{E}_{\tau}$ (see \S~\ref{the representation rho_c(tau)}).
Consider the Fourier coefficient
\begin{align*}
\xi_1(b)=\xi^{V_{(2c,c^{rk-2})},\psi}(b)=\int\limits_{V_{(2c,c^{rk-2})}(F)\backslash V_{(2c,c^{rk-2})}(\A)}\xi(\langle v,(\eta_{rkc}^{\diamondsuit})^{-1}(v)\rangle b)\psi^{-1}(v)\,dv,
\end{align*}
where $\psi$ is the restriction of the $(rk,c)$ character to $V_{(2c,c^{rk-2})}$ and
$b\in\GL_{rkc}^{(m)}(\A)$. The normalizer of $V_{(2c,c^{rk-2})}$ and stabilizer of $\psi|_{V_{(2c,c^{rk-2})}}$ in $M_{(2c,c^{rk-2})}$ is given by $P=M\ltimes V_{(c^2)}$ where $M=\{e_1(g_1)e_2(g_2):g_1,g_2\in\GL_c\}\cong M_{(c^2)}$. The coefficient
$\xi^{V_{(2c,c^{rk-2})},\psi}$ is then an automorphic form on $\widetilde{P}(\A)$ with respect to the splitting of $P(F)$ obtained by restricting $(\eta_{rkc}^{\diamondsuit})^{-1}$.

The group $M(F)$ acts on the set of characters of $V_{(c^2)}(F)\backslash V_{(c^2)}(\A)$ with $c+1$ orbits. We can take representatives $\psi^j(v)=\psi(\tr(\left(\begin{smallmatrix}I_j &  \\& 0\end{smallmatrix}\right)v))$, $0\leq j\leq c$. Denote the stabilizer of $\psi^j$ in $M$ by $\mathrm{St}_j$, e.g., $\mathrm{St}_c=\GL_c^{\triangle}$, and $\psi^c\cdot(\psi|_{V_{(2c,c^{rk-2})}})$ is the $(rk,c)$ character of $V_{(c^{rk})}$. Define for $b\in \widetilde{P}(\A)$,
\begin{align*}
\xi_1^{V_{(c^2)},\psi^j}(b)=
\int\limits_{V_{(c^2)}(F)\backslash V_{(c^2)}(\A)}\xi_1(\langle v,(\eta_{rkc}^{\diamondsuit})^{-1}(v)\rangle b)
(\psi^j)^{-1}(v)\,dv.
\end{align*}
The Fourier expansion of $\xi_1$ along $V_{(c^2)}$ is given by
\begin{align}\label{eq:varphi Fourier expansion}
\xi_1(b)=&\sum_{j=0}^c
\quad\sum\limits_{y\in \mathrm{St}_j(F)\backslash M_{(c^2)}(F)}
\xi_1^{V_{(c^2)},\psi^j}(\langle y,(\eta_{rkc}^{\diamondsuit})^{-1}(y)\rangle b).
\end{align}
Let $\Lambda$ be a global $(rk,c)$ functional on $\mathcal{E}_{\tau}$ and put
$\Lambda_{\xi}(h)=\Lambda(\mathcal{E}_{\tau}(h)\xi)$ ($h\in\GL_{rkc}^{(m)}(\A)$).
Since $\xi_1^{V_{(c^2)},\psi^0}=\xi_1^{V_{(c^2)}}$ is the constant term, we can rewrite \eqref{eq:varphi Fourier expansion} in the form
\begin{align}\label{eq:varphi Whittaker expansion}
(\xi_1-\xi_1^{V_{(c^2)}})(b)=
&\sum_{j=1}^{c-1}\quad
\sum\limits_{y\in \mathrm{St}_j(F)\backslash M_{(c^2)}(F)}
\xi_1^{V_{(c^2)},\psi^j}(\langle y,(\eta_{rkc}^{\diamondsuit})^{-1}(y)\rangle b)
\\&+\sum\limits_{y\in \GL_c(F)}\Lambda_{\xi}(\langle e_1(y),(\eta_{rkc}^{\diamondsuit})^{-1}(e_1(y))\rangle b).\nonumber
\end{align}
The above description applies in the same manner to $\xi_2=\xi^{V_{(c^{rk-2},2c)},\psi}$, which is an automorphic form on $P'={}^wM\ltimes V_{(c^2)}$ where $V_{(c^2)}$ is identified with $\diag(I_{(rk-2)c},V_{(c^2)})$.
Then \eqref{eq:varphi Whittaker expansion} applies to $\xi_2-\xi_2^{V_{(c^2)}}$ as well, except we have ${}^we_1(y)$ instead of $e_1(y)$ in the last summation.

By virtue of the ``exchange of roots" \cite[Lemma~7.1]{RGS},
\begin{align}\label{eq:X of roots}
\xi_1(b)&=\int\limits_{V_{((rk-2)c,c)}^-(\A)}\xi_2(\langle [v],(\eta_{rkc}^{\diamondsuit,-})^{-1}([v])\rangle
\langle w,(\eta_{rkc}^{\diamondsuit})^{-1}(w)\rangle b)\,dv.
\end{align}
Here $u^-\mapsto\langle u^-,(\eta_{rkc}^{\diamondsuit,-})^{-1}(u^-)\rangle$ is the splitting of $N_{\GL_{rkc}}^-(\A)$ in $\GL_{rkc}(\A)^{(m)}$.
The convergence of the r.h.s.~ of \eqref{eq:X of roots} is in the sense of \cite[Lemma~7.1]{RGS}.

Next we describe the setup for the integral, globalizing \eqref{eq:RS lift of embedding}. Define
\begin{align}\label{eq:RS rho L and rho R}
\rho_{L}(b,b')=\rho_{rkc}^{\diamondsuit}(e_1(b),e_1(b')),\qquad
\rho_{R}(b,b')=(\rho_{rkc}^{\diamondsuit})^{-1}(e_2(b),e_2(b')).
\end{align}
Note that $e_1(\GL_c(\A))$ and $e_2(\GL_c(\A))$ commute in $\GL_{rkc}(\A)^{(m)}$, because they commute locally.
The mappings $\langle b,\epsilon\rangle\mapsto\langle \mathfrak{e}_1(b),\epsilon\rangle$ and
$\langle b,\epsilon\rangle\mapsto\langle \mathfrak{e}_2(b),\epsilon^{-1}\rangle$ (cf. \eqref{eq:RS embeddings}) define an embedding
\begin{align}\label{eq:RS gbl lift of embedding}
\{(\epsilon_1,\epsilon_2)\in\mu_m^2:\epsilon_1=\epsilon_2\}\backslash \GL_c(\A)[\rho_L]\times \GL_c(\A)[\rho_R] \rightarrow \GL_{kc}^{(m)}(\A).
\end{align}
(Cf. \eqref{eq:RS lift of embedding}.) Recall from \S~\ref{Global coverings} that automorphic forms on $\GL_{d}^{(m)}(\A)$ are defined with respect to the splitting of $\GL_{d}(F)$ given by $y\mapsto\langle y,\eta_{d}^{\diamondsuit}(y)^{-1}\rangle$, we use this here for $d=rkc,c$; automorphic forms on $\GL_c(\A)[\rho_L]$ and $\GL_c(\A)[\rho_R]$ are defined by the restrictions of $y\mapsto\langle y,\eta_{rkc}^{\diamondsuit}(e_1(y))^{-1}\rangle$ and
$y\mapsto\langle y,\eta_{rkc}^{\diamondsuit}(e_2(y))\rangle$ (resp.). By \eqref{Appendixeq:rho beta and rho square globally}
and \eqref{eq:RS rho top rk on Levi},
$\rho_L=\rho_c^{\diamondsuit}$ and $\rho_R=(\rho_c^{\diamondsuit})^{1-rk}=\rho_c^{\diamondsuit}\cdot(\det ,\det)_{m/r}^{k}$ in
$\mathrm{H}^2(\GL_d(\A),\mu_m)$.

We turn to define the integral. Note that both $\pi$ and $\pi^{\vee}$ are representations of
$\GL_{c}(\A)^{(m)}=\GL_{c}(\A)[\rho_c^{\diamondsuit}]$. Let $\varphi_1$ (resp., $\varphi_2$) be a cusp form in the space of $\pi^{\vee}$ (resp., $\pi$). Note that $\varphi_1$ is anti-genuine. Let $\eta_c^{\diamondsuit,(rk)}$ be defined as in \S~\ref{Global coverings},
and let $\eta_{\tau}'$ and $\gamma_{\psi'}$ be the functions appearing in \eqref{eq:aut character eta identity} where if $m|rk$, we ignore
all occurrences of $\gamma_{\psi'}$ in this section. Define
\begin{align*}
\varphi_1^L(\langle b,\epsilon\rangle)=\varphi_1(\langle b,\eta_{(c^{rk})}^{-1}(e_1(b))\epsilon\rangle),\qquad
\varphi_2^R(\langle b,\epsilon\rangle)=\gamma_{\psi'}^{-1}(\det b)\varphi_2(\langle b,\eta_{(c^{rk})}(e_2(b))
\eta_c^{\diamondsuit,(rk)}(b)\epsilon\rangle).
\end{align*}
The function $\varphi_1^L$ (resp., $\varphi_2^R$) is a cusp form in the space of $\pi^{\vee}$ (resp., $\pi$) as a representation of $\GL_{c}(\A)[\rho_L]$ (resp., $\GL_{c}(\A)[\rho_R]$). For example if $y\in \GL_c(F)$ and $\epsilon_y=\eta_{rkc}^{\diamondsuit}(e_2(y))$,
\begin{align*}
\varphi_2^R&(\langle y,\epsilon_y\rangle\langle b,1\rangle)=
\gamma_{\psi'}^{-1}(\det yb)\varphi_2(\langle yb,\rho_R(e_2(y),e_2(b))\eta_{(c^{rk})}(e_2(yb))
\eta_c^{\diamondsuit,(rk)}(yb)\epsilon_y\rangle)
\\&=\gamma_{\psi'}^{-1}(\det b)(\det y,\det b)_{m/r}^k
\varphi_2(\langle yb,(\rho_c^{\diamondsuit})^{1-rk}(y,b)\eta_{(c^{rk})}(e_2(y))\eta_{(c^{rk})}(e_2(b))
\eta_c^{\diamondsuit,(rk)}(yb)\epsilon_y\rangle)
\\&=\gamma_{\psi'}^{-1}(\det b)
\varphi_2(\langle yb,(\rho_c^{\diamondsuit})(y,b)\eta_{(c^{rk})}(e_2(y))\eta_{(c^{rk})}(e_2(b))
\eta_c^{\diamondsuit,(rk)}(y)\eta_c^{\diamondsuit,(rk)}(b)\epsilon_y\rangle)
\\&=\gamma_{\psi'}^{-1}(\det b)
\varphi_2(\langle y,\epsilon_y\eta_{(c^{rk})}(e_2(y))\eta_c^{\diamondsuit,(rk)}(y)\rangle\langle b,\eta_{(c^{rk})}(e_2(b))\eta_c^{\diamondsuit,(rk)}(b)\rangle)
\\&=\varphi_2^R(\langle b,\eta_{(c^{rk})}(e_2(b))\eta_c^{\diamondsuit,(rk)}(b)\rangle).
\end{align*}
Here for the second equality we used \eqref{eq:RS rho L and rho R}, \eqref{Appendixeq:rho beta and rho square globally} and the fact that $\gamma_{\psi'}^{-1}$ is trivial on $F^*$; for the third equality we used \eqref{eq:RS rho top rk on Levi}; For the last, first note that
$\epsilon_y\eta_{(c^{rk})}(e_2(y))=(\eta_c^{\diamondsuit})^{rk-1}(y)$, because \eqref{eq:eta beta nu} globalizes to $y\in\GL_c(F)$ (see \S~\ref{Global coverings}). Second, by Lemma~\ref{lemma:eta 2c rk and eta 2c (rk)}, $(\eta_{c}^{\diamondsuit})^{rk}(y)\eta_{c}^{\diamondsuit,(rk)}(y)=1$ thus
$\eta_c^{\diamondsuit}(y)^{rk-1}\eta_c^{\diamondsuit,(rk)}(y)=\eta_c^{\diamondsuit}(y)^{-1}$. Third, $\varphi_2$ is an automorphic form.
For $\varphi_1^L$ we show
$\varphi_1^L(\langle y,\epsilon_y^{-1}\rangle\langle b,1\rangle)=\varphi_1^L(\langle b,1\rangle)$ where
$\epsilon_y=\eta_{rkc}^{\diamondsuit}(e_1(y))$, the computation is similar to the above but simpler since we do not need
to handle $\eta_c^{\diamondsuit,(rk)}$ and $\gamma_{\psi'}^{-1}$.

Choose a finite set of places $S'$ such that $\varphi_1$ and $\varphi_2$ are right-invariant on
$\{\langle y,1\rangle:y\in K_{\GL_c,\nu}\}$ and $\xi$ is right-invariant on
$\{\langle y,1\rangle:y\in K_{\GL_{rkc},\nu}\}$ for all $\nu\notin S'$ (see \cite[p.~111--112]{me12}), and let
$C_d=\{xI_d:x\in F^*\A^{*r}\prod_{\nu\notin S'}\mathcal{O}_{\nu}^*\}$. Then $C_{r,d}(\A)<C_d<C_{\GL_d}(\A)$ and
$[C_{\GL_d}(\A):C_d]<\infty$.

The global integral is given by
\begin{align}\label{eq:RS global 1}
Z(s,\varphi_1,\varphi_2,\xi)=&\int\limits_{(C_{2c}M_{(c^2)}(F))\backslash M_{(c^2)}(\A)}(\xi_1-\xi_1^{V_{(c^2)}})(\langle e_1(a),1\rangle
\langle e_2(b),1\rangle)
\\&\quad\varphi_1^L(\langle a,1\rangle)
\varphi_2^R(\langle b,1\rangle)\eta_{\tau}'(\det b^{-1})|\det ab^{-1}|^{s'-(rkc-2c+1)/2}\,da\,db.\nonumber
\end{align}
First we check that \eqref{eq:RS global 1} is formally well defined. By \eqref{eq:RS gbl lift of embedding}, the integrand is indeed a function on $(\GL_c(F)\times \GL_c(F))\backslash(\GL_c(\A)\times \GL_c(\A))$ (left-invariant under the rational points and invariant of the section), and right-invariant on $\GL_c(\A)\times \GL_c(\A)$
(see \cite[Proposition~A.1]{me12} and its proof). Regarding the invariance under $C_{2c}$ it suffices to show that the integrand is left-invariant under the change of variables $a\mapsto za$ and $b\mapsto zb$ where $z\in C_{r,c}(\A)$. In this case by
\eqref{Appendixeq:rho beta and rho square globally} and \eqref{eq:RS rho top rk on Levi} and since
$\varphi_i(\langle z,1\rangle\langle b,1\rangle)=
\varphi_i(\langle z,1\rangle)\varphi_i(\langle b,1\rangle)$ for both $i$, we have
\begin{align*}
&\varphi_1^L(\langle za,1\rangle)=\rho_{rkc}^{\diamondsuit}(e_1(z),e_1(a))\varphi_1^L(\langle z,1\rangle)\varphi_1^L(\langle a,1\rangle),\\
&\varphi_2^R(\langle zb,1\rangle)=\rho_{rkc}^{\diamondsuit}(e_2(z),e_2(b))\varphi_2^R(\langle z,1\rangle)\varphi_2^R(\langle b,1\rangle)\nonumber
\end{align*}
and
\begin{align*}
&(\xi_1-\xi_1^{V_{(c^2)}})(\langle e_1(za),1\rangle
\langle e_2(zb),1\rangle)\\&=
\prod_{i=1}^2\rho_{rkc}^{\diamondsuit}(e_i(z),e_i(b))^{-1}
(\xi_1-\xi_1^{V_{(c^2)}})(\langle z^{\triangle},\rho_{rkc}^{\diamondsuit}(e_1(z),e_2(z))\rangle\langle e_1(a),1\rangle
\langle e_2(b),1\rangle).\nonumber
\end{align*}
Combining these identities and because $\varphi_1(\langle z,1\rangle)\varphi_2(\langle z,1\rangle)=1$,
\begin{align*}
&(\xi_1-\xi_1^{V_{(c^2)}})(\langle e_1(za),1\rangle\langle e_2(zb),1\rangle)
\varphi_1^L(\langle za,1\rangle)\varphi_2^R(\langle zb,1\rangle)
\\&=(\xi_1-\xi_1^{V_{(c^2)}})(\langle z^{\triangle},1\rangle \langle e_1(a),1\rangle\langle e_2(b),1\rangle)
\varphi_1^L(\langle a,1\rangle)\varphi_2^R(\langle b,1\rangle)
\\&\quad\times\eta_{(c^{rk})}(e_1(z))\eta_{(c^{rk})}(e_2(z))\rho_{rkc}^{\diamondsuit}(e_1(z),e_2(z))\eta_c^{\triangle,(rk)}(z)\gamma_{\psi'}^{-1}(\det z)
\\&=(\xi_1-\xi_1^{V_{(c^2)}})(\langle z^{\triangle},\eta_{(c^{rk})}(z^{\triangle})\eta_c^{\triangle,(rk)}(z)\rangle \langle e_1(a),1\rangle\langle e_2(b),1\rangle)
\varphi_1^L(\langle a,1\rangle)\varphi_2^R(\langle b,1\rangle)
\gamma_{\psi'}^{-1}(\det z)
\\&=(\xi_1-\xi_1^{V_{(c^2)}})(\langle e_1(a),1\rangle\langle e_2(b),1\rangle)
\varphi_1^L(\langle a,1\rangle)\varphi_2^R(\langle b,1\rangle)\eta_{\tau}'(\det z).
\end{align*}
Here for the second equality we used \eqref{Appendixeq:rho beta and rho square globally} and for the last \eqref{eq:central aut character eta identity}.
The factor $\eta_{\tau}'(\det z)$ cancels with $\eta_{\tau}'(\det z^{-1})$ obtained in \eqref{eq:RS global 1}
by the change to $b$.
This completes the verification.

Integral~\eqref{eq:RS global 1} is absolutely convergent for $\Real(s)\gg0$. This can be proved, e.g., by repeating the arguments
below and reaching \eqref{eq:RS global 1 decompose 3} with the integrand replaced by its absolute value.
In this domain, when we plug \eqref{eq:varphi Whittaker expansion} into the integral we obtain
\begin{align}\label{eq:RS global 1 decompose 1}
&\sum_{j=1}^c\int\limits_{(C_{2c}\mathrm{St}_j(F))\backslash M_{(c^2)}(\A)}
\xi_1^{V_{(c^2)},\psi^j}(\langle e_1(a),1\rangle
\langle e_2(b),1\rangle)
\\&\quad\varphi_1^L(\langle a,1\rangle)
\varphi_2^R(\langle b,1\rangle)\eta_{\tau}'(\det b^{-1})|\det ab^{-1}|^{s'-(rkc-2c+1)/2}\,da\,db.\nonumber
\end{align}
For $j<c$ the integral vanishes because there is a unipotent subgroup of $e_1(G)$ under which the Fourier coefficient is invariant
and $\varphi_1^L$ is a cusp form. The proof follows from the local statement in the proof of \cite[Theorem~C.1]{CFK} (the global vanishing properties of $\mathcal{E}_{\tau}$ follow from the local ones). Thus \eqref{eq:RS global 1 decompose 1} becomes
\begin{align}\label{eq:RS global 1 decompose 2}
&\int\limits_{(C_{2c}\GL_c(F))\backslash M_{(c^2)}(\A)}
\Lambda_{\xi}(\langle e_1(a),1\rangle \langle e_2(b),1\rangle)
\\&\quad\varphi_1^L(\langle a,1\rangle)
\varphi_2^R(\langle b,1\rangle)\eta_{\tau}'(\det b^{-1})|\det ab^{-1}|^{s'-(rkc-2c+1)/2}\,da\,db.\nonumber
\end{align}
Here $\GL_c$ is diagonally embedded in $M_{(c^2)}$.
Now we change variables $a\mapsto ba$. By \eqref{Appendixeq:rho beta and rho square globally},
\begin{align*}
&\varphi_1^L(\langle ba,1\rangle)=
\rho_{rkc}^{\diamondsuit}(e_1(b),e_1(a))\eta_{(c^{rk})}(e_1(b))
\varphi_1(\langle b,1\rangle\langle a,\eta_{(c^{rk})}^{-1}(e_1(a))\rangle),\\
&\varphi_2^R(\langle b,1\rangle)=\gamma_{\psi'}^{-1}(\det b)\eta_{(c^{rk})}(e_2(b))\eta_c^{\diamondsuit,(rk)}(b)\varphi_2(\langle b,1
\rangle),\\
&\Lambda_{\xi}(\langle e_1(ba),1\rangle \langle e_2(b),1\rangle)
=\rho_{rkc}^{\diamondsuit}(e_1(b),e_1(a))^{-1}\Lambda_{\xi}(\langle b^{\triangle},
\rho_{rkc}^{\diamondsuit}(e_1(b),e_2(b))\rangle\langle e_1(a),1\rangle).
\end{align*}
Then as with $z$ above, using \eqref{Appendixeq:rho beta and rho square globally}, \eqref{eq:global eta triangle} and \eqref{eq:aut character eta identity} we deduce
\begin{align*}
&\Lambda_{\xi}(\langle e_1(ba),1\rangle \langle e_2(b),1\rangle)\varphi_1^L(\langle ba,1\rangle)\varphi_2^R(\langle b,1\rangle)
\\&=\eta_{\tau'}(\det b)\Lambda_{\xi}(\langle e_1(a),1\rangle)\varphi_1(\langle b,1\rangle\langle a,\eta_{(c^{rk})}^{-1}(e_1(a))\rangle)\varphi_2(\langle b,1\rangle).
\end{align*}
Hence the $db$-integral forms a matrix coefficient $\omega$ on $\pi^{\vee}$ given by
\begin{align*}
\omega(a_0)=\int\limits_{C_c\GL_c(F)\backslash\GL_c(\A)}\varphi_1(\langle b,1\rangle a_0)\varphi_2(\langle b,1\rangle)\,db,\qquad a_0\in\GL_c^{(m)}(\A).
\end{align*}
If we then set $\omega^L(\langle a,\epsilon\rangle)=\omega(\langle a,\eta_{(c^{rk})}^{-1}(e_1(a))\epsilon\rangle)$ and apply Fubini's Theorem,
\eqref{eq:RS global 1 decompose 2} becomes
\begin{align}\label{eq:RS global 1 decompose 3}
&\int\limits_{\GL_c(\A)}
\Lambda_{\xi}(\langle e_1(a)\rangle)\omega^L(\langle a,1\rangle)|\det a|^{s'-(rkc-2c+1)/2}\,da.
\end{align}
Thus for decomposable data, by \eqref{eq:factorizable rk c functional}, $Z(s,\varphi_1,\varphi_2,\xi)=\prod_{\nu}Z(s,\omega_{\nu},W_{\nu})$.

We turn to the second global integral for the definition of the functional equation. For this integral we change the
definitions of $\rho_L$ and $\rho_R$ in \eqref{eq:RS rho L and rho R} and of $\varphi_1^L$ and $\varphi_2^R$, so that
$e_1(b)$ and $e_2(b)$ are replaced with ${}^we_1(b)$ and ${}^we_2(b)$. Now \eqref{eq:RS gbl lift of embedding} holds with these modifications. Denote the new versions of $\varphi_1^L$ and $\varphi_2^R$ by $\varphi_1^{{}^wL}$ and $\varphi_2^{{}^wR}$. Define
\begin{align}\label{eq:RS global 2}
Z^*(s,\varphi_1,\varphi_2,\xi)&=\int\limits_{(C_{2c}M_{(c^2)}(F))\backslash M_{(c^2)}(\A)}
\int\limits_{V_{((rk-2)c,c)}^-(\A)}(\xi_2-\xi_2^{V_{(c^2)}})
\\&(\langle {}^we_1(a),1\rangle\nonumber
\langle {}^we_1(b),1\rangle
\langle [v],(\eta_{rkc}^{\diamondsuit,-})^{-1}([v])\rangle
\langle w,(\eta_{rkc}^{\diamondsuit})^{-1}(w)\rangle)\\&\quad
\varphi_1^{{}^wL}(\langle a,1\rangle)\varphi_2^{{}^wL}(\langle b,1\rangle)\eta_{\tau}'(\det b^{-1})|\det ab^{-1}|^{s'-1+(rkc-2c+1)/2}\,dv\,da\,db.\nonumber
\end{align}
To show this integral is formally well defined we argue as above. It is absolutely convergent
for $\Real(s)\ll0$, where we can unfold it and for decomposable data
$Z^*(s,\varphi_1,\varphi_2,\xi)=\prod_{\nu}Z^*(s,\omega_{\nu},W_{\nu})$.

To obtain \eqref{eq:RS global property} we adapt the classical argument of \cite{Tate} (see also \cite{JS2,CPS,G7}).
Denote $M_{(c^2)}(\A)^1=\{g\in M_{(c^2)}(\A):\delta_{P_{(c^2)}}(g)=1\}$ and identify $\R_{>0}$ with the subgroup of ideles whose finite components are $1$ and infinite components are all equal and positive. We rewrite \eqref{eq:RS global 1} using
$M_{(c^2)}(\A)=M_{(c^2)}(\A)^1\times \{\diag(tI_c,I_c):t\in\R_{>0}\}$ and the formula $dadb=t^{-1}d(a,b)dt$, where $d(a,b)$ is a measure on
$M_{(c^2)}(\A)^1$ and $dt$ is the standard Lebesgue measure of $\R$. Then
\begin{align}\label{eq:RS global for functional equation}
Z(s,\varphi_1,\varphi_2,\xi)=&\int_0^{\infty}\int\limits_{(C_{2c}M_{(c^2)}(F))\backslash M_{(c^2)}(\A)^1}(\xi_1-\xi_1^{V_{(c^2)}})(\langle e_1(ta),1\rangle
\langle e_2(b),1\rangle)
\\&\quad\varphi_1^L(\langle ta,1\rangle)
\varphi_2^R(\langle b,1\rangle)\eta_{\tau}'(\det b^{-1})t^{c(s'-(rkc-2c+1)/2)-1}\,d(a,b)\,dt.\nonumber
\end{align}

Write \eqref{eq:RS global for functional equation} as a sum of integrals $\int_1^{\infty}+\int_0^1$. For the first
summand we utilize \eqref{eq:varphi Whittaker expansion} and apply the unfolding argument ($M_{(c^2)}(\A)^1$ contains the unipotent subgroups of $M_{(c^2)}(\A)$). Write the second summand as the difference $\int_0^1(\xi_1-\xi_1^{V_{(c^2)}})\ldots\,dt=\int_0^1\xi_1\ldots\,dt-\int_0^1\xi_1^{V_{(c^2)}}\ldots\,dt$. Next, applying
\eqref{eq:X of roots} to $\xi_1$ we obtain $\xi_2$ with the additional inner integral $dv$, then we write
$\int_0^1\xi_2=\int_0^1(\xi_2-\xi_2^{V_{(c^2)}})+\int_0^1\xi_2^{V_{(c^2)}}$. In total
$Z(s,\varphi_1,\varphi_2,\xi)$ becomes a sum of $4$ integrals:
\begin{align}\label{eq:RS global for functional equation summand 1}
&\int_1^{\infty}\int\limits_{(C_{2c}M_{(c^2)}(F))\backslash M_{(c^2)}(\A)^1}\Lambda_{\xi}(\langle e_1(ta),1\rangle
\langle e_2(b),1\rangle)
\\&\quad\varphi_1^L(\langle ta,1\rangle)
\varphi_2^R(\langle b,1\rangle)\eta_{\tau}'(\det b^{-1})t^{c(s'-(rkc-2c+1)/2)-1}\,d(a,b)\,dt,\nonumber\\
\label{eq:RS global for functional equation summand 2}
&\int_0^1\int\limits_{(C_{2c}M_{(c^2)}(F))\backslash M_{(c^2)}(\A)^1}
\int\limits_{V_{((rk-2)c,c)}^-(\A)}(\xi_2-\xi_2^{V_{(c^2)}})(\langle [v],(\eta_{rk}^{\diamondsuit,-})^{-1}([v])\rangle
\langle w,(\eta_{rk}^{\diamondsuit})^{-1}(w)\rangle\\&\quad\langle e_1(ta),1\rangle
\langle e_2(b),1\rangle)
\varphi_1^L(\langle ta,1\rangle)
\varphi_2^R(\langle b,1\rangle)\eta_{\tau}'(\det b^{-1})t^{c(s'-(rkc-2c+1)/2)-1}\,dv\,d(a,b)\,dt,\nonumber\\
\label{eq:RS func eq 2.2}
&\int_0^1\int\limits_{(C_{2c}M_{(c^2)}(F))\backslash M_{(c^2)}(\A)^1}
\int\limits_{V_{((rk-2)c,c)}^-(\A)}\xi_2^{V_{(c^2)}}(\langle [v],(\eta_{rk}^{\diamondsuit,-})^{-1}([v])\rangle
\langle w,(\eta_{rk}^{\diamondsuit})^{-1}(w)\rangle
\\&\quad\langle e_1(ta),1\rangle
\langle e_2(b),1\rangle)\varphi_1^L(\langle ta,1\rangle)
\varphi_2^R(\langle b,1\rangle)\eta_{\tau}'(\det b^{-1})t^{c(s'-(rkc-2c+1)/2)-1}\,dv\,d(a,b)\,dt,\nonumber\\
\label{eq:RS func eq 2.3}
&-\int_0^1\int\limits_{(C_{2c}M_{(c^2)}(F))\backslash M_{(c^2)}(\A)^1}
\xi_1^{V_{(c^2)}}(\langle e_1(ta),1\rangle\langle e_2(b),1\rangle)
\\&\quad\varphi_1^L(\langle ta,1\rangle)
\varphi_2^R(\langle b,1\rangle)\eta_{\tau}'(\det b^{-1})t^{c(s'-(rkc-2c+1)/2)-1}\,d(a,b)\,dt.\nonumber
\end{align}
We can compute \eqref{eq:RS global for functional equation summand 2} using \eqref{eq:varphi Whittaker expansion}. As above only the summand corresponding to $\Lambda_{\xi}$ remains. To compute this summand, first we shift $\langle e_1(ta),1\rangle$ and
$\langle e_2(b),1\rangle$ to the left. For the conjugations ${}^w\langle e_1(ta),1\rangle$ and
${}^w\langle e_2(b),1\rangle$ we show that for both $i$ and
all $g\in\GL_c(\A)$,
\begin{align}\label{eq:RS conjugation by w for ei g}
{}^w\langle e_i(g),\eta_{(c^{rk})}(e_i(g))\rangle =\langle {}^we_i(g),\eta_{(c^{rk})}({}^we_i(g))\rangle.
\end{align}
Indeed for each place $\nu$ of $F$, by Proposition~\ref{proposition:action of W on torus is trivial on Sp}
\begin{align*}
\sigma_{rkc,\nu}^{\diamondsuit}(w_{\nu},e_i(g_{\nu}))\sigma_{rkc,\nu}^{\diamondsuit}(w_{\nu}e_i(g_{\nu}),w_{\nu}^{-1})\sigma_{rkc,\nu}^{\diamondsuit}(w_{\nu},w_{\nu}^{-1})^{-1}=1,
\end{align*}
then by \eqref{eq:nu and sigma for covering of H},
\begin{align}\label{eq:RS rho ei g wei g}
&\rho_{rkc,\nu}^{\diamondsuit}(w_{\nu},e_i(g_{\nu}))\rho_{rkc,\nu}^{\diamondsuit}(w_{\nu}e_i(g_{\nu},w_{\nu}^{-1}))\rho_{rkc,\nu}^{\diamondsuit}(w_{\nu},w_{\nu}^{-1})^{-1}
=\frac{\eta_{rkc,\nu}^{\diamondsuit}(e_i(g_{\nu}))}
{\eta_{rkc,\nu}^{\diamondsuit}({}^{w_{\nu}}e_i(g_{\nu}))}.
\end{align}
In addition, because $e_i(g),{}^we_i(g)\in M_{(c^{rk})}(\A)$, \eqref{eq:eta beta nu} implies
\begin{align}\label{eq:RS eta g eta wg lhs}
&\eta_{(c^{rk}),\nu}(e_i(g_{\nu}))=(\eta^{\diamondsuit}_{c,\nu})^{l}(e_i(g_{\nu}))/\eta^{\diamondsuit}_{rkc,\nu}(e_i(g_{\nu})),\\
\label{eq:RS eta g eta wg rhs}
&\eta_{(c^{rk}),\nu}({}^{w_{\nu}}e_i(g_{\nu}))=(\eta^{\diamondsuit}_{c,\nu})^{l}(e_i(g_{\nu}))/\eta^{\diamondsuit}_{rkc,\nu}({}^{w_{\nu}}e_i(g_{\nu})).
\end{align}
Here when $i=1$, $l=1$ and for $i=2$, $l=rk-1$. The local version of
\eqref{eq:RS conjugation by w for ei g} follows from \eqref{eq:RS rho ei g wei g}--\eqref{eq:RS eta g eta wg rhs}, but because $\eta_{(c^{rk})}$ is well defined globally on $M_{(c^{rk})}(\A)$, we obtain \eqref{eq:RS conjugation by w for ei g}.
Now since $\xi_1$ and $\varphi_2$ are genuine and $\varphi_1$ is anti-genuine,
\begin{align*}
&\Lambda_{\xi}(\ldots\langle w,(\eta_{rk}^{\diamondsuit})^{-1}(w)\rangle\langle e_1(a),1\rangle\langle e_2(b),1\rangle)
\varphi_1^L(\langle ta,1\rangle)\varphi_2^R(\langle b,1\rangle)
\\&=\Lambda_{\xi}(\ldots\langle {}^we_1(a),1\rangle\langle {}^we_2(b),1\rangle\langle w,(\eta_{rk}^{\diamondsuit})^{-1}(w)\rangle)
\varphi_1^{{}^wL}(\langle ta,1\rangle)\varphi_2^{{}^wR}(\langle b,1\rangle).
\end{align*}
In addition when we conjugate
the subgroup of elements $[v]$ by $({}^w(e_1(ta)e_2(b)))^{-1}$, the change of measure is
$|\det tab^{-1}|^{(rk-2)c}=|t|^{(rk-2)c^2}$. Therefore \eqref{eq:RS global for functional equation summand 2} becomes
(cf. \eqref{eq:RS global 2})
\begin{align}\label{eq:RS global for functional equation summand 2.2}
&\int_0^1\int\limits_{(C_{2c}M_{(c^2)}(F))\backslash M_{(c^2)}(\A)^1}
\int\limits_{V_{((rk-2)c,c)}^-(\A)}\Lambda_{\xi}(
\langle {}^we_1(ta),1\rangle \langle {}^we_2(b),1\rangle\langle [v],(\eta_{rk}^{\diamondsuit,-})^{-1}([v])\rangle
\\&\quad\langle w,(\eta_{rk}^{\diamondsuit})^{-1}(w)\rangle)
\varphi_1^{{}^wL}(\langle ta,1\rangle)
\varphi_2^{{}^{w}R}(\langle b,1\rangle)\eta_{\tau}'(\det b^{-1})t^{c(s'-1+(rkc-2c+1)/2)-1}\,dv\,d(a,b)\,dt.\nonumber
\end{align}

Since $\tau$ is a cuspidal representation of $\GL_k^{(m)}(\A)$, $\xi_1^{V_{(c^2)}}$ and $\xi_2^{V_{(c^2)}}$
which factor through the constant terms along $V_{(c,(rk-1)c)}$ and $V_{((rk-1)c,c)}$ (resp.) vanish unless $k|c$, in which case they can be computed as in
\cite[Theorem~4.4]{KM} (see also \cite[\S~II.1]{KP}):
\begin{align*}
&\xi_1^{V_{(c^2)}}(\langle e_1(t),1\rangle)=\tau(\langle tI_k,1\rangle)\delta_{P_{(c,(rk-1)c)}}^{(rk-1)/(2rk)}(\diag(tI_c,I_{(rk-1)c})),\\
&\xi_2^{V_{(c^2)}}(\langle {}^we_1(t),1\rangle)=\tau(\langle tI_k,1\rangle)\delta_{P_{((rk-1)c,c)}}^{(rk-1)/(2rk)}(\diag(I_{(rk-1)c},tI_c)).
\end{align*}
It follows that \eqref{eq:RS func eq 2.2} and \eqref{eq:RS func eq 2.3} vanish unless $\tau(\langle tI_k,1\rangle)\varphi^L(\langle t,1\rangle)=1$, in which case we can compute $dt$ explicitly and obtain $r/((s-c/(2k)-1/2)c)$ for \eqref{eq:RS func eq 2.2} (recall $dv$ changes when we conjugate by $t$) and $r/((s+c/(2k)-1/2)c)$ for \eqref{eq:RS func eq 2.3}. E.g., when $k=c$ the only possible poles are at $s=1,0$ (cf. \cite[Theorem~2]{G7}).

The integrals \eqref{eq:RS global for functional equation summand 1} and
\eqref{eq:RS global for functional equation summand 2.2} are
absolutely convergent for all $s$ and uniformly for $s$ in a vertical strip (see e.g., \cite[\S~3.4]{JS2}).
Hence they are entire functions of $s$. It follows that $Z(s,\varphi_1,\varphi_2,\xi)$ admits meromorphic continuation to $\C$.

Carrying out this computation for $Z^*(s,\varphi_1,\varphi_2,\xi)$, when we write the analog of \eqref{eq:RS global for functional equation}
as $\int_1^{\infty}+\int_0^1$, the integral $\int_0^1$ equals \eqref{eq:RS global for functional equation summand 2.2}. Then
\begin{align*}
\int_1^{\infty}(\xi_2-\xi_2^{V_{(c^2)}})\ldots\,dt=\int_1^{\infty}(\xi_1-\xi_1^{V_{(c^2)}})\ldots\,dt+\int_1^{\infty}\xi_1^{V_{(c^2)}}-\int_1^{\infty}\xi_2^{V_{(c^2)}}\ldots\,dt,
\end{align*}
where we apply \eqref{eq:X of roots} in the opposite direction (i.e., from $\xi_2$ to $\xi_1$) thereby removing the $dv$-integral.
The first integral on the r.h.s.~ becomes \eqref{eq:RS global for functional equation summand 1}, the second --- \eqref{eq:RS func eq 2.3} and the last becomes \eqref{eq:RS func eq 2.2}. This shows that as meromorphic continuations $Z^*(s,\varphi_1,\varphi_2,\xi)=Z(s,\varphi_1,\varphi_2,\xi)$, thus $\prod_{\nu}Z^*(s,\omega_{\nu},W_{\nu})=\prod_{\nu}Z(s,\omega_{\nu},W_{\nu})$, whence \eqref{eq:RS global property} follows from \eqref{RS unr integral}, \eqref{second RS unr integral} and \eqref{eq:gamma RS}.

\begin{corollary}\label{corollary:partial L function holomorphic}
Assume $\pi$ and $\tau$ are genuine cuspidal representations of $\GL_c^{(m)}(\A)$ and $\GL_{k}^{(m)}(\A)$, and $S$ is large.
Then $L^S(s,\widetilde{\pi}\times\tau)$ is entire unless $lk=c$, in which case it is holomorphic except at most simple poles at $s=\tfrac12\pm\tfrac l2$, whose residues can be read off \eqref{eq:RS func eq 2.2}--\eqref{eq:RS func eq 2.3}.
\end{corollary}
(Cf. \cite[Theorem~2]{G7} for the coverings of \cite{KP}.) Note that the meromorphic continuation of $Z(s,\varphi_1,\varphi_2,\xi)$ on its own can be deduced directly from the Eulerian expression and the local results, since we already know $L^S(s,\widetilde{\pi}\times\tau)$ is meromorphic by \cite{Gao2018}.

\subsection{Local $\epsilon$-, $L$-factors and complete $L$-function}\label{RS complete L factors}
We define local $\epsilon$- and $L$-factors for representations of $\GL_c^{(m)}\times\GL_k^{(m)}$ following Shahidi's method \cite{Sh3}.
Assume $F$ is non-archimedean and \eqref{eq:tau mult} holds. When $\tau$ and $\widetilde{\pi}$ are tempered let $P(X)\in\C[X]$ be such that the zeroes of $P(q^{-s})$ are those of $\gamma(s,\widetilde{\pi}\times\tau,\psi)$ and $P(0)=1$, and define $L(s,\widetilde{\pi}\times\tau)=P(q^{-s})^{-1}$. This definition is independent of $\psi$ by \eqref{eq:RS dependence on psi}, and
\eqref{eq:RS functional equation} implies
\begin{align}\label{eq:RS eps}
\epsilon(s,\widetilde{\pi}\times\tau,\psi)=\frac{\gamma(s,\widetilde{\pi}\times\tau,\psi)L(s,\widetilde{\pi}\times\tau)}{L(1-s,\pi\times\widetilde{\tau})}\in \C[q^{-s},q^s]^*.
\end{align}
Then we have
\begin{align}\label{eq:RS func eq}
\gamma(s,\widetilde{\pi}\times\tau,\psi)=\epsilon(s,\widetilde{\pi}\times\tau,\psi)
\frac{L(1-s,\pi\times\widetilde{\tau})}{L(s,\widetilde{\pi}\times\tau)}.
\end{align}
The essentially tempered case is handled using \eqref{eq:RS Unramified twisting}.
In general $\widetilde{\pi}$ and $\tau$ are given in terms of the Langlands' classification, then $L(s,\widetilde{\pi}\times\tau)$ and
$\epsilon(s,\widetilde{\pi}\times\tau,\psi)$ are the products of the $L$- and $\epsilon$-factors for essentially
tempered representations.
Then \eqref{eq:RS func eq} still holds, by \eqref{eq:RS Unramified twisting}, \eqref{eq:RS mult I} and \eqref{eq:tau mult}.
Note that for $k=1$ \eqref{eq:tau mult} holds vacuously.

When $F=\C$ define $L(s,\widetilde{\pi}\times\tau)=L(s,\mathrm{st}\circ\varphi)$
and $\epsilon(s,\widetilde{\pi}\times\tau,\psi)=\epsilon(s,\mathrm{st}\circ\varphi,\psi)$ with the notation of
Theorem~\ref{theorem:RS ten commendments}.
Then \eqref{eq:RS func eq} holds because of \eqref{eq:RS Archimedean property}.

When data are unramified and tempered, by Lemma~\ref{lemma:unr unitary Satake} and \eqref{eq:RS unramified factors} we deduce
that $L(s,\widetilde{\pi}\times\tau)$ equals the
unramified $L$-function (see \S~\ref{unr L functions}), then \eqref{eq:RS eps} and \eqref{eq:RS unramified factors} also imply $\epsilon(s,\widetilde{\pi}\times\tau,\psi)=1$. Now these statements hold in the unramified case in general, by the definitions.

For genuine cuspidal
representations $\pi$ and $\tau$, define $L(s,\widetilde{\pi}\times\tau)=\prod_{\nu}L(s,\widetilde{\pi}_{\nu}\times\tau_{\nu})$.
\begin{theorem}\label{theorem:complete RS GL L function}
The $L$-function $L(s,\widetilde{\pi}\times\tau)$ is absolutely convergent for $\Real(s)\gg0$, admits meromorphic continuation to $\C$ and satisfies a functional equation $L(s,\widetilde{\pi}\times\tau)=\epsilon(s,\widetilde{\pi}\times\tau)L(1-s,\pi\times\widetilde{\tau})$.
\end{theorem}
\begin{proof}
Since \eqref{eq:RS global 1} is absolutely convergent for $\Real(s)\gg0$ and so are the local integrals,
identities \eqref{eq:RS global 1 decompose 3} and \eqref{RS unr integral} imply $L^S(s,\widetilde{\pi}\times\tau)$ is absolutely convergent there ($S$ as in \eqref{eq:RS global property}). Then so is $L(s,\widetilde{\pi}\times\tau)$. The meromorphic continuation follows from
Corollary~\ref{corollary:partial L function holomorphic} and the definitions of the local $L$-functions.
The last assertion follows from \eqref{eq:RS global property}, \eqref{eq:RS func eq} and \eqref{eq:RS dependence on psi}.
\end{proof}

As a corollary we define the ``standard" $L$-function of $\pi$. Locally,
we have the character $\varepsilon\otimes\vartheta$ of $\widetilde{C}_{r,1}$
where $\vartheta=1$ unless $m\equiv2\,(4)$ in which case $\vartheta=\gamma_{\psi'}$.
Then we have the induced representation of $\GL_1^{(m)}$ constructed by the Stone--von Neumann Theory (\cite[\S~0.3]{KP}, \cite[\S~13.5]{McNamara}),
re-denoted $\varepsilon\otimes\vartheta$ (an irreducible finite-dimensional representation). By definition
$(\varepsilon\otimes\vartheta)^*=\varepsilon\otimes\vartheta$.
Then for an irreducible representation $\pi$ of $\GL_c^{(m)}$ we define
\begin{align*}
\gamma(s,\pi,\psi)=\gamma(s,\pi\times(\varepsilon\otimes\vartheta),\psi),\quad
L(s,\pi)=L(s,(\pi\times(\varepsilon\otimes\vartheta)),\quad\epsilon(s,\pi,\psi)=\epsilon(s,\pi\times(\varepsilon\otimes\vartheta),\psi).
\end{align*}
We recall that \eqref{eq:RS mult I} is valid and we do not assume \eqref{eq:tau mult}, since now $k=1$ anyway.
If $\pi$ is unramified and also parameterized using $\vartheta$, $L(s,\pi)=L_{\vartheta}(s,\pi)$. This notation is perhaps misleading, since for $m\equiv2\,(4)$ the r.h.s.~ depends on $\vartheta$ (but so does the l.h.s., which is short for
$L(s,\pi\times(\varepsilon\otimes\vartheta))$). In \S~\ref{The integrals} and onward these standard factors
will only be used for odd $m$.

Globally, define $\vartheta$ as in \S~\ref{unr L functions}, then
for each place $\nu$ we have the character $\varepsilon\otimes\vartheta_{\nu}$. This defines a genuine character
$\varepsilon\otimes\vartheta=\otimes'_{\nu}(\varepsilon\otimes\vartheta_{\nu})$ of
the center of $\GL_1^{(m)}(\A)$, which we can extend to a maximal abelian subgroup of $\GL_1^{(m)}(\A)$ then induce to
a genuine (irreducible) cuspidal representation of $\GL_1^{(m)}(\A)$, by the global Stone--von Neumann Theory (\cite[\S~0.3, \S~II.1]{KP}). (In a global context we identified the local subgroups $\mu_{m}$ with a single subgroup of $\C^*$, so that the notation $\varepsilon_{\nu}$ is unnecessary.) Re-denote this representation by $\varepsilon\otimes\vartheta$. Define
$L(s,\pi)=L(s,\pi\times(\varepsilon\otimes\vartheta))$, then $L(s,\pi)=\prod_{\nu}L(s,\pi_{\nu})$.
Finally by Corollary~\ref{corollary:partial L function holomorphic}:
\begin{theorem}
The partial $L$-function $L^S(s,\pi)$ for sufficiently large sets $S$ is
holomorphic except at most simple poles at $s=\tfrac12\pm\tfrac c2$.
\end{theorem}
We emphasize that here $k=1$ and $r$ is arbitrary, the only assumption is $\mu_{2m}<F^*$.

Of course one can repeat the above construction with the roles of $\pi$ and $\tau$ reversed, i.e., $c=1$, $\pi$ is the fixed representation of
$\GL_1^{(m)}$ and $k$ is arbitrary. The theorem will then provide the $L$-function $L^S(s,\tau)$ which is entire unless $k=1$ (since the poles occur only
if $k|c$), and holomorphic for $k=1$ except at most simple poles at $s=1,0$. This resembles the familiar result in the linear case. However,
this construction would depend on the conjectures of \S~\ref{the representation rho_c(tau)}.

\section{The doubling integrals}\label{The integrals}
Here $F$ is a local field. We recall that all conjectures from \S~\ref{the representation rho_c(tau)} are assumed to hold.
\subsection{The doubling setup}\label{The doubling setup}
Let $n,k$ and $m$ be positive integers, $r=m$ when $m$ is odd and $r=m/2$ otherwise.
Let $G$ be either $\Sp_{2n}$ or $\GL_n$. If $G=\Sp_{2n}$, let $c=2n$, $H=\Sp_{2rkc}$, $\epsilon_0=-1$, and $Q=M_Q\ltimes U$ be the standard parabolic subgroup with $M_Q=\GL_c\times\ldots\times\GL_c\times\Sp_{2c}$.
For $G=\GL_n$ put $c=n$, $H=\GL_{2rkc}$, $\epsilon_0=1$, $Q=M_Q\ltimes U$ with $M_Q=M_{(c^{rk-1},2c,c^{rk-1})}$ and $U=V_{(c^{rk-1},2c,c^{rk-1})}$.

Fix a nontrivial additive character $\psi$ of $F$. Define the following character $\psi_U$ of $U$.
Assume $rk>1$.
If $G=\Sp_{c}$, $\psi_U(\diag(v,I_{2c},v^*))=\psi_{rk-1}(v)$ for $v\in V_{(c^{rk-1})}$,
and for $G=\GL_c$, $\psi_U(\diag(v_1,I_{2c},v_2))=\psi_{rk-1}^{-1}(v_1)\psi_{rk-1}^{-1}(v_2)$, $v_1,v_2\in V_{(c^{rk-1})}$.
Write the middle $4c\times 4c$ block of an element in $U$ in the form
\begin{align}\label{matrix:middle 4c block of u}
\left(\begin{smallmatrix}I_c&u&v\\&I_{2c}&u'\\&&I_c\end{smallmatrix}\right).
\end{align}
For $G=\Sp_{c}$, let $u^{1,1}$ (resp., $u^{2,2}$) be the top left (resp., bottom right) $c/2\times c/2$ block of $u$,
otherwise $G=\GL_c$, $u^{1,1}$ is the top left $c\times c$ block of $u$ and $u^{2,2}$ is the top $c\times c$ block of $u'$.
In both cases $\psi_U$ is defined on \eqref{matrix:middle 4c block of u} by $\psi(\tr(-\epsilon_0u^{1,1}+u^{2,2}))$. For $rk=1$, $U$ and $\psi_U$ are trivial.

The product $G\times G$ is embedded in $M_Q$ in the stabilizer of $\psi_U$, the embedding is given by
\begin{align*}
(g_1,g_2)=\begin{cases}
\diag(g_1,\ldots,g_1,\left(\begin{smallmatrix} g_{1,1}&&g_{1,2}\\ &g_2&\\ g_{1,3}&&g_{1,4}\end{smallmatrix}\right),g_1^*,\ldots,g_1^*),&G=\Sp_c,\\
\diag(g_1,\ldots,g_1,g_1,g_2,g_1,\ldots,g_1),&G=\GL_c.\end{cases}
\end{align*}
Here if $G=\Sp_c$, $g_1=\left(\begin{smallmatrix}g_{1,1}&g_{1,2}\\g_{1,3}&g_{1,4}\end{smallmatrix}\right)$ with $g_{1,i}\in \Mat_n$, and for
$G=\GL_c$ there are $rk-1$ copies of $g_1$ to the right of $g_2$. For clarity,
put $\mathfrak{e}_1(g)=(g,I_c)$ and $\mathfrak{e}_2(g)=(I_c,g)$, $g\in G$.

Consider $G=\Sp_c$. We realize $H^{(m)}$ using $\sigma_{2rkc}$ and $G^{(m)}$ using $\sigma_c$. By \cite[Proposition~7]{me12},
\begin{align}\label{eq:the $2$-cocycle on G times G formula}
\sigma_{2rkc}((g_1,g_2),(g_1',g_2'))&=\sigma^{*,rk}_{c}(g_1,g_1')^{-1}\sigma_{c}(g_2,g_2'),\qquad g_1,g_2,g_1',g_2'\in G.
\end{align}
In particular $\mathfrak{e}_1(G)$ and $\mathfrak{e}_2(G)$ commute in $H^{(m)}$. The mappings
\begin{align}\label{eq:embeddings coverings G and G into H}
\langle g,\epsilon\rangle\mapsto\langle \mathfrak{e}_1(g),\epsilon^{-1}\rangle,\qquad
\langle g,\epsilon\rangle\mapsto\langle \mathfrak{e}_2(g),\epsilon\rangle
\end{align}
define the lift of the embedding $G\times G<H$ to an embedding
\begin{align}\label{eq:lift of embedding}
\{(\epsilon_1,\epsilon_2)\in\mu_m^2:\epsilon_1=\epsilon_2\}\backslash G[\sigma_c^{*,rk}]\times G^{(m)}\rightarrow H^{(m)}.
\end{align}
In addition
\begin{align}\label{eq:g_1 and g_2 product in H}
\langle\mathfrak{e}_1(g_1),\epsilon_1^{-1}\rangle\langle \mathfrak{e}_2(g_2),\epsilon_2\rangle=\langle (g_1,g_2),\epsilon_1^{-1}\epsilon_2\rangle.
\end{align}
To pass between $G[\sigma_c^{*,rk}]$ and $G^{(m)}(=G[\sigma_c])$ we use the canonical isomorphism
\begin{align}\label{eq:iso sigma sigma * rk}
G[\sigma_c^{*,rk}]\rightarrow G^{(m)},\qquad \langle g,\epsilon\rangle\mapsto \langle g,\varsigma_{*,c}^{rk+1}(g)\epsilon\rangle.
\end{align}

Assume $G=\GL_c$. Realize $H^{(m)}$ using $\sigma_{2rkc}^{\diamondsuit}$ and (both copies of) $G^{(m)}$ via $\sigma_c^{\diamondsuit}$.
Then by
\eqref{eq:block compatibility on Levi of P} and with notation similar to \eqref{eq:the $2$-cocycle on G times G formula},
\begin{align}\label{eq:the $2$-cocycle on G times G formula GL}
\sigma^{\diamondsuit}_{2rkc}((g_1,g_2),(g_1',g_2'))&=\sigma^{\diamondsuit}_{c}(g_1,g_1')^{2rkc-1}\sigma_{c}^{\diamondsuit}(g_2,g_2')
=\sigma^{\diamondsuit}_{c}(g_1,g_1')^{-1}\sigma_{c}^{\diamondsuit}(g_2,g_2').
\end{align}
Then \eqref{eq:embeddings coverings G and G into H} implies the analog of \eqref{eq:lift of embedding} with $G^{(m)}$ instead of
$G[\sigma_c^{*,rk}]$, and \eqref{eq:g_1 and g_2 product in H} also holds.

\subsection{Sections and intertwining operators}\label{Sections and intertwining operators}
We describe the spaces and intertwining operators necessary for the definitions of the local factors.
See \cite[\S~3.3, \S~4]{CFK} for more details.

Let $\tau_0$ be a genuine irreducible generic representation of $\GL_k^{(m)}$,
$\rho_c(\tau_0)$ is an $(rk,c)$ representation. If $H=\Sp_{2rkc}$, let $P$ be the standard Siegel parabolic subgroup of $H$, i.e., $M_P=\GL_{rkc}$.
Denote $\tau=\tau_0$, $\rho_c(\tau)$ is a representation of $\widetilde{M}_P$.
For $H=\GL_{2rkc}$, let $P=P_{(rkc,rkc)}$, $\tau=\tau_0\otimes\tau_0^*$, $\rho_c(\tau)=\rho_c(\tau_0)\otimes\rho_c(\tau_0^*)$
and $W_{\psi}(\rho_c(\tau))=W_{\psi}(\rho_c(\tau))\otimes W_{\psi}(\rho_c(\tau^*))$, and
extend $\det$ to a function of $M_P$ by $\det(\diag(m_1,m_2))=\det{m_1}\det {m_2^{-1}}$. We use similar notation in the following setup: For a  composition $\beta=(\beta_1,\beta_2)$ of $l$, let $H=\GL_{rlc}$, $P=P_{\beta rc}$ and
$\tau=\tau_1\otimes\tau_2$, where each $\tau_i$ is a representation of $\GL_{\beta_i}^{(m)}$ with properties similar to $\tau_0$, e.g., $\rho_c(\tau_i)$ is an $(r\beta_i,c)$ representation.

For $s\in\C$ we have the space
$V(s,\rho_c(\tau))$ of $\Ind_{\widetilde{P}}^{H^{(m)}}(|\det|^{r^{-1}(s-1/2)}W_{\psi}(\rho_c(\tau)))$. A function $f$ on $\C\times H^{(m)}$ is called an entire section of $V(\rho_c(\tau))$ if $f(s,\cdot)\in V(s,\rho_c(\tau))$ for all $s$, and $f(\cdot,h)$ is an entire function of $s$ for each $h$. We call $f$ a meromorphic section if $\varphi(s)f(s,h)$ is an entire section for an entire and not identically zero function $\varphi:\C\rightarrow\C$. Over non-archimedean fields an entire (resp., meromorphic) function will always be polynomial (resp., rational) in $q^{\mp s}$. If $F$ is unramified, a function $f(s,\cdot)\in V(s,\rho_c(\tau))$ is unramified if it is right-invariant under
the subgroup $\{\langle y,\eta_{2rkc}(y)\rangle:y\in K_H\}$, and normalized if $f(s,\langle I_{2rkc},1\rangle)=1$ (if $H\ne\Sp_{2rkc}$, $\eta_{2rkc}$ is replaced by $\eta_{2rkc}^{\diamondsuit}$). The normalized unramified section $f$ is the unique entire section such that $f(s,\cdot)$ is normalized unramified for all $s$. If $h_0\in H^{(m)}$, $h_0\cdot f$ is the section defined by $h_0\cdot f(s,h)=f(s,hh_0)$.

For $H=\Sp_{2rkc}$, let $\mathrm{d}_{rk,c}=\diag(-I_c,I_c,\ldots,(-1)^{rk}I_c)\in T_{\GL_{rkc}}$ regarded as an element of $M_P$,
and denote $w_P=\delta_0^{-1}\mathrm{d}_{rk,c}$, $P'=P$ and $\tau'=\tau^{*}$. Also let $\psi_{U_P}(\left(\begin{smallmatrix}I_{rkc}&u\\&I_{rkc}\end{smallmatrix}\right))=\psi(\tr(x))$, where $x$ is the bottom left $c\times c$ block of $u$. Define
$Y_{rk,c}=V_{(c^{rk})}\ltimes {U_{P}}$ and $\psi_{rk,c}=\psi_{rk-1}\otimes\psi_{U_P}$.

If $H=\GL_{rlc}$ and $\tau=\tau_1\otimes\tau_2$, we have $\beta'=(\beta_2,\beta_1)$, $w_P=w_{\beta'rc}$, $\tau'=\tau_2\otimes\tau_1$ and $V(1-s,W_{\psi}(\rho_c(\tau')))$ is induced from $P'=P_{\beta'rc}$ (instead of $P$). In particular for $\tau=\tau_0\otimes\tau_0^*$, $\tau'=\tau_0^*\otimes\tau_0$. In addition $Y_{rl,c}=V_{(c^{rl})}$ and $\psi_{rl,c}=\psi_{rl-1}$ (when $P=P_{(rkc,rkc)}$, $l=2k$ and $\beta'=\beta=(k,k)$).

Consider the following intertwining operators:
\begin{align*}
&M(s,\rho_c(\tau),w_P):V(s,W_{\psi}(\rho_c(\tau)))\rightarrow V(1-s,W_{\psi}(\rho_c(\tau'))),\\
&M(1-s,\rho_c(\tau'),w_{P'}):V(1-s,W_{\psi}(\rho_c(\tau')))\rightarrow V(s,W_{\psi}(\rho_c(\tau))).
\end{align*}
For a meromorphic section $f$ of $V(\rho_c(\tau))$, $M(s,\rho_c(\tau),w_P)f$ is defined
for $\Real(s)\gg0$ by the absolutely convergent integral
\begin{align}\label{int:intertwining operator}
M(s,\rho_c(\tau),w_P)f(s,h)=\int\limits_{U_{P'}}f(s,\langle w_P^{-1}u,1\rangle h)\,du,
\end{align}
then by meromorphic continuation to $\C$.
We normalize these operators as follows. Define
\begin{align}\label{eq:lambda functional}
\lambda(s,c,\tau,\psi)f=\int_{U_{P'}}f(s,\langle \delta_0u,1\rangle)\psi_{rl,c}^{-1}(u)\,du.
\end{align}
Here for $H=\Sp_{2rkc}$, $l=k$.
The r.h.s.~ is absolutely convergent for $\Real(s)\gg0$, and can be made nonzero for a given $s$. Over non-archimedean fields
it can be made constant.
\begin{theorem}\label{theorem:uniqueness for normalizing intertwining}
For all $s$, $\dim\Hom_{Y_{rl,c}}(V(s,\rho_c(\tau)),\psi_{rl,c})\leq1$.
\end{theorem}
In the case $m=1$ for $H=\Sp_{2rkc}$ this was proved in \cite[Appendix~B]{CFK}. The proof
was based on the analysis of the distributions on the orbits of the right action of $Y_{rk,c}$ on the homogeneous space $P\backslash H$. Since $\widetilde{P}\backslash H^{(m)}\cong P\backslash H$ and $H^{(m)}$ is split over $Y_{rk,c}$, the proof carries over immediately to arbitrary $m$. In fact one simply replaces all occurrences of $k$ in \textit{loc. cit.} with $rk$.

Note that for $H=\GL_{rlc}$ the dimension is always $1$, by Proposition~\ref{prop:rodier non}.

We also mention that the space in the theorem can still be zero for some $s$ when $H=\Sp_{2rkc}$, otherwise the corollary in \cite[\S~1]{Banks}
would imply $\lambda(s,c,\tau,\psi)$ is polynomial in $q^{\mp s}$, contradicting \cite[Proposition~4]{LR} for $m=1$.

\begin{corollary}\label{corollary:analytic continuation of lambda functional}
The functional $\lambda(s,c,\tau,\psi)$ admits meromorphic continuation, which is continuous in $f$ over $\C$.
\end{corollary}
\begin{proof}
Use Theorem~\ref{theorem:uniqueness for normalizing intertwining} and \cite{Banks} for non-archimedean fields, and
\cite[\S~6.13]{CFK} for $\C$.
\end{proof}
Thus there is a meromorphic function $C(s,c,\tau,\psi)$ such that for all meromorphic sections $f$,
\begin{align}\label{eq:func equation for normalization}
&\lambda(s,c,\tau,\psi)f=
C(s,c,\tau,\psi)\lambda(1-s,c,\tau',\psi)M(s,\rho_c(\tau),w_P)f.
\end{align}
The choice of measure on $U_{P'}$ is fixed as in \cite[\S~4]{CFK} using $d_{\psi}x$ (see \S~\ref{ints and RS integrals}). Define the normalized intertwining operator by
\begin{align}\label{eq:def normalized intertwining}
M^*(s,c,\tau,\psi)=C(s,c,\tau,\psi)M(s,\rho_c(\tau),w_P).
\end{align}
It then follows that
\begin{align}\label{eq:composition of intertwining operators}
M^*(1-s,c,\tau',\psi)M^*(s,c,\tau,\psi)=1.
\end{align}

Assume $\tau$ is genuine irreducible and unramified, and for $H=\GL_{rlc}$ further assume $l=2k$ and $\tau=\tau_0\otimes\tau_0^*$ (this is the only case needed below). Using the definitions from \S~\ref{unr L functions}, let
\begin{align*}
&a(s,c,\tau)\\&=\begin{cases}[L(s-rc/2,\tau)]\prod\limits_{j=1}^{\lfloor rc/2\rfloor}L(2s-rc+2j-1,\tau,\vee^2)\prod_{j=1}^{\lceil rc/2\rceil}L(2s-rc+2j-2,\tau,\wedge^2)&H=\Sp_{2rkc},\\
\prod_{1\leq j\leq rc}L(2s+j-rc-1,\tau_0\times\tau_0)&H=\GL_{2rkc},
\end{cases}\\
&b(s,c,\tau)=\begin{cases}[L(s+rc/2,\tau)]\prod\limits_{j=1}^{rc/2}L(2s+2j-2,\tau,\vee^2)L(2s+2j-1,\tau,\wedge^2)&H=\Sp_{2rkc},\\
\prod_{1\leq j\leq rc}L(2s+j-1,\tau_0\times\tau_0)&H=\GL_{2rkc}.\\
\end{cases}
\end{align*}
Here factors in square brackets are only included when $m$ is odd. Also for
$x\in\R$, $\lfloor x \rfloor$ (resp., $\lceil x \rceil$) is the largest (resp., smallest) integer smaller (resp., greater) than or equal to $x$. For $m=1$ the definitions agree with \cite[\S~4]{CFK} (also for $m=2$, except the factors in square brackets).

If $\psi$ is unramified and $f_{\tau}$ (resp., $f_{\tau'}$) is the normalized unramified section of $V(\rho_c(\tau))$ (resp., $V(\rho_c(\tau'))$), the computations in \cite[Lemmas~27, 33]{CFGK2} and \cite[Lemmas~79, 80]{me12}
(the inducing character for $\Sp_{2rkc}$ is $\vartheta\otimes_{1\leq i\leq k,1\leq j\leq rc}\chi_i|~|^{r^{-1}(s-1/2)-c/2+j/r-1/(2r)}$, cf. \cite[p.~89]{me12}) show
\begin{align}\label{intertwining operator on unramified}
M(s,\rho_c(\tau),w_P)f_{\tau}={a(s,c,\tau)}{b(s,c,\tau)^{-1}}f_{\tau'}.
\end{align}
This identity is in fact the Gindikin--Karpelevich formula \cite[Corollary~7.4]{Gao2018}, which is the extension of \cite[Theorem~3.1]{CS1} to coverings.

\subsection{The integrals and $\gamma$-factor}\label{The local integrals}
For $G=\Sp_{c}$ let $\iota=\left(\begin{smallmatrix}&I_n\\I_{n}\end{smallmatrix}\right)$, ${}^{\iota}g=\iota^{-1}g\iota$ is an involution of $G$ which according to \cite[\S~1.6]{me12} lifts uniquely to an involution of $G^{(m)}$ (locally and globally), also denoted ${}^{\iota}$. For $G=\GL_c$ set $\iota=I_c$. Let
\begin{align*}
\delta=\delta_0\delta_1,\qquad
\delta_0=\left(\begin{smallmatrix}&I_{rkc}\\\epsilon_0I_{rkc}\end{smallmatrix}\right),\qquad
\delta_1=\diag(I_{(rk-1)c},\left(\begin{smallmatrix}I_c&I_c\\&I_c\end{smallmatrix}\right),I_{(rk-1)c}),\qquad
U_0=U\cap U_P.
\end{align*}

Let $\pi$ and $\tau$ be genuine irreducible admissible representations, $\pi$ of $G^{(m)}$ and
$\tau$ is either $\tau_0$ or $\tau_0\otimes\tau_0^*$ as in \S~\ref{Sections and intertwining operators} (according to $G$).
Let $\omega$ be a matrix coefficient of $\pi^{\vee}$ and $f$ be an entire section of $V(\rho_c(\tau))$. The local integral is defined by
\begin{align}\label{eq:local integral}
&Z(s,\omega,f)=
\int\limits_{G}\int\limits_{U_0}
\omega(\langle g,1\rangle)f(s,\langle\delta u_0,1\rangle
\,{}^{\iota}\langle\mathfrak{e}_2(g),1\rangle)\,\psi_{U}(u_0)\,du_0\,dg.
\end{align}
It is absolutely convergent in $\Real(s)\gg0$ independently of the data $(\omega,f)$.
Over non-archimedean fields \eqref{eq:local integral} can be made constant. Over $\C$, the results in the linear case imply that
the integrals admit meromorphic continuation in $s$ which is continuous in the input data, and for a given $s$, one can choose data for which the integral is holomorphic and nonzero in a neighborhood of $s$, and $f$ is entire and $K_H$-finite (\cite[\S~6.13, Corollary~44]{CFK} with $rk$ instead of $k$).

In its domain of definition, by \cite[Propositions~68, 75]{me12} the integral belongs to
\begin{align}\label{Homspace}
\Hom_{(G,G)^{(m)}}(J_{U,\psi_U^{-1}}(V(s,\rho_c(\tau))),\pi^{\vee}\otimes\pi^{\iota}).
\end{align}
By \cite[Theorem~3.1]{DimaKaplan} this space is at most one-dimensional, outside a discrete subset of $s$. In particular, over non-archimedean fields this and Bernstein's continuation principle (in \cite{Banks}) imply $Z(s,\omega,f)$ admits meromorphic continuation
to a rational function in $q^{-s}$, for a meromorphic section $f$. Note that we only consider meromorphic sections.

We define the $\gamma$-factor as in \cite[\S~5]{CFK} (see also \cite{LR,me6}).
For $G=\Sp_c$, let $N=c+1$ if $m$ is odd and otherwise $N=c$ (we will not use $N$ for
$\GL_c$). Put $\mathfrak{i}_G=\langle -I_c,1\rangle$. Since $\mu_{2m}\subset F^*$,
$\mathfrak{i}_G\in C_{G^{(m)}}$. Denote
\begin{align*}
&Z^*(s,\omega,f)=Z(1-s,\omega,M^*(s,c,\tau,\psi)f),\\
&\vartheta(s,c,\tau,\psi)=\begin{cases}
[\gamma(s,\tau,\psi)]
\tau(\mathfrak{i}_{\GL_k})^{rn}\tau(\langle r^rI_k,1\rangle)^{-c}\eta_{\tau}(\langle2,1\rangle)^{-c}|2r|^{-kc(s-1/2)}&G=\Sp_{c},\\
\tau_0(\langle r^rI_k,1\rangle)^{-2c}|r|^{-2kc(s-1/2)}\tau_0(\mathfrak{i}_{\GL_k})^{rc}&G=\GL_c.
\end{cases}
\end{align*}
Here $\gamma(s,\tau,\psi)$ is the $\gamma$-factor defined in \S~\ref{RS complete L factors}
($\tau$ here is $\pi$ there) and appears only if $m$ is odd.

Since outside a discrete subset of $s$ the space \eqref{Homspace} is at most one-dimensional, and by the basic properties of the integrals described above, there is a well defined function $\gamma(s,\pi\times\tau,\psi)$ which is
meromorphic and not identically zero, such that for all data $(\omega,f)$,
\begin{align}\label{eq:Gamma def}
\gamma(s,\pi\times\tau,\psi)Z(s,\omega,f)=\pi(\mathfrak{i}_G)^{rk}
\vartheta(s,c,\tau,\psi)Z^*(s,\omega,f).
\end{align}
For $c=0$, $\Sp_0=\{1\}$ and we define
$\gamma(s,\pi\times\tau,\psi)=\gamma(s,\tau,\psi)$ when $m$ is odd and $\gamma(s,\pi\times\tau,\psi)=1$ otherwise.
As in \S~\ref{ints and RS integrals}, Schur's Lemma (applied to $\pi$) and the uniqueness of the $(rk,c)$ model for $\rho_c(\tau)$
imply the $\gamma$-factor is independent of concrete realizations of $\pi,\tau$ and the $(rk,c)$ model.

The following is our main theorem. We formulate it for $G=\Sp_c$, except
\eqref{eq:multiplicativity I} and \eqref{eq:multiplicativity II} which are formulated for $\GL_c$ as well and \eqref{eq:GL factors}
(see the ensuing remarks).
\begin{theorem}\label{theorem:ten commendments}
The $\gamma$-factor satisfies the following properties.
\begin{itemize}[leftmargin=*]
\item Unramified twisting:
\begin{align}\label{eq:Unramified twisting}
\gamma(s,\pi\times|\det|^{s_0}\tau,\psi)=\gamma(s+rs_0, \pi\times\tau,\psi).
\end{align}
\item Multiplicativity: Let $\pi$ be a quotient of $\Ind_{\widetilde{R}}^{G^{(m)}}(\sigma_{\beta'}\otimes\pi')$
where $R$ is a standard parabolic subgroup of $G$ with $M_R=M_{\beta'}\times G'$, $\beta'$ is a $d'$ parts composition of $l\leq n$,
and $\sigma_{\beta'}\otimes\pi'$ is a genuine irreducible admissible representation of $\widetilde{M}_R$ ($\pi'$ is omitted when $l=n$).
Let $\tau$ be a quotient of $\Ind_{\widetilde{P}_{\beta}}^{\GL_k^{(m)}}(\otimes_{i=1}^d\tau_i)$ with notation as in \eqref{rep:rho c tau in general} or Conjecture~\ref{conjecture:rho c tau inductive for tempered tau}. Then
    \begin{align}\label{eq:multiplicativity I}
    &\gamma(s,\pi\times\tau,\psi)=\gamma(s,\pi'\times\tau,\psi)\prod_{i=1}^{d'}\gamma(s,\sigma_i\times(\tau\otimes\tau^*),\psi),\qquad\\
    &\gamma(s,\pi\times\tau,\psi)=\prod_{i=1}^d\gamma(s,\pi\times\tau_i,\psi).\label{eq:multiplicativity II}
    \end{align}
Here if $G=\Sp_c$, $l=n$ and $m$ is odd, $\gamma(s,\pi'\times\tau,\psi)=\gamma(s,\tau,\psi)$ (by the definitions for $\Sp_{2(n-l)}$).
For $G=\GL_c$: $\gamma(s,\pi'\times\tau,\psi)$ is omitted, $\tau\otimes\tau^*$ is replaced by $\tau$
and \eqref{eq:multiplicativity II} holds as is.
\item
Unramified factors: For unramified data and when $\pi,\tau$ are parameterized with one parameter,
    \begin{align}\label{eq:unramified factors}
    \gamma(s,\pi\times\tau,\psi)=
    \frac{L(1-s,\widetilde{\pi}\times\widetilde{\tau})}{L(s,\pi\times\tau)}.
    \end{align}

\item
Dependence on $\psi$: For $b\in F^*$,
    \begin{align}\label{eq:dependence on psi}
    \gamma(s,\pi\times\tau,\psi_b)=\eta_{\tau}(\langle b,1\rangle)^{c}[\tau(\langle b^r,1\rangle)\eta_{\varepsilon\otimes1}(\langle b,1\rangle)^k]|b|^{kN(s-1/2)}\gamma(s,\pi\times\tau,\psi).
    \end{align}
Here the factors in square brackets appear only when $m$ is odd.
\item
Duality:
\begin{align}\label{eq:self-duality}
\gamma(s,\widetilde{\pi}\times\tau,\psi)=\gamma(s,\pi\times\tau,\psi).
\end{align}

\item
Functional equation:
    \begin{align}\label{eq:functional equation}
    \gamma(s,\pi\times\tau,\psi)\gamma(1-s,\widetilde{\pi}\times\widetilde{\tau},\psi^{-1})=1.
    \end{align}

\item $\GL_c^{(m)}\times\GL_k^{(m)}$ factors:
    \begin{align}\label{eq:GL factors}
    \gamma(s,\pi\times(\tau_0\otimes\widetilde{\tau_0}),\psi)=\gamma(s,\pi\times\tau_0,\psi)\gamma(s,\widetilde{\pi}\times\tau_0,\psi).
    \end{align}
The $\gamma$-factors on the r.h.s.~ are the ones defined in \S~\ref{RS integrals}.
\item $F=\C$: if $m$ is odd denote ${}^LG=\SO_N(\C)$ and $\theta(\pi^r)=\pi^r$, for even $m$ put
${}^LG=\Sp_N(\C)$ and let $\theta(\pi^r)$ denote the representation attached to $\pi^r$ by the theta correspondence (\cite{Howe1989,AdamsBarbasch1995}). Let
$\varphi:\C^*\rightarrow{}^LG\times{}^L\GL_k$ be the homomorphism attached to
$\theta(\pi^r)\otimes\tau^r$ and $\epsilon(s,\mathrm{st}\circ\varphi,\psi)$ and $L(s,\mathrm{st}\circ\varphi)$ be Artin's local factors attached to $\mathrm{st}\circ\varphi$ by Langlands' correspondence (\cite{Bo,La3}), where $\mathrm{st}$ is the standard representation.
Then
\begin{align}\label{eq:Archimedean property}
\gamma(s,\pi\times\tau,\psi)=\epsilon(s,\mathrm{st}\circ\varphi,\psi)\frac{L(1-s,\mathrm{st}^{\vee}\circ\varphi)}{L(s,\mathrm{st}\circ\varphi)}.
\end{align}
\item
Crude functional equation: Let $F$ be a number field with a ring of adeles $\A$ and $\psi$ be a nontrivial character of $F\backslash\A$. Assume $\pi$ and $\tau$ are genuine cuspidal representations of $G^{(m)}(\A)$ and $\GL_{k}^{(m)}(\A)$.
With the global parametrization and the set $S$ as in \S~\ref{unr L functions},
    \begin{align}\label{eq:global property}
    L^S(s,\pi\times\tau)=\prod_{\nu\in S}\gamma(s,\pi_{\nu}\times\tau_{\nu},\psi_{\nu})L^S(1-s,
    \widetilde{\pi}\times\widetilde{\tau}).
    \end{align}
\end{itemize}
Properties \eqref{eq:multiplicativity I}, \eqref{eq:multiplicativity II}, \eqref{eq:dependence on psi}, \eqref{eq:GL factors} for $c=k=1$, and \eqref{eq:global property} uniquely determine $\gamma(s,\pi\times\tau,\psi)$ for any local field $F$ where $|m|=1$ (for both $G=\Sp_c$ and $\GL_c$).
\end{theorem}
\begin{remark}
For $G=\GL_c$, \eqref{eq:GL factors} and Theorem~\ref{theorem:RS ten commendments}
also determine the $\gamma$-factors, except here we are able to prove \eqref{eq:multiplicativity II}.
\end{remark}
\begin{remark}
The remaining properties for $\GL_c$ can easily be determined from
\eqref{eq:GL factors} and Theorem~\ref{theorem:RS ten commendments}. In some detail:
in \eqref{eq:Unramified twisting} $|\det|^{s_0}\tau=|\det|^{s_0}\tau_0\otimes|\det|^{-s_0}\tau_0^*$ and
an unramified twisting property of $\pi$ can be read from \eqref{eq:RS Unramified twisting};
for \eqref{eq:unramified factors} the denominator is replaced by $L(s,\pi\times\tau_0)L(s,\widetilde{\pi}\times\tau_0)$;
in \eqref{eq:dependence on psi} replace
$(\eta_{\tau},N)$ by $(\eta_{\tau_0},2c)$; \eqref{eq:self-duality} and \eqref{eq:functional equation} remain unchanged;
and in \eqref{eq:global property} $\tau=\tau_0\otimes\tau_0^*$ where $\tau_0$ is a genuine
cuspidal representation of $\GL_k^{(m)}(\A)$.
\end{remark}

\section{Proof of Theorem~\ref{theorem:ten commendments}}\label{proof of theorem:ten commendments}
Arguments where coverings impose no additional difficulty are described briefly, for details see \cite[\S~6]{CFK} and the references therein. Note that the proofs are given for both $\Sp_c$ and $\GL_c$.

\subsection{Uniqueness}\label{almost Uniqueness}
By \eqref{eq:multiplicativity I}, \eqref{eq:multiplicativity II}, and \eqref{eq:GL factors} for $c=k=1$, it remains to consider a non-archimedean
field $F_{\nu}$ (with $\mu_{2m}<F_{\nu}^*$) such that $|m|_{\nu}=1$, and genuine supercuspidal representations $\pi_{\nu}$ and $\tau_{\nu}$
of $G^{(m)}$ and $\GL_k^{(m)}$.
Let $F'$ be a number field with $F'_{\nu}=F_{\nu}$ and $\mu_{2m}<(F')^*$ (adding roots of unity to $F'$ will not change $F'_{\nu}$),
and $S'$ be the finite set of places $\nu'$ of $F'$ where $|m|_{\nu'}<1$. First take genuine supercuspidal
representations $(\pi_{\nu'},\tau_{\nu'})_{\nu'\in S'}$, and apply the Poincar\'{e} series argument of \cite[Appendice 1]{GH2} (see also \cite[p.~1004]{GanIchino2018}) to obtain genuine cuspidal representations $(\pi,\tau)$ of $G^{(m)}(\A)$ and $\GL_k^{(m)}(\A)$, such that for each $\nu'\notin S'$, $\pi_{\nu'}$ and $\tau_{\nu'}$ are constituents of principal series representations. Also globalize $\psi_{\nu}$ to $\psi$. Then \eqref{eq:multiplicativity I}, \eqref{eq:multiplicativity II}, \eqref{eq:dependence on psi} and for $c=k=1$ also
\eqref{eq:GL factors} uniquely determine $\gamma(s,\pi_{\nu'}\times\tau_{\nu'},\psi_{\nu'})$ for all $\nu'\notin S'$. By
\eqref{eq:global property} with $S'\subset S$ we deduce
$\prod_{\nu'\in S'}\gamma(s,\pi_{\nu'}\times\tau_{\nu'},\psi_{\nu'})$ is also uniquely determined.
Repeating this globalization argument with $S'\cup \nu$ implies the uniqueness of
$\gamma(s,\pi_{\nu}\times\tau_{\nu},\psi_{\nu})$.

\subsection{Unramified twisting}\label{Unramified twisting}This holds because
$V(s,\rho_c(|\det|^{s_0}\tau))=V(s+rs_0,\rho_c(\tau))$ and
$\vartheta(s,c,|\det|^{s_0}\tau,\psi)=\vartheta(s+rs_0,c,\tau,\psi)$, where we also
use \eqref{eq:RS Unramified twisting}.

\subsection{Dependence on $\psi$}\label{Dependence on psi}
Consider $G=\Sp_{c}$.
Changing the character $\psi$ entails changing the $(k,c)$ model of $\rho_c(\tau)$ and the normalization of the intertwining operator.
Let $\lambda$ be an $(rk,c)$ functional realizing $W_{\psi}(\rho_c(\tau))$. Denote
\begin{align*}
&t_b=\diag(b^{rk-1}I_c,\ldots,b I_c,I_{2c},b^{-1} I_c,\ldots b^{1-rk}I_c)\in T_H.
\end{align*}

First we compute $Z(s,\omega,\langle t_b,1\rangle\cdot f)$.
Since the middle block of $t_b$ is $I_{2c}$ and $\mathfrak{e}_2(G)$ is contained in the middle $c\times c$ block, \eqref{eq:BLS block compatible} implies $t_b$ commutes with $\mathfrak{e}_2(G)$ in $H^{(m)}$. Also
$t_b$ normalizes $U_0$ and commutes with $\delta_1$, these hold in $H^{(m)}$ by \eqref{eq:sigma conjugate v by h}.
The choice of $t_b$ implies ${}^{t_b^{-1}}\psi_{U}=(\psi_b)_{U}$ on $U_0$ (${}^{x^{-1}}\psi_U(y)=\psi_U({}^{x}y)$). Set $y_b={}^{\delta_0}t_b$.
By Proposition~\ref{proposition:action of W on torus is trivial on Sp},
$\langle y_b,1\rangle={}^{\delta_0}\langle t_b,1\rangle$. The map $\xi\mapsto\lambda(\langle y_b,1\rangle\cdot\xi)$ is an $(rk,c)$ functional realizing $W_{\psi_b}(\rho_c(\tau))$. Thus if $f$ is a meromorphic section on $V(\rho_c(\tau))$ where $\rho_c(\tau)$ is realized in $W_{\psi}(\rho_c(\tau))$, $Z(s,\omega,\langle t_b,1\rangle\cdot f)$ is the similar integral with $\psi$ replaced by $\psi_b$ (so, $\rho_c(\tau)$ is now realized in $W_{\psi_b}(\rho_c(\tau))$), multiplied by a measure constant $c_b$.

We similarly compute $Z^*(s,\omega,\langle t_b,1\rangle\cdot f)$. Since ${}^{w_P^{-1}}{}^{\delta_0}t_b=(b^{rk-1}I_{rkc})y_b$ and $\langle b^{rk}I_{rkc},1\rangle\in\widetilde{C}_{r,rkc}$, by Proposition~\ref{proposition:action of W on torus is trivial on Sp}, \eqref{eq:BLS $2$-cocycle on torus} and because $(b,b)_m=1$,
\begin{align}\label{eq:conjugation w_P in dep on psi}
{}^{w_P^{-1}}{}^{\delta_0}\langle t_b,1\rangle=\langle (b^{rk-1}I_{rkc})y_b,1\rangle=\langle b^{rk}I_{rkc},1\rangle
\langle b^{-1}I_{rkc},1\rangle\langle y_b,1\rangle.
\end{align}
Thus we can take out the character $\rho_c(\tau)(\langle b^{r}I_{rkc},1\rangle)^k$. The functional realizing
$W_{\psi_b}(\rho_c(\tau))$ is given by $\lambda_b(\xi)=\lambda(\langle b^{-1}I_{rkc},1\rangle\langle y_b,1\rangle\xi)$. By
\eqref{eq:def of eta}, $\lambda(\langle b^{-1}I_{rkc},1\rangle\langle y_b,1\rangle\xi)=\eta_{\tau,c}(\langle b,1\rangle)^{-1}\lambda(\langle y_b,1\rangle\xi)$ (see \S~\ref{Definition and generalities}). Then
\begin{align}\label{dep on psi Z with M}
&Z(1-s,\omega,M(s,W_{\psi}(\rho_c(\tau)),w_P)t_b\cdot f)=
c_b|b|^{-d/2}\eta_{\tau,c}(\langle b,1\rangle)^{-1}\rho_c(\tau)(\langle b^{r}I_{rkc},1\rangle)^k\\\nonumber&\qquad \times|b|^{(rk-1)rkc(r^{-1}(s-1/2)+d/2)}\delta_{P}({}^{\delta_0}t_b)Z(1-s,\omega,M(s,W_{\psi_b}(\rho_c(\tau)),w_P)f).
\end{align}
Here we used the notation $M(s,W_{\psi}(\rho_c(\tau)),w_P)$ instead of $M(s,\rho_c(\tau),w_P)$ to note that the character of the $(rk,c)$ model of $\rho_c(\tau)$ changes; $d$ is an integer and $|b|^{-d/2}$ is a measure constant.

It remains to relate between $C(s,c,\tau,\psi_b)$ and $C(s,c,\tau,\psi)$. Let
\begin{align*}
h_b=\diag(b^{rk}I_{c/2},b^{rk-1}I_c\ldots,bI_c,I_c,b^{-1}I_c,\ldots,b^{-rk+1}I_c,b^{-rk}I_{c/2})\in T_H.
\end{align*}
Put $z_b={}^{\delta_0}h_b$, then
$\langle z_b,1\rangle={}^{\delta_0}\langle h_b,1\rangle$ by Proposition~\ref{proposition:action of W on torus is trivial on Sp}, and $\xi\mapsto\lambda(\langle z_b,1\rangle\xi)$ realizes $W_{\psi_b}(\rho_c(\tau))$.
Let $f$ be a section of $V(W_{\psi}(\rho_c(\tau)))$ as above. Using \eqref{eq:sigma conjugate v by h},
\begin{align*}
\lambda(s,c,\tau,\psi)h_b\cdot f&=\delta_{P}(h_b)\int_{U_P}f(\langle z_b,1\rangle\langle \delta_0u,1\rangle )\psi_b(u)\,du=
|b|^{-d/2}\delta_{P}(h_b)\lambda(s,c,\tau,\psi_b)f.
\end{align*}
Additionally ${}^{w_{P}^{-1}}z_b=(b^{rk}I_{rkc})z_b$ and
${}^{w_{P}^{-1}}\langle z_b,1\rangle=\langle b^{rk}I_{rkc},1\rangle\langle z_b,1\rangle$. Then
\begin{align*}
&\lambda(1-s,c,\tau^*,\psi)M(s,W_{\psi}(\rho_c(\tau)),w_P)h_b\cdot f\\&=
|b|^{-d}\rho_c(\tau)(\langle b^{r}I_{rkc},1\rangle)^k|b|^{(rk)^2c(r^{-1}(s-1/2)+d/2)}\delta_{P}({}^{w_P^{-1}}h_b)\delta_{P}(h_b)
\\&\quad\times\lambda(1-s,c,\tau^*,\psi_b)M(s,W_{\psi_b}(\rho_c(\tau)),w_P)f.
\end{align*}
Therefore by \eqref{eq:func equation for normalization},
\begin{align}\label{eq:C with twist}
C(s,c,\tau,\psi)=\rho_c(\tau)(\langle b^{r}I_{rkc},1\rangle)^{-k}|b|^{d/2-rk^2c(s-1/2)}C(s,c,\tau,\psi_b).
\end{align}
Altogether \eqref{dep on psi Z with M}, \eqref{eq:C with twist}, Lemma~\ref{lemma:eta tau c 1} and the definitions imply
\begin{align*}
\gamma(s,\pi\times\tau,\psi_b)=\eta_{\tau}(\langle b,1\rangle)^{c}|b|^{kc(s-1/2)}\gamma(s,\pi\times\tau,\psi)
\frac{\vartheta(s,c,\tau,\psi_b)}{\vartheta(s,c,\tau,\psi)}.
\end{align*}
Denote by $\eta_{\tau_b,c}$ the character defined by \eqref{eq:def of eta} with $\lambda_b(\xi)$, 
$\eta_{\tau_b,c}(\langle2,1\rangle)$ appears in $\vartheta(s,c,\tau,\psi_b)$ (by Lemma~\ref{lemma:eta tau c 1}). 
Since the commutator of $(b^{-1}I_{rkc})y_b$ and $xI_{rkc}$ in $\GL_{rkc}^{(m)}$ is $(b,x)_{m}^{(rk+1)rkc}=1$ for any $x\in F^*$ 
(use \eqref{eq:Nice GL $2$-cocycle on torus}), \eqref{eq:def of eta} implies $\eta_{\tau_b,c}=\eta_{\tau,c}$. 
This proves the result if $m$ is even (then $N=c$).
In the odd case we also use \eqref{eq:RS dependence on psi} which for the ``standard" factors (see \S~\ref{RS complete L factors}) reads $\gamma(s,\tau,\psi_b)=\tau(\langle b^r,1\rangle)\eta_{\varepsilon\otimes1}(\langle b,1\rangle)^k$.

For $G=\GL_c$, the element $t_b$ remains the same; ${}^{w_P^{-1}}{}^{\delta_0}t_b=\diag(b^{rk-1}I_{rkc},b^{1-rk}I_{rkc})y_b$
and by Proposition~\ref{proposition:action of W on torus is trivial on Sp} and \eqref{eq:Nice GL $2$-cocycle on torus},
\begin{align*}
{}^{w_P^{-1}}{}^{\delta_0}\langle t_b,1\rangle=\langle \diag(b^{rk}I_{rkc},b^{-rk}I_{rkc}),1\rangle
\langle \diag(b^{-1}I_{rkc},bI_{rkc}),1\rangle\langle y_b,1\rangle.
\end{align*}
Since $\eta_{\tau_0,c}=\eta_{\tau_0^*,c}^{-1}$ (\cite[Proposition~77]{me12}), $\lambda$ translates on the left under $\langle \diag(b^{-1}I_{rkc},bI_{rkc}),1\rangle$ by 
$\eta_{\tau_0,c}(\langle b,1\rangle^{-1})\eta_{\tau_0^*,c}(\langle b,1\rangle)=\eta_{\tau_0}(\langle b,1\rangle)^{-2c}$.
In addition
\begin{align*}
&h_b=\diag(b^{rk-1}I_c,\ldots,bI_c,I_c,b^{-1}I_c,\ldots,b^{-rk}I_c),\qquad z_b={}^{\delta_0}h_b,
\end{align*}
${}^{w_{P}^{-1}}\langle z_b,1\rangle=\langle \diag(b^{rk}I_{rkc},b^{-rk}I_{rkc})z_b,1\rangle$, and $\rho_c(\tau)=\rho_c(\tau_0)\otimes\rho_c(\tau_0^*)$. Altogether we obtain
\begin{align}\label{GL dependence on psi}
\gamma(s,\pi\times(\tau_0\otimes\tau_0^*),\psi_b)=
\eta_{\tau_0}(\langle b,1\rangle)^{2c}|b|^{2kc(s-1/2)}\gamma(s,\pi\times(\tau_0\otimes\tau_0^*),\psi).
\end{align}

\subsection{Multiplicativity II: \eqref{eq:multiplicativity II}}\label{Multiplicativity II}
Let $G=\Sp_{c}$. By \eqref{rep:rho c tau in general} and Conjecture~\ref{conjecture:rho c tau inductive for tempered tau}, $\rho_c(\tau)$ is a quotient of $\Ind_{\widetilde{P}_{\beta rc}}^{\GL_{rkc}^{(m)}}(\otimes_{i=1}^d\rho_c(\tau_i))$.
For simplicity we will only consider $d=2$. Let $H'$, $P'$, $U_0'$, $\delta'=\delta'_0\delta'_1$ be the groups and elements defined in  \S~\ref{The integrals} for the $G^{(m)}\times\GL_{\beta_2}^{(m)}$ integral for $\pi\times\tau_2$.
Let $L=M_L\ltimes U_L$ be the standard parabolic subgroup of $H$ with $M_{L}=\GL_{\beta_1rc}\times H'$.
Plugging \eqref{eq:(k,c)  functional using an integral} into \eqref{eq:local integral},
\begin{align}\label{eq:mult II identity 1}
Z(s,\omega,f)=\int\limits_G\omega(g)\int\limits_{U_0}\int\limits_{V_{\beta'rc}}f(s,
\langle w_{\beta rc}v,1\rangle \langle\delta u_0,1\rangle
{}^{\iota}\langle\mathfrak{e}_2(g),1\rangle)\psi^{-1}(v)\psi_U(u_0)\,dv\,du_0\,dg.
\end{align}
In fact there is an additional twist by a complex parameter, but this is similar to the proof in the linear case (\cite[\S~6.3]{CFK}) and omitted. Denote the r.h.s.~ of \eqref{eq:mult II identity 1} by $\mathcal{I}(f)$.

By \eqref{eq:sigma on h and v}, \eqref{eq:sigma conjugate v by h} and because $\sigma_{2rkc}$ is trivial on $\mathfrak{W}^+_{2rkc}$,
\begin{align*}
\langle w_{\beta rc}v,1\rangle \langle\delta u_0,1\rangle
=\langle w_{\beta rc},1\rangle
\langle v,1\rangle
\langle \delta,1\rangle
 \langle u_0,1\rangle
 =\langle w_{\beta rc}\delta_0,1\rangle
\langle \delta_1,1\rangle
\langle {}^{\delta^{-1}}v,1\rangle
 \langle u_0,1\rangle.
 \end{align*}
We have the following properties.
\begin{enumerate}[leftmargin=*]
\item\label{II item 1}$U_0=(U_0\cap U_L)\rtimes U_0'$, ${}^{\delta^{-1}}V_{\beta'rc}$ normalizes $U_0$ and $U_{L}={}^{\delta^{-1}}V_{\beta'rc}\ltimes(U_0\cap U_L)$. Then if $u_0=u_0''u_0'$ with $u_0''\in U_0\cap U_L$ and $u_0'\in U_0'$, \eqref{eq:sigma on h and v} implies
\begin{align*}
\langle {}^{\delta^{-1}}v,1\rangle\langle u_0,1\rangle
=\langle {}^{\delta^{-1}}v,1\rangle\langle u_0'',1\rangle\langle u_0',1\rangle
=\langle u,1\rangle\langle u_0',1\rangle,\qquad u=({}^{\delta^{-1}}v)u_0'' \in U_L.
 \end{align*}
\item\label{II item 2} $w_{\beta rc}\delta_0=\delta_0'w_{L}$, where ${}^{w_{L}}U_L=U_L^-$ ($w_L=\left(\begin{smallmatrix}&&I_{\beta_1rc}\\&I_{2\beta_2rc}\\ \epsilon_0I_{\beta_1rc}\end{smallmatrix}\right)$),
    then $\langle w_{\beta rc}\delta_0,1\rangle=\langle \delta_0',1\rangle\langle w_L,1\rangle$.
\item\label{II item 3} $\delta_1=\delta'_1$.
\item\label{II item 4} $w_{L}$ commutes with $\delta'_{1}$ and $U_0'$, this extends to $H^{(m)}$ by \eqref{eq:sigma conjugate v by h}.
\item\label{II item 5} ${}^{\iota}\langle\mathfrak{e}_2(g),1\rangle$ normalizes the image of $U_L$ in $H^{(m)}$ (by \eqref{eq:sigma conjugate v by h}).
\item\label{II item 6} ${}^{\iota}\langle\mathfrak{e}_2(g),1\rangle$ is the corresponding element
in the $G^{(m)}\times\GL_{\beta_2}^{(m)}$ integral, and
$w_L$ commutes with ${}^{\iota}\langle\mathfrak{e}_2(g),1\rangle$ in $H^{(m)}$ (equivalently,
$w_L$ and $\langle\mathfrak{e}_2(G),1\rangle$ commute in $H^{(m)}$).
The first statement is immediate from the definition, the second follows from
Lemma~\ref{lemma:conjugation commutes} because ${}^{w_L}\mathfrak{e}_2(g)=\mathfrak{e}_2(g)$.
\end{enumerate}
Using these properties,
\begin{align}\label{eq:mult II identity 2}
\mathcal{I}(f)
=\int\limits_{U_L}Z'(s,\omega,\langle w_Lu,1\rangle\cdot f)\psi^{-1}(u)\,du.
\end{align}
Here $Z'$ is the $G^{(m)}\times\GL_{\beta_2}^{(m)}$ integral for $\pi\times\tau_2$; $\psi$ is the character of $U_L$ extended trivially
from the character of ${}^{\delta^{-1}}V_{\beta'rc}$, and $\langle w_Lu,1\rangle\cdot f$ is regarded as a meromorphic section of $V(\rho_c(\tau_2))$. Thus by \eqref{eq:Gamma def}
(formally, for the justification see Corollary~\ref{corollary:mult II obtaining inner integral} below or \cite[Lemma~3.4]{Soudry2}),
\begin{align*}
&\gamma(s,\pi\times\tau_2,\psi)\mathcal{I}(f)
=\pi(\mathfrak{i}_G)^{r\beta_2}\vartheta(s,c,\tau_2,\psi)\int\limits_{U_L}{Z'}^*(s,\omega,(w_Lu)\cdot f)\psi^{-1}(u)\,du.
\end{align*}
Reversing steps \eqref{eq:mult II identity 1}--\eqref{eq:mult II identity 2},
\begin{align}\label{eq:mult II identity 3}
&\gamma(s,\pi\times\tau_2,\psi)\mathcal{I}(f)
=\pi(\mathfrak{i}_G)^{r\beta_2}\vartheta(s,c,\tau_2,\psi)
\mathcal{I}(M^*(s,c,\tau_2,\psi)f).
\end{align}
Note that the restriction of $M^*(s,c,\tau_2,\psi)f$ to $\widetilde{M}_P$ belongs to the space of
\begin{align*}
\Ind_{\widetilde{P}_{\beta rc}}^{\GL_{rkc}^{(m)}}(W_{\psi}(\rho_c(\tau_1))\otimes W_{\psi}(\rho_c(\tau_2^*))).
\end{align*}

Next, since the $dv$-integration of \eqref{eq:(k,c)  functional using an integral} comprises the l.h.s.~ of \eqref{eq:func equation for normalization},
\begin{align}\label{eq:mult II gl part}
\mathcal{I}(M^*(s,c,\tau_2,\psi)f)=\mathcal{I}(M^*(s,c,\tau_1\otimes\tau_2^*,\psi)M^*(s,c,\tau_2,\psi)f).
\end{align}
Here on the r.h.s.~ $\beta$ is replaced with $(\beta_2,\beta_1)$ in \eqref{eq:mult II identity 1} (e.g., $\beta'$ becomes $(\beta_1,\beta_2)$).

To proceed we utilize the multiplicativity of the intertwining operators, namely
\begin{align}\label{eq:mult II of ops}
M^*(s,c,\tau,\psi)=M^*(s,c,\tau_1,\psi)
M^*(s,c,\tau_1\otimes\tau_2^*,\psi)
M^*(s,c,\tau_2,\psi).
\end{align}
To see this, first apply \eqref{eq:lambda functional} to $f$ to obtain
\begin{align*}
\int\limits_{U_P}\int\limits_{V_{\beta'rc}}f(s,\langle w_{\beta rc}v,1\rangle \langle \mathrm{d}_{rk,c},1\rangle\langle\delta_0 u,1\rangle)\psi^{-1}(v)\psi^{-1}(u)\,dv\,du.
\end{align*}
(See \S~\ref{Sections and intertwining operators} for the notation.) Applying
\eqref{II item 1}, \eqref{II item 2} and \eqref{II item 4} to this integral and observing that
${}^{w_{\beta rc}}\langle\mathrm{d}_{rk,c},1\rangle=\langle\diag(d',\mathrm{d}_{r\beta_2,c}),1\rangle$ by Proposition~\ref{proposition:action of W on torus is trivial on Sp}, where $d'$ is a suitable diagonal matrix, we can use
\eqref{eq:func equation for normalization} for $H'$, thereby replacing $f$ by $M^*(s,c,\tau_2,\psi)f$, then apply
\eqref{eq:func equation for normalization} for $\GL_{rkc}^{(m)}$, then again
use \eqref{II    item 1}, \eqref{II item 2}, \eqref{II item 4} and \eqref{eq:func equation for normalization}. This proves \eqref{eq:mult II of ops}.

Applying \eqref{eq:mult II identity 1}--\eqref{eq:mult II identity 2} to the r.h.s.~ of
\eqref{eq:mult II gl part} and using \eqref{eq:mult II of ops} we deduce
\begin{align}\label{eq:mult II identity 3}
&\prod_{i=1}^2\gamma(s,\pi\times\tau_i,\psi)\mathcal{I}(f)
=\pi(\mathfrak{i}_G)^{r\beta_i}\vartheta(s,c,\tau_i,\psi)
\mathcal{I}(M^*(s,c,\tau,\psi)f).
\end{align}
Since $\eta_{\tau_1}\cdot\eta_{\tau_2}=\eta_{\tau}$ by a direct computation using \eqref{eq:(k,c)  functional using an integral}, and
using \eqref{eq:RS mult I} when $m$ is odd,
$\prod_{i=1}^2\vartheta(s,c,\tau_i,\psi)=\vartheta(s,c,\tau,\psi)$. We conclude \eqref{eq:multiplicativity II}.

For $G=\GL_c$, $\tau=\tau_0\otimes\tau_0^{*}$, $\tau_0$ is a quotient of $\Ind_{\widetilde{P}_{\beta}}^{\GL_k^{(m)}}(\varrho_1\otimes\varrho_2)$ and $\tau_i=\varrho_i\otimes\varrho_i^*$.
The formula \eqref{eq:multiplicativity II} for $\GL_c^{(m)}$ reads
$\gamma(s,\pi\times\tau,\psi)=\prod_{i=1}^2\gamma(s,\pi\times\tau_i,\psi)$.
The argument is similar to the above, note that Proposition~\ref{proposition:action of W on torus is trivial on Sp}
is still valid, and the intertwining operator applied in \eqref{eq:mult II gl part} is replaced by
$M^*(s,c,\varrho_1\otimes\varrho_2^*,\psi)M^*(s,c,\varrho_2\otimes\varrho_1^*,\psi)$.

The proof implies the following corollary. Assume
$\rho_c(\tau)$ is an $(rk,c)$ representation and a quotient of $\Ind_{\widetilde{P}_{\beta r c}}^{\GL_{rkc}^{(m)}}(\otimes_{i=1}^d\rho_c(\tau_i))$, with any $d\geq2$ and where $\rho_c(\tau_i)$ are $(r\beta_i,c)$ representations. Let
$V'(s,\rho_c(\tau_d))$ be the space corresponding to the
representation induced from $P'$ to $H'$, where $H'$ and $P'$ are defined as above, for the $G^{(m)}\times\GL_{\beta_d}^{(m)}$ integral involving $\pi\times\tau_d$ (the corollary applies to $G=\GL_c$ as well).
\begin{corollary}\label{corollary:mult II obtaining inner integral}
Let $\omega$ be a matrix coefficient of $\pi^{\vee}$.
For every entire section $f'\in V'(\rho_c(\tau_d))$
there is an entire section $f\in V(\rho_c(\tau))$ such that
$Z(s,\omega,f)=Z(s,\omega,f')$.
\end{corollary}
\begin{proof}
As in \cite[Corollary~33]{CFK} this follows from \eqref{eq:mult II identity 2} using the fact that ${}^{w_L}U_L=U_L^-$ to control the 
unipotent integral, by choosing $f$ such that $\langle w_L,1\rangle\cdot f$ is supported in the preimage of 
$L\mathcal{N}$ in $H^{(m)}$ and right-invariant under $\mathcal{N}$, 
where $\mathcal{N}$ is a small neighborhood of the identity in $H$. Note that such $f$ exists, because for a sufficiently small $\mathcal{N}$, $H^{(m)}$ is split over $\mathcal{N}$.
\end{proof}

\subsection{Multiplicativity I: Identity~\eqref{eq:multiplicativity I}}\label{subsection:Multiplicativity I}

First assume $G=\Sp_{c}$.
The case $l=c$ was proved in \cite[Lemma~79]{me12} for unramified data, although most of the proof is independent of this fact.
However, the general case ($l\leq c$) is more complicated. The linear version of the proof
is found in \cite[\S~6.4.1]{CFK} and we focus on the differences relevant to the covering.

It is enough to assume $R$ is maximal, $\sigma$ (resp., $\pi'$) is a genuine irreducible admissible representation of $\GL_l^{(m)}$ (resp., ${G'}^{(m)}$) and $l\leq n$. Set $\upsilon=\sigma\otimes\pi'$. We may take
$\pi=\Ind_{\widetilde{R}}^{G^{(m)}}(\upsilon)$, then $\pi^{\vee}=\Ind_{\widetilde{R}}^{G^{(m)}}(\upsilon^{\vee})$.
If $\varphi$ (resp., $\varphi^{\vee}$) belongs to the space of $\pi$ (resp., $\pi^{\vee}$), for $g_0,g\in G$,
\begin{align*}
\varphi(\langle g_0,1\rangle\langle g,1\rangle)=\sigma_c(g_0,g)\varphi(\langle g_0g,1\rangle),\qquad
\varphi^{\vee}(\langle g_0,1\rangle\langle g,1\rangle)=\sigma_c^{-1}(g_0,g)\varphi^{\vee}(\langle g_0g,1\rangle).
\end{align*}
($\pi^{\vee}$ is anti-genuine.) Let $\{,\}$ be the canonical $\widetilde{M}_R$-invariant pairing on $\upsilon\otimes\upsilon^{\vee}$. We realize the $G^{(m)}$-invariant pairing on $\pi\times\pi^{\vee}$ using a semi-invariant measure $dg_0$ on $R\backslash G$ (see \cite[1.21]{BZ1}):
\begin{align*}
\{\varphi,\varphi^{\vee}\}=\int\limits_{R\backslash G}\{\varphi(\langle g_0,1\rangle),\varphi^{\vee}(\langle g_0,1\rangle)\}\,dg_0.
\end{align*}
We can assume
\begin{align*}
\omega(\langle g,1\rangle)=\int\limits_{R\backslash G}\{\varphi(\langle g_0,1\rangle),\varphi^{\vee}(\langle g_0,1\rangle \langle g,1\rangle)\}\,dg_0.
\end{align*}

First we replace the $dg$-integral of $Z(s,\omega,f)$ with an iterated integral over $(G^{\triangle}\backslash G\times G)\times (R\backslash G)$, where $G^{\triangle}$ is the diagonal embedding of $G$ in $G\times G$. For $g\in G$,
\begin{align}\label{eq:cocycles for varphi varphi vee}
&\varphi(\langle g_0,1\rangle\langle gg_1,1\rangle)\varphi^{\vee}(\langle g_0,1\rangle\langle gg_2,1\rangle)
=\sigma_c^{-1}(g,g_1)\sigma_c(g,g_2)\varphi(\langle g_0g,1\rangle\langle g_1,1\rangle)\varphi^{\vee}(\langle g_0g,1\rangle\langle g_2,1\rangle).
\end{align}
Also by \eqref{eq:the $2$-cocycle on G times G formula},
\begin{align*}
&f(s,\langle\delta u_0,1\rangle{}^{\iota}\langle(gg_1,gg_2),1\rangle)
=\sigma^{*,rk}_{c}(g,g_1)\sigma_c^{-1}(g,g_2)f(s,\langle\delta u_0,1\rangle{}^{\iota}\langle(g,g),1\rangle)
{}^{\iota}\langle(g_1,g_2),1\rangle).
\end{align*}
Recall $\sigma^{*,rk}_{c}$ and $\sigma_c$ are cohomologous and by \eqref{eq:sigma l * rk}, for $g,g_1\in G$,
\begin{align}\label{eq:sigma * rk c}
\sigma^{*,rk}_c(g,g_1)=\left(\frac{\varsigma_{*,c}(g)\varsigma_{*,c}(g_1)}{\varsigma_{*,c}(gg_1)}\right)^{rk+1}\sigma_c(g,g_1).
\end{align}
Using \eqref{eq:sigma * rk c} to express $\sigma^{*,rk}_c(g,g_1)\sigma_c^{-1}(g,g_1)$ we obtain
\begin{align*}
&\{\varphi(\langle g_0,1\rangle\langle gg_1,1\rangle ),\varphi^{\vee}(\langle g_0,1\rangle\langle gg_2,1\rangle )\}
f(s,\langle\delta u_0,1\rangle{}^{\iota}\langle(gg_1,gg_2),\varsigma_{*,c}^{rk+1}(gg_1)\rangle)
\\&=
\{\varphi(\langle g_0g,1\rangle\langle g_1,1\rangle ),\varphi^{\vee}(\langle g_0g,1\rangle\langle g_2,1\rangle )\}
f(s,\langle\delta u_0,1\rangle{}^{\iota}\langle(g,g),\varsigma_{*,c}^{rk+1}(g)\rangle
{}^{\iota}\langle(g_1,g_2),\varsigma_{*,c}^{rk+1}(g_1)\rangle).
\end{align*}
According to \cite[Corollary~69]{me12}, for any section $f$,
\begin{align*}
&\int\limits_{U_0}
f(s,\langle\delta u_0 ,1\rangle\,
{}^{\iota}\langle(g,g),\varsigma_{*,c}^{rk+1}(g)\rangle)\,\psi_{U}(u_0)\,du_0
=\int\limits_{U_0}
f(s,\langle\delta u_0 ,1\rangle)\,\psi_{U}(u_0)\,du_0.
\end{align*}
Therefore
\begin{align}\label{eq:factoring}
&\int\limits_{U_0}\{\varphi(\langle g_0,1\rangle\langle gg_1,1\rangle ),\varphi^{\vee}(\langle g_0,1\rangle\langle gg_2,1\rangle )\}
f(s,\langle\delta u_0,1\rangle{}^{\iota}\langle(gg_1,gg_2),\varsigma_{*,c}^{rk+1}(gg_1)\rangle)\,\psi_U(u_0)\,du_0
\\&=\nonumber
\int\limits_{U_0}\{\varphi(\langle g_0g,1\rangle\langle g_1,1\rangle ),\varphi^{\vee}(\langle g_0g,1\rangle\langle g_2,1\rangle )\}
f(s,\langle\delta u_0,1\rangle
{}^{\iota}\langle(g_1,g_2),\varsigma_{*,c}^{rk+1}(g_1)\rangle)\,\psi_U(u_0)\,du_0.
\end{align}
Thus we can write $Z(s,\omega,f)$ in the form
\begin{align}\label{eq:mult I start 0}
&\int\limits_{G^{\triangle}\backslash G\times G}
\int\limits_{R\backslash G}\int\limits_{U_0}\{\varphi(\langle g_0,1\rangle\langle g_1,1\rangle),\varphi^{\vee}(\langle g_0,1\rangle\langle g_2,1\rangle)\} \\&f(s,
\langle \delta u_0,1\rangle{}^{\iota}\langle (g_1,g_2),\varsigma_{*,c}^{rk+1}(g_1)\rangle)\,\psi_U(u_0)\,du_0\,dg_0\,d(g_1,g_2).\nonumber
\end{align}
Applying \eqref{eq:factoring} again in the opposite direction with $g_0=I_c$ and $g=g_0$, the integral becomes
\begin{align*}
&\int\limits_{G^{\triangle}\backslash G\times G}
\int\limits_{R\backslash G}\int\limits_{U_0}\{\varphi(\langle g_0g_1,1\rangle),\varphi^{\vee}(\langle g_0g_2,1\rangle)\} \\&f(s,
\langle \delta u_0,1\rangle{}^{\iota}\langle (g_0g_1,g_0g_2),\varsigma_{*,c}^{rk+1}(g_0g_1)\rangle)\,\psi_U(u_0)\,du_0\,dg_0\,d(g_1,g_2).
\end{align*}
The $dg_0$-integral collapses into the $d(g_1,g_2)$-integral, and further factoring through $R$ we have
\begin{align}\label{eq:mult I start}
&\int\limits_{R\times R\backslash G\times G}\,
\int\limits_{M_R}\int\limits_{U_R}\int\limits_{U_0}
\delta_R^{-1/2}(m)
\{\varphi(\langle g_1,1\rangle),\upsilon^{\vee}(m)\varphi^{\vee}(\langle g_2,1\rangle)\}\\&
 \quad f(s,\langle \delta u_0,1\rangle{}^{\iota}\langle\mathfrak{e}_2(zm),1\rangle{}^{\iota}\langle (g_1,g_2),\varsigma_{*,c}^{rk+1}(g_1)\rangle)\,\psi_U(u_0)\,du_0\,dz\,dm\,d(g_1,g_2).\notag
\end{align}
(Cf. \cite[(6.12)]{CFK}.) Here we also used \eqref{eq:sigma on h and v} to separate between $z$ and $mg_2$, and
\eqref{eq:cocycles for varphi varphi vee} and \eqref{eq:the $2$-cocycle on G times G formula} to separate between $m$ and $g_2$ in $f$ and $\varphi^{\vee}$.

By \eqref{eq:multiplicativity II} (or trivially for $k=1$) we may assume $\tau$ is essentially tempered.
Applying Corollary~\ref{corollary:tempered} twice, first to $\rho_c(\tau)$ with $l$, and second to $\rho_{c-l}(\tau)$ with
$c'=c-2l$, and using transitivity of induction, we can assume
$f$ is a section of the space of
\begin{align}\label{rep:triple induced multiplicativity I}
\Ind_{\widetilde{P}}^{H^{(m)}}(|\det|^{r^{-1}(s-1/2)}\Ind_{\widetilde{P}_{(rkl,rkc',rkl)}}^{\GL_{rkc}^{(m)}}((W_{\psi}(\rho_l(\tau))\otimes W_{\psi}(\rho_{c'}(\tau))
\otimes W_{\psi}(\rho_{l}(\tau)))\delta_{P_{(rkl,rkc',rkl)}}^{-1/(2rk)}))
\end{align}
(Cf. \cite[(6.13)]{CFK}). Now for any $h\in H^{(m)}$, the $du_0$-integration of \eqref{eq:mult I start} takes the form
\begin{align}\label{eq:mult 1 identity for normalization 1}
\int\limits_{U_0}\int\limits_{V_1}\int\limits_{V_2}f(s,\langle \diag(I_{rkl},\kappa_{c',l})v_2,1\rangle
\langle\kappa_{l,c-l}v_1,1\rangle \langle\delta u_0,1\rangle h)\,\psi_U(u_0)\,dv_2\,dv_1\,du_0.
\end{align}
(Cf. \cite[(6.14)]{CFK}.)
Here $\kappa_{l,c-l}$ and $V_1$ (resp., $\kappa_{c',l}$ and $V_2$) correspond to the application of \eqref{eq:mnk functional using w_{n,m,k}} to $W_{\psi}(\rho_c(\tau))$ (resp., $W_{\psi}(\rho_{c-l}(\tau))$). Note that we can use \eqref{eq:sigma on h and v} and the fact that $\sigma_{2rkc}$ is trivial on $\mathfrak{W}^+_{2rkc}$ for manipulations involving $\kappa_{l,c-l},\kappa_{c',l},v_1$ and $v_2$.
As in the linear case we can shift $v_1$ and $v_2$ to the right of $u_0$ in \eqref{eq:mult I start}, using
\eqref{eq:sigma on h and v}--\eqref{eq:sigma conjugate v by h}. Set
$V={}^{(\kappa_{l,c-l}\delta_0)^{-1}}V_2\ltimes {}^{\delta_0^{-1}}V_1<V_{(c^{rk})}$ and $\kappa=\diag(I_{rkl},\kappa_{c',l})\kappa_{l,c-l}$.
Then \eqref{eq:mult I start} equals
\begin{align}\label{eq:mult I start factoring through R}
&\int\limits_{R\times R\backslash G\times G}\,
\int\limits_{M_R}\int\limits_{U_R}\int\limits_{V}\int\limits_{U_0}
\delta_R^{-1/2}(m)
\{\varphi(\langle g_1,1\rangle),\upsilon^{\vee}(m)\varphi^{\vee}(\langle g_2,1\rangle)\}\\&
 \quad f(s,\langle\kappa\delta u_0v,1\rangle{}^{\iota}\langle\mathfrak{e}_2(zm),1\rangle{}^{\iota}\langle(g_1,g_2),\varsigma_{*,c}^{rk+1}(g_1)\rangle)\,\psi_U(u_0)\,du_0\,dv\,dz\,dm\,d(g_1,g_2).\notag
\end{align}
(Cf. \cite[(6.16)]{CFK}.) Here all conjugations involved take elements of $N_H$ into elements of $N_H$.

Next we shift $\langle\mathfrak{e}_2(z),1\rangle$ to the left, and use the coordinates of $z$ and $u_0$ to form a unipotent subgroup $U^{\circ}<U_P$. Part of the coordinates of $U^{\circ}$, including those from $z$, will be used for the application of an intertwining operator
$m(s,\tau,w)$ below. Conjugating $\delta_1u_0v$ by ${}^{\iota}\mathfrak{e}_2(z)$ we write
\begin{align*}
\kappa\delta u_0v(1,{}^{\iota}z)={}^{\kappa\delta_0}(1,{}^{\iota}z)\,\kappa\,\delta_0\, x_z\,\delta_1\,u_z\,r_z\,a_{u_0,z}\,b_z\,v,
\end{align*}
where $x_z,u_z,r_z,a_{u_0,z},b_z\in N_H$ (the precise description of these elements was given in \cite{CFK}).
This identity holds in $H^{(m)}$ as well, by \eqref{eq:sigma conjugate v by h} (note that $\delta_1,u_0,v\in N_H$) and \eqref{eq:sigma on h and v}. Also while ${}^{\iota}\mathfrak{e}_2(z)\in N_H^-$, ${}^{\kappa\delta_0}({}^{\iota}\mathfrak{e}_2(z))\in N_H$, and moreover ${}^{\kappa\delta_0}(1,{}^{\iota}z)\in V_{(rkl,rk(c-l))}\ltimes U_P$ so that $h\mapsto f(s,h)$
is left-invariant under ${}^{\kappa\delta_0}({}^{\iota}\langle \mathfrak{e}_2(z),1\rangle)$. Now we can apply the steps
\cite[p.~48, (2)--(4)]{CFK} which include a change of variables and using the equivariance properties of $f$ under $V_{(c^{rk})}\ltimes U_P$,
and arrive at the analog of \cite[(6.19)]{CFK}:
\begin{align}\label{int:before fixing g_1 and g_2}
&\int\limits_{R\times R\backslash G\times G}\,
\int\limits_{M_R}\int\limits_{V}\int\limits_{U^{\circ}}
\delta_R^{-1/2}(m)
\{\varphi(\langle g_1,1\rangle),\upsilon^{\vee}(m)\varphi^{\vee}(\langle g_2,1\rangle)\}\\&
 \quad f(s,\langle \kappa\delta uv,1\rangle{}^{\iota}\langle\mathfrak{e}_2(m),1\rangle{}^{\iota}\langle(g_1,g_2),\varsigma_{*,c}^{rk+1}(g_1)\rangle)\,\psi_U(u)\,du\,dv\,dm\,d(g_1,g_2).\notag
\end{align}
Here $U^{\circ}>U_0$ is the subgroup of elements $u_zr_z$ (see $U^{\bullet}$ below); $\psi_U$ is extended trivially to $U^{\circ}$.

Let $H^{\sigma}=\GL_{2rkl}$, $U_0^{\sigma}$, $P^{\sigma}$ and $\delta^{\sigma}=\delta_0^{\sigma}\delta_1^{\sigma}$ be the groups and elements defined in \S~\ref{The integrals} for the $\GL_l^{(m)}\times\GL_k^{(m)}$ integral, and
$H',P',U_0'$ and $\delta'=\delta_0'\delta_1'$ be the notation for the ${G'}^{(m)}\times\GL_k^{(m)}$ integral.
Fix the standard parabolic subgroup $L<H$ with $M_L=H^{\sigma}\times H'$, and identify $H^{\sigma}$ and $H'$ with the factors of $M_L$.

Denote $\kappa^{\bullet}={}^{\delta_0^{-1}}\kappa=\diag(\kappa_{l,c'},I_{rkl})\kappa_{c-l,l}$ and
$U^{\bullet}={}^{\kappa^{\bullet}}U^{\circ}<U_P$. Denote the top right $rkc\times rkc$ block of elements of $U^{\bullet}$ by $(u^{i,j})_{1\leq i,j\leq3}$. Then
$\left(\begin{smallmatrix}I_{rkl}&u^{1,1}\\&I_{rkl}\end{smallmatrix}\right)$ (resp., $\left(\begin{smallmatrix}I_{rkc'}&u^{2,2}\\&I_{rkc'}\end{smallmatrix}\right)$) is a general element of $U_0^{\sigma}$ (resp., $U_0'$), $u^{2,1}\in\Mat_{rkc'\times rkl}$ (resp., $u^{3,1}\in\Mat_{rkl}$) and its bottom left $c'\times l$ (resp., $l\times l$) block is $0$. The remaining blocks take arbitrary coordinates but such that $U^{\bullet}<H$.
In addition $\psi_U$ restricts to $\psi_{U_0^{\sigma}}$ (resp., $\psi_{U_0'}$) on $u^{1,1}$ (resp., $u^{2,2}$) and is trivial on the other coordinates.

Write $\delta_0=w^{-1}\delta_0'\delta_0^{\sigma}w$ ($\delta_0'\in H'<M_L$), where
\begin{align}\label{eq:mult 1 w^-1}
w^{-1}=\diag(I_{rkl},\left(\begin{smallmatrix}&I_{rkc'}\\&&&I_{rkl}\\ \epsilon_0I_{rkl}\\&&I_{rkc'}\end{smallmatrix}\right),I_{rkl}).
\end{align}
Then ${}^{w}({}^{(\delta_0^{-1}\kappa\delta_0)}\delta_1)=\delta_1^{\sigma}\delta_1'$. Let $[u^{i,j}]$ be the subgroup of $U^{\bullet}$ generated by elements whose coordinates $u^{t,t'}$ are zeroed out for $(t,t')\ne(i,j)$,
\begin{align*}
U_0^{\sigma}={}^{w}[u^{1,1},u^{3,3}],\quad U_0'={}^{\delta^{\sigma}w}[u^{2,2}], \quad
Z={}^{\delta'\delta^{\sigma}w}[u^{1,2},u^{1,3},u^{2,3}],\quad
O=[u^{2,1},u^{3,1},u^{3,2}].
\end{align*}
We will write the integration $du$ as an iterated integral according to these subgroups.

Let $L_0$ be the standard parabolic subgroup of $H$ whose Levi part is $M_{(rkl,rkc',rkl)}$.
Denote
\begin{align*}
m(s,\tau,w)f(s,h)=\int\limits_{Z}f(s,\langle w^{-1}z,1\rangle h)dz\qquad
(Z=\left\{\diag(I_{rkl},\left(\begin{smallmatrix}I_{rkl}&z_1&&z_2\\&I_{rkc'}\\&&I_{rkc'}&z_1^*\\&&&I_{rkl}\end{smallmatrix}\right),I_{rkl})\in H\right\}).
\end{align*}
This is a standard intertwining operator from the space of the representation of $H^{(m)}$ induced from
$\widetilde{L}_0$ and the representation of $\widetilde{P}_{(rkl,rkc',rkl)}$ appearing in \eqref{rep:triple induced multiplicativity I}, to
\begin{align}\label{space:rep induced mult 0}
\Ind_{\widetilde{L}}^{H^{(m)}}\left(
\delta_{L}^{-1/2}
\left(|\det|^dV(s,W_{\psi}(\rho_l(\tau))\otimes W_{\psi}(\rho_l(\tau^*)))\,\otimes\,V(s,W_{\psi}(\rho_{c'}(\tau)))\right)\right).
\end{align}
Here $d=(rk-1/2)(c-l)+1/2$. This follows from the decomposition
\begin{align*}
w^{-1}z=
\left(\begin{smallmatrix}I_{rkc'}\\&&I_{rkl}\\&\epsilon_0I_{rkl}\\&&&I_{rkc'}\end{smallmatrix}\right)
\left(\begin{smallmatrix}I_{rkc'}\\&I_{rkl}&z_2\\&&I_{rkl}\\&&&I_{rkc'}\end{smallmatrix}\right)
\left(\begin{smallmatrix}&I_{rkc'}\\I_{rkl}\\&&&I_{rkl}\\&&I_{rkc'}\end{smallmatrix}\right)
\left(\begin{smallmatrix}I_{rkl}&z_1\\&I_{rkc'}\\&&I_{rkc'}&z_1^*\\&&&I_{rkl}\end{smallmatrix}\right).
\end{align*}

Returning to \eqref{int:before fixing g_1 and g_2}, we obtain
\begin{align}\label{int:after fixing g_1 and g_2}
&\int\limits_{R\times R\backslash G\times G}\,
\int\limits_{M_R}
\int\limits_{V}\int\limits_{O}
\int\limits_{U_0^{\sigma}}\int\limits_{U_0'}
\delta_R^{-1/2}(m)
\{\varphi(\langle g_1,1\rangle),\upsilon^{\vee}(m)\varphi^{\vee}(\langle g_2,1\rangle)\}\\&\nonumber
m(s,\tau,w)f(s,\langle \delta'u',1\rangle\langle\delta^{\sigma}u^{\sigma},1\rangle\langle  w_1o,1\rangle
\,\langle\kappa^{\bullet} v ,1\rangle{}^{\iota}\langle\mathfrak{e}_2(m),1\rangle{}^{\iota}\langle(g_1,g_2),\varsigma_{*,c}^{rk+1}(g_1)\rangle)
\,\psi_{U'}(u')\psi_{U^{\sigma}}(u^{\sigma})
\\&du'\,du^{\sigma}\,do\,dv\,dm\,d(g_1,g_2).\notag
\end{align}
(Cf. \cite[(6.22)]{CFK}.)
Let $m=\diag(a,g',a^*)\in M_R$ ($a\in\GL_l$, $g'\in G'$). Then by \eqref{eq:BLS block compatible} in $G^{(m)}$,
\begin{align*}
\langle m,1\rangle=\langle \diag(a,I_{c'},a^*),1\rangle\langle \diag(I_l,g',I_{l}),1\rangle
\end{align*}
and a similar identity holds in $H^{(m)}$ for the image of $M_R$ under $\mathfrak{e}_2$. Hence we can consider $a$ and $g'$ separately.

Observe that $\kappa^{\bullet}\in\{\diag(x,I_{2l},x^*):x\in \GL_{rkc-l}\}<H$.
Hence ${}^{\iota}\mathfrak{e}_2(\diag(a,I_{c'},a^*))$ commutes with $\kappa^{\bullet}$ in $H$, and also in
$H^{(m)}$ by \eqref{eq:BLS block compatible}, noting that as an element of $\GL_{rkc-l}$, $\det\kappa^{\bullet}=\mp1$ so that
$(\det a,\det \kappa^{\bullet})_m=1$. Additionally ${}^{\iota}\mathfrak{e}_2(\diag(a,I_{c'},a^*))$ commutes with $v$, normalizes $O$ (with a change of measure $|\det{a}|^{(1-rk)(c-l)}$) and
\begin{align*}
{}^{w}({}^{\iota}\mathfrak{e}_2(\diag(a,I_{c'},a^*)))=\diag(I_{rkl},a,I_{2(rkc-rkl-l)},a^*,I_{rkl})=\mathfrak{e}_2^{\sigma}(a),
\end{align*}
where $\mathfrak{e}_2^{\sigma}$ is the embedding of the right copy of $\GL_l$ in $H^{\sigma}$ for the
$\GL_l^{(m)}\times\GL_k^{(m)}$ integral. Then by Lemma~\ref{lemma:conjugation commutes},
\begin{align}\label{eq:a becomes a sigma}
{}^{w}({}^{\iota}\langle\mathfrak{e}_2(\diag(a,I_{c'},a^*)),1\rangle)=\langle\mathfrak{e}_2^{\sigma}(a),1\rangle.
\end{align}
(See also \cite[(4.20)]{me12}, proved for $l=n$ but the argument applies to $l\leq n$.)

Regarding $\mathfrak{e}_2(\diag(I_l,g',I_{l}))$, as observed in \cite{CFK} it does not normalize
$V$ nor $O$, but this was handled by combining a subgroup of $O$ with a subgroup of
$V$. The manipulations from \textit{loc. cit.} extend to $H^{(m)}$ using \eqref{eq:sigma on h and v} and
\eqref{eq:sigma conjugate v by h}. Since ${}^{w\kappa^{\bullet}}({}^{\iota}\mathfrak{e}_2(\diag(I_l,g',I_{l})))={}^{\iota'}\mathfrak{e}_2'(g')$,
${}^{\iota'}({}^{w\kappa^{\bullet}}({}^{\iota}\mathfrak{e}_2(\diag(I_l,g',I_{l}))))=\mathfrak{e}_2'(g')$ and
Lemma~\ref{lemma:conjugation commutes} implies
\begin{align*}
{}^{\iota'}({}^{w\kappa^{\bullet}}({}^{\iota}\langle\mathfrak{e}_2(\diag(I_l,g',I_{l})),1\rangle))=\langle\mathfrak{e}_2'(g'),1\rangle,
\end{align*}
whence the analog of \eqref{eq:a becomes a sigma} also holds:
\begin{align*}
{}^{w\kappa^{\bullet}}({}^{\iota}\langle\mathfrak{e}_2(\diag(I_l,g',I_{l})),1\rangle)={}^{\iota'}\langle\mathfrak{e}_2'(g'),1\rangle.
\end{align*}

Thus \eqref{int:after fixing g_1 and g_2} equals
\begin{align}\label{int:before functional eq mult I}
&\mathcal{I}(m(s,\tau,w)f)\\\nonumber&=\int\limits_{R\times R\backslash G\times G}\,
\int\limits_{V}\int\limits_{O}
\int\limits_{\GL_l}
\int\limits_{U_0^{\sigma}}
\int\limits_{G'}
\int\limits_{U_0'}
\delta_R^{-1/2}(a)|\det{a}|^{(1-rk)(c-l)}\\\nonumber
&\{\varphi(\langle g_1,1\rangle),\sigma^{\vee}(\langle a,1\rangle)\otimes{\pi'}^{\vee}(\langle g',1\rangle)\varphi^{\vee}(\langle g_2,1\rangle)\}\\&\nonumber
m(s,\tau,w)f(s,(\langle\delta'u',1\rangle{}^{\iota'}\langle\mathfrak{e}_2'(g'),1\rangle)\,(\langle\delta^{\sigma}u^{\sigma},1\rangle\langle
\mathfrak{e}_2^{\sigma}(a),1\rangle)\,\langle w_1o,1\rangle\langle\kappa^{\bullet} v,1\rangle{}^{\iota}\langle(g_1,g_2),\varsigma_{*,c}^{rk+1}(g_1)\rangle)
\\&\psi_{U'}(u')\psi_{U^{\sigma}}(u^{\sigma})\,du'\,dg'\,du^{\sigma}\,da\,do\,dv\,d(g_1,g_2)\notag\nonumber.
\end{align}
(Cf. \cite[(6.24)]{CFK}.)
Note that $\delta_R^{-1/2}(a)|\det{a}|^{(1-rk)(c-l)}=|\det{a}|^{-d}$.
The $du^{\sigma}da$-integral is the $\GL_l^{(m)}\times\GL_k^{(m)}$ integral for
$\sigma\times(\tau\otimes\tau^*)$ and the $du'dg'$-integral is the ${G'}^{(m)}\times\GL_k^{(m)}$ integral for $\pi'\times\tau$.
Therefore when we multiply $Z(s,\omega,f)$ by the appropriate $\gamma$-factors we deduce
\begin{align}\label{intert 1}
&\gamma(s,\sigma\times(\tau\otimes\tau^*),\psi)\gamma(s,\pi'\times\tau,\psi)Z(s,\omega,f)\\&=\nonumber
\sigma(\mathfrak{i}_{\GL_l})^{rk}
\tau(\langle r^rI_k,1\rangle)^{-2l}|r|^{-2kl(s-1/2)}
\tau(\mathfrak{i}_{\GL_k})^{rl}\pi'(\mathfrak{i}_{G'})^{rk}\vartheta(s,c',\tau,\psi)\\&\quad\times\mathcal{I}(M^*(s,l,\tau\otimes\tau^*,\psi)M^*(s,c',\tau,\psi)m(s,\tau,w)f).\nonumber
\end{align}
Here if $l=n$ and $m$ is even, we set $\vartheta(s,c',\tau,\psi)=1$, and if $l=n$ and $m$ is odd, 
$\vartheta(s,c',\tau,\psi)=\gamma(s,\tau,\psi)(=\gamma(s,\pi'\times\tau,\psi))$. 
This step is justified as in the linear case (see \cite[(6.25)]{CFK}). Applying the same manipulations
\eqref{eq:mult I start 0}--\eqref{int:before functional eq mult I} to
$Z^*(s,\omega,f)$ yields
\begin{align*}
Z^*(s,\omega,f)=\mathcal{I}(m(1-s,\tau^*,w)M^*(s,c,\tau,\psi)f).
\end{align*}
For any $b\in F^*$, set $C(b)=\eta_{\tau}(\langle b,1\rangle)^{2l}|b|^{2kl(s-1/2)}$. Next we claim
\begin{align}\label{proportionality between inter}
M^*(s,l,\tau\otimes\tau^*,\psi)M^*(s,c',\tau,\psi)m(s,\tau,w)=C(1/2)m(1-s,\tau^*,w)M^*(s,c,\tau,\psi).
\end{align}
This identity follows as in \cite[(6.26)]{CFK}. Specifically, both sides are proportional (the uniqueness result in \cite[pp.~51--52]{CFK}
extends to our context by replacing $k$ with $rk$), and the proportionality factor can be computed
by expressing $\lambda(s,c,\tau,\psi)$ as an iterated integral and applying \eqref{eq:func equation for normalization}. In  more detail,  repeating the part of the above argument concerning unipotent integrations and Weyl elements (e.g., ignoring the integral over $G\times G$), we see that $\lambda(s,c,\tau,\psi)$ can be evaluated as the composition of $\lambda_{2}(s,l,\tau\otimes\tau^*,\psi)$ and $\lambda(s,c',\tau,\psi)$, where $\lambda_{2}(\cdots)$ is given by \eqref{eq:lambda functional} except the character $\psi_{rk,c}^{-1}$ appearing in \eqref{eq:lambda functional} is replaced with $(\psi_2)_{rk,c}^{-1}$ (but $\rho_l(\tau)$ is still realized in $W_{\psi}(\rho_l(\tau))$). See \cite[(6.30)--(6.32)]{CFK}. The factor $C(2)$ ($C(2)^{-1}=C(1/2)$) is obtained when in \eqref{eq:func equation for normalization}, $f$ is replaced with its right translate by $\langle\diag(2I_{rkl},I_{rkl}),1\rangle$ and we use Proposition~\ref{proposition:action of W on torus is trivial on Sp}, and note that $\eta_{\tau,l}=\eta_{\tau}^l$ and $2rkl(r^{-1}(s-1/2))=2kl(s-1/2)$ (see \cite[(6.31)]{CFK}).

Since $\mathfrak{i}_G=\mathfrak{i}_{\GL_l}\mathfrak{i}_{G'}$ in $M_L$ (by \eqref{eq:BLS block compatible}), $\pi(\mathfrak{i}_G)=\sigma(\mathfrak{i}_{\GL_l})\pi'(\mathfrak{i}_{G'})$. Also
\begin{align*}
C(1/2)\tau(\langle r^rI_k,1\rangle)^{-2l}|r|^{-2kl(s-1/2)}\tau(\mathfrak{i}_{\GL_k})^{rl}\vartheta(s,c',\tau,\psi)=\vartheta(s,c,\tau,\psi).
\end{align*}
(With the aforementioned convention for $\vartheta(s,c',\tau,\psi)$ when $l=n$.) 
We conclude
\begin{align*}
\gamma(s,\sigma\times(\tau\otimes\tau^*),\psi)\gamma(s,\pi'\times\tau,\psi)Z(s,\omega,f)=\pi(\mathfrak{i}_G)^{rk}\vartheta(s,c,\tau,\psi)Z^*(s,\omega,f).
\end{align*}

Assume $G=\GL_c$. The argument was carried out in \cite[Lemma~80]{me12} in the unramified case, but except for the choice
of realization for the matrix coefficient $\omega$, the proof was applicable in the ramified or archimedean cases as well. Moreover, as opposed to the symplectic case \cite[Lemma~80]{me12} already treated all $l\leq c$, so we can be more brief. See also \cite[\S~6.4.5]{CFK}.

Recall $\tau=\tau_0\otimes\tau_0^*$ and we assume $\tau_0$ (and thereby $\tau_0^*$) is essentially tempered. Now $R=P_{(l,c-l)}$, $\upsilon=\sigma_1\otimes\sigma_2$ is a representation of $\GL_l^{(m)}\times\GL_{c-l}^{(m)}$ and $\pi=\Ind_{\widetilde{R}}^{G^{(m)}}(\upsilon)$.
Formula \eqref{eq:multiplicativity I} reads
$\gamma(s,\pi\times\tau,\psi)=\prod_{i=1}^2\gamma(s,\sigma_i\times\tau,\psi)$.
Since $\pi^{\vee}=\Ind_{\widetilde{R}}^{G^{(m)}}(\upsilon^{\vee})$, we again realize $\{\varphi,\varphi^{\vee}\}$
using a semi-invariant measure on $R\backslash G$,
\begin{align*}
\omega(\langle g,1\rangle)=\int\limits_{R\backslash G}\{\varphi(\langle g_0,1\rangle),\varphi^{\vee}(\langle g_0,1\rangle \langle g,1\rangle)\}\,dg_0.
\end{align*}
For $g\in G$, we have \eqref{eq:cocycles for varphi varphi vee} with $\sigma_c^{\diamondsuit}$ instead of
$\sigma_c$, and by \eqref{eq:the $2$-cocycle on G times G formula GL},
\begin{align*}
&f(s,\langle\delta u_0,1\rangle\langle(gg_1,gg_2),1\rangle)
=\sigma^{\diamondsuit}_{c}(g,g_1)\sigma^{\diamondsuit}_c(g,g_2)^{-1}f(s,\langle\delta u_0,1\rangle\langle(g,g),1\rangle)
\langle(g_1,g_2),1\rangle).
\end{align*}
Hence
\begin{align*}
&\{\varphi(\langle g_0,1\rangle\langle gg_1,1\rangle ),\varphi^{\vee}(\langle g_0,1\rangle\langle gg_2,1\rangle )\}
f(s,\langle\delta u_0,1\rangle\langle(gg_1,gg_2),1\rangle)
\\&=
\{\varphi(\langle g_0g,1\rangle\langle g_1,1\rangle ),\varphi^{\vee}(\langle g_0g,1\rangle\langle g_2,1\rangle )\}
f(s,\langle\delta u_0,1\rangle\langle(g,g),1\rangle
\langle(g_1,g_2),1\rangle).
\end{align*}
According to \cite[Corollary~76]{me12},
\begin{align*}
&\int\limits_{U_0}
f(s,\langle\delta u_0 ,1\rangle\,
\langle(g,g),1\rangle)\,\psi_{U}(u_0)\,du_0
=\int\limits_{U_0}
f(s,\langle\delta u_0 ,1\rangle)\,\psi_{U}(u_0)\,du_0.
\end{align*}
Thus we obtain a formula similar to \eqref{eq:factoring} (without $\iota$ and $\varsigma_{*,c}^{rk+1}$), which we use exactly as above
to write $dg$ over $G^{\triangle}\backslash G\times G$, then collapse $dg_0$ into $d(g_1,g_2)$ and obtain the analog of
\eqref{eq:mult I start}.
Now apply the arguments for \cite[(4.33)--(4.44)]{me12} (we can also use Lemma~\ref{lemma:conjugation commutes}) and obtain
\begin{align}\label{int:GL after 7 props 6 1}
&\mathcal{I}(m(s,\tau,w)f)\\&\nonumber=
\int\limits_{R\times R\backslash G\times G}\,
\int\limits_{V}\int\limits_{V}\int\limits_{U^3}
\int\limits_{\GL_b}
\int\limits_{U^4}
\int\limits_{\GL_a}
\int\limits_{U^1}
\omega'(\langle x,1\rangle)\omega''(\langle y,1\rangle
|\det{x}|^{b/2-rkb}|\det{y}|^{rka-a/2}\\&\quad\notag m(s,\tau,w)f(s,
\langle \delta'u^1,1\rangle\langle \mathfrak{e}_2'(x),1\rangle
\langle \delta''u^4,1\rangle\langle \mathfrak{e}_2''(y),1\rangle
\langle w^{-1}u^3,1\rangle
\langle k^{\bullet}\diag(v',v),1\rangle\langle (g_1,g_2)\rangle)\nonumber\\&\quad\psi_{U}(u^1)\psi_{U}(u^4)
\,du^1\,dx\,du^4\,dy\,du^3\,dv\,dv'\,d(g_1,g_2).\notag
\end{align}
(Cf. \eqref{int:before functional eq mult I} and \cite[(4.44)]{me12}.) Here
$\omega'$, $\delta'$ and $\mathfrak{e}_2'$ (resp., $\omega''$, $\delta''$ and $\mathfrak{e}_2''$) correspond to the
$\GL_l^{(m)}\times\GL_{k}^{(m)}$ (resp., $\GL_{c-l}^{(m)}\times\GL_{k}^{(m)}$) integral for
$\sigma_1\times(\tau_0\otimes\tau_0^*)$ (resp., $\sigma_2\times(\tau_0\otimes\tau_0^*)$), and $U^3$ plays the role of $O$
in \eqref{int:before functional eq mult I} ($U^3$ was the notation of \cite[(4.44)]{me12}).

The analog of \eqref{intert 1} takes the form
\begin{align*}
&\gamma(s,\sigma_1\times\tau,\psi)\gamma(s,\sigma_2\times\tau,\psi)
Z(s,\omega,f)\\&=\nonumber
\pi(\mathfrak{i}_G)^{rk}\vartheta(s,c,\tau,\psi)\mathcal{I}(M^*(s,l,\tau,\psi)
M^*(s,c-l,\tau,\psi)m(s,\tau,w)f).
\end{align*}
The proof is then complete once we prove the analog of \eqref{proportionality between inter}:
\begin{align*}
M^*(s,l,\tau,\psi)
M^*(s,c-l,\tau,\psi)m(s,\tau,w)=m(1-s,\tau^*,w)M^*(s,c,\tau,\psi),
\end{align*}
which is similar and simpler, and note that there is no additional twist of $\psi$ by $2$ here.

\subsubsection{The $\GL_c^{(m)}\times\GL_k^{(m)}$-factors}\label{GL(n) factors example n=1}
In the linear case the analog of \eqref{eq:GL factors} was derived directly in \cite[\S~6.10.2]{CFK} using the $\gamma$-factor of \cite{JPSS}, which is relevant in the context of the doubling integrals only for $c=1$. The general case was then obtained using other properties including multiplicativity and globalization (\cite[\S~6.10.1]{CFK}). In the context of covering groups we have to argue locally, thus we have to treat all $c$ simultaneously, which is possible because here \eqref{eq:GL factors} involves the $\gamma$-factor defined in \S~\ref{RS integrals}.
Nonetheless, the argument closely follows the methods of \cite[\S~3.7]{CFGK2}, \cite[\S~6.10.2]{CFK} and
\cite[\S~4.3]{me12}, the latter of which treated the $\GL_1^{(m)}\times\GL_k^{(m)}$ case in the unramified setting.

For $rk=1$ \eqref{eq:GL factors} follows from \cite{Gan} (see also \cite[(4.62), (4.63)]{me12}), hence in this section $rk>1$.
We reformulate several results from \cite[\S~4.3]{me12} with arbitrary $c$, which are applicable also when data are not unramified.
The section $f$ is on the representation
\begin{align}\label{eq:ind sections}
\mathrm{I}(s,W_{\psi}(\rho_c(\tau)))=\Ind_{\widetilde{P}}^{H^{(m)}}(|\det|^{r^{-1}(s-1/2)}W_{\psi}(\rho_c(\tau_0))\otimes |\det|^{-r^{-1}((s-1/2)}W_{\psi}(\rho_c(\tau_0^*))).
\end{align}
Write $v\in V_{(rkc,rkc)}$ in the form
\begin{align*}
[\begin{smallmatrix}y&z\\u&x\end{smallmatrix}]=\begin{pmatrix}I_{(rk-1)c}&&y&z\\&I_c&u&x\\&&I_c\\&&&I_{(rk-1)c}\end{pmatrix}.
\end{align*}
With this notation $U_0=\{[\begin{smallmatrix}y&z\\0&x\end{smallmatrix}]\}$ and $\psi_U([\begin{smallmatrix}y&z\\0&x\end{smallmatrix}])=\psi(\tr x_1)$, where $x=(x_1,\ldots,x_{rk-1})$ with $x_i\in\Mat_c$.
By Proposition~\ref{prop:rodier non}, $\mathrm{I}(s,W_{\psi}(\rho_c(\tau)))$ admits a unique $(2rk,c)$ model. We fix a character of $V_{(c^{2rk})}<N_H$ by
\begin{align}\label{eq:character for Ind GL 1}
\psi(\diag(d,d')[\begin{smallmatrix}y&z\\u&x\end{smallmatrix}])=\psi(\tr (\sum_{i=1}^{rk-1} d_{i,i+1}-u+\sum_{i=1}^{rk-1} d'_{i,i+1})),\qquad d,d'\in V_{(c^{rk})}.
\end{align}
(Cf. \cite[(4.51)]{me12}.)
The corresponding $(2rk,c)$ model of \eqref{eq:ind sections} is spanned by the functions
\begin{align}\label{int:standard Whittaker on I W W}
W_f(s,h)=\int\limits_{V_{(rkc,rkc)}} f(s,\langle\delta_0[\begin{smallmatrix}y&z\\u&x\end{smallmatrix}],1\rangle h)\psi(\tr u)\,dx\,dy\,dz\,du.
\end{align}
The $(2rk,c)$ model of $\mathrm{I}(s,W_{\psi}(\rho_c(\tau)))^*$ with respect to the inverse of
\eqref{eq:character for Ind GL 1} is spanned by
\begin{align}\label{int:standard Whittaker on I W W *}
W_f^*(s,h)=\int\limits_{V_{(rkc,rkc)}} f(s,\langle\delta_0[\begin{smallmatrix}y&z\\u&x\end{smallmatrix}],1\rangle\,{}^*h)\psi(\tr u)\,dx\,dy\,dz\,du.
\end{align}
Denote
\begin{align*}
[t,v]=\diag(I_{rkc},\left(\begin{smallmatrix}I_c&&\\&I_{(rk-2)c}&\\-t&v&I_c\end{smallmatrix}\right)),\qquad
w'=\left(\begin{smallmatrix} I_{rkc} \\ & & I_{(rk-1)c} \\ & I_c \end{smallmatrix}\right).
\end{align*}
Let $v^-\mapsto\langle v^-,\varsigma^-(v^-)\rangle$ be the splitting of $N_{H}^-$ in $H^{(m)}$. Also let $\zeta\in\C$. The study of the $\GL_c^{(m)}\times\GL_k^{(m)}$ case is based on the following integral, defined
for a matrix coefficient $\omega$ of $\pi^{\vee}$ and a holomorphic section $f$ of \eqref{eq:ind sections}:
\begin{align}\label{int:after functional equation to compare}
\Psi(\zeta,s,\omega,f)=&\int\limits_{\GL_c}\int\limits_{\Mat_{((rk-2)c)\times c}}\int\limits_{\Mat_c} W_f(s,\langle\diag(I_{(2rk-1)c},a),1\rangle\langle [t,v],\varsigma^-([t,v])\rangle
\langle w',1\rangle)\\\nonumber&\omega(\langle a,1\rangle)|\det a|^{\zeta+(rk-1)c}\,dt\,dv\,da.
\end{align}
Assume $F$ is non-archimedean.
This integral is formally well defined and there are constants $B,D\in\R$ such that
\eqref{int:after functional equation to compare} is absolutely convergent for $\Real(\zeta)\leq B\Real(s)+D$ for all data $(\omega,f)$. In its domain of convergence, the integral belongs to \eqref{Homspace} with $\pi$ replaced by $|~|^{-\zeta}\pi$. Consequently, it admits meromorphic continuation to a function in $\C(q^{-\zeta},q^{-s})$, and moreover
outside finitely many values of $q^{-s}$, the continuation with $\zeta=0$ belongs to \eqref{Homspace}. To prove all of these observations
we repeat the arguments in the proof of \cite[Proposition~83]{me12} but for $c>1$, with the following additional remarks.

For the convergence note that after dealing with the unipotent integrals, in order to establish the convergence of the integral over $\GL_c$ we use the invariance properties of $W$, with respect to a small compact open neighborhood of the identity, to pass to an integral over $T_{\GL_c}$. Then the diagonal coordinates will each be bounded from below.

For the equivariance properties with respect to $e_1(g_1)$ note that $e_1(g_1)$ commutes with $w'$ also in $H^{(m)}$ by
Lemma~\ref{lemma:conjugation commutes}; $e_1(g_1)$ normalizes the subgroup of elements $[t,v]$ without affecting the measure;
although $(g_1,g_1)$ does not belong to $C_H$ when $c>1$, it does commute
with $\delta_0$ in $H^{(m)}$ by Lemma~\ref{lemma:conjugation commutes}, and normalizes $V_{(rkc,rkc)}$ without changing the measure.
It then follows that $W_f(s,\langle(g_1,g_1),1\rangle h)=W_f(s,h)$ by \cite[Proposition~77]{me12}.

Regarding the meromorphic continuation, for the application of \cite{Banks} we also need to show the integral can be made nonzero, and to this end in the proof of \cite[Proposition~83]{me12} we chose a function in \eqref{eq:ind sections}, whose restriction to 
the preimage of the so-called mirabolic subgroup of $\GL_{2rk}$ was supported in a small compact open neighborhood of the identity. This was obtained by the analog of \cite[5.15]{BZ1} for coverings.
For a general $(rk,c)$ model the same result applies with the preimage of the subgroup of elements $\{\left(\begin{smallmatrix}g'&v\\&I_c\end{smallmatrix}\right):g'\in\GL_{(2rk-1)c}\}$, see \cite[Lemma~3.11, Corollary~3.12]{LapidMao2018} (the Kirillov--Shalika model).

Over $F=\C$, to obtain the convergence, meromorphic continuation and continuity of the continuation
of \eqref{int:after functional equation to compare} note that after handling the unipotent integration,
the remaining integral is similar to the integral of \cite[Appendix~C]{CFK} with the $(2rk,c)$ representation \eqref{eq:ind sections} instead of the $(k,c)$ representation there.

Consequently by \cite[Theorem~3.1]{DimaKaplan}, \eqref{int:after functional equation to compare} is proportional to $Z(s,\omega,f)$.

\begin{proposition}\label{proposition:substitution}
We have $\gamma(s,\widetilde{\pi}\times\tau_0,\psi)\pi(\mathfrak{i}_{\GL_c})^{rk-1}\vartheta(s,\widetilde{\pi},\tau)^{-1}Z(s,\omega,f)=
\Psi(0,s,\omega,f)$, as meromorphic continuations.
\end{proposition}
\begin{proof}
The proof technique is due to Soudry \cite{Soudry,Soudry3}.
The linear analog of this result for $c=1$ was proved in \cite[Claim~37]{CFK}, and the covering version for the unramified case
in \cite[Lemma~85]{me12} (for the linear unramified version see \cite[Claim~36]{CFGK2}).
We reformulate the argument for all $c$.

Define the following integral:
\begin{align*}
Z'(s,\omega,f)=
\int\limits_{\GL_c}\omega(\langle a,1\rangle)
\int f(s,\langle\delta_0[\begin{smallmatrix}y&z\\u&x\end{smallmatrix}]\jmath(t),1\rangle\langle\mathfrak{e}_2(a),1\rangle)\psi(\tr x_1)\,dx\,dy\,dz\,du\,dt\,da.
\end{align*}
Here $\jmath(t)=\diag(I_{rkc},\left(\begin{smallmatrix}I_c&-t\\&I_c\end{smallmatrix}\right),I_{(rk-2)c})\in P\cap N_{H}$ with $t\in\Mat_c$.
This integral is defined in the sense of \cite[(4.65)]{me12}, can be regarded as an element of \eqref{Homspace} and $Z'(s,\omega,f)=Z(s,\omega,f)$ (as meromorphic continuations). See \cite[Lemma~85]{me12}, the main observation
used in \textit{loc. cit.} was that $\int_{\Mat_c}\psi(\tr((u-I_c)t))dt=0$ unless $u=I_c$, and this is valid for all $c$.
The only difference to note is that, while the splitting of $\mathcal{N}_l=\{[\begin{smallmatrix}0&0\\b&0\end{smallmatrix}]:b\in\Mat_c(\varpi^l\mathcal{O})\}$ for $l>0$ is not unique
when $|m|<1$ (in \cite[Lemma~85]{me12}, $|m|=1$), still any $f$ is right-invariant under right-translations by
$\{\langle y,1\rangle:y\in\mathcal{N}_l\}$ when $l\gg0$ (see \S~\ref{RS The functional equation} and \cite[p.~321]{BJ}).
We proceed to prove the analog of \cite[(4.66)]{me12}:
\begin{align*}
\gamma(s,\widetilde{\pi}\times\tau_0,\psi)\pi(\mathfrak{i}_{\GL_c})^{rk-1}\vartheta(s,\widetilde{\pi},\tau)^{-1}Z'(s,\omega,f)=\Psi(0,s,\omega,f).
\end{align*}
(Note that $\pi(\mathfrak{i}_{\GL_c})^{rk-1}=\vartheta(s,\widetilde{\pi},\tau)=1$ in the unramified case and
$\vartheta(s,\widetilde{\pi},\tau)=1$ when $r=1$.)
Indeed the substitution from \textit{loc. cit.} (see \cite[(4.67), (4.68)]{me12})
but with an arbitrary $W\in W_{\psi}(\rho_c(\tau_0))$, shows that
$Z'(s,\omega,f)$ equals the integral $Z(s,\omega,W)$ of \eqref{eq:Z integral GL 1 GL rk}, and for the same substitution
$\Psi(0,s,\omega,f)=Z^*(s,\omega,W)$ of \eqref{eq:second Z integral GL 1 GL rk} (see \cite[(4.69), (4.71)]{me12}). Together with
\eqref{eq:gamma RS} this completes the proof
(as opposed to \textit{loc. cit.} we do not need to compute $Z(s,\omega,W)$ and $Z^*(s,W,\omega)$, since we have \eqref{eq:gamma RS}).
\end{proof}
\begin{corollary}\label{eq:coro c=1 arbitrary k GL}
$\gamma(s,\pi\times(\tau_0\otimes\tau_0^*),\psi)=\gamma(s,\pi\times\tau_0,\psi)\gamma(s,\widetilde{\pi}\times\tau_0,\psi)$, i.e., \eqref{eq:GL factors} holds.
\end{corollary}
\begin{proof}
First observe that $f\mapsto W_f$ is the functional $\lambda(s,c,\tau_0\otimes\tau_0^*,\psi)$ of \eqref{eq:lambda functional} except that $\psi_{rk,c}^{-1}$ appearing in \eqref{eq:lambda functional} is replaced with $\psi_{rk,c}$. Denote this functional by
$\lambda_{-1}=\lambda_{-1}(s,c,\tau_0\otimes\tau_0^*,\psi)$.
Thus when we formally apply \eqref{eq:func equation for normalization} to $\Psi(\zeta,s,\omega,f)$ we obtain
\begin{align*}
\Psi(\zeta,s,\omega,f)=&\int\limits_{\GL_c}\int\limits_{\Mat_{((rk-2)c)\times c}}\int\limits_{\Mat_c}
\lambda_{-1}((\langle\diag(I_{(2rk-1)c},a),1\rangle\langle [t,v],\varsigma^-([t,v])\rangle
\langle w',1\rangle)\\&\nonumber\cdot M^*(s,c,\tau_0\otimes\tau_0^*,\psi)f)\omega(\langle a,1\rangle)|a|^{\zeta+(rk-1)c}\,dt\,dv\,d^*a.
\end{align*}
This application is justified by observing that the integrand of $\Psi(\zeta,s,\omega,f)$ is a Schwartz function of $[t,v]$.
Indeed let $Y_i=\{[0_{((i-1)c)\times c},y,0_{((rk-1-i)c)\times c}]:y\in\Mat_c\}$,
\begin{align*}
&X=\left\{\diag(I_{(rk-1)c},\left(\begin{smallmatrix}I_{(rk-1)c}&&x\\&I_c&\\&&I_c\end{smallmatrix}\right)):x\in\Mat_{((rk-1)c)\times c}\right\},
\end{align*}
write ${}^tx=(x_1,\ldots,x_{rk-1})$ where $x_j\in\Mat_c$
and let $X_i<X$ be the subgroup of elements whose coordinates other than $x_i$ are zero.
Any given $f$ is right-invariant under ${}^{{w'}^{-1}}\langle x,1\rangle$ whenever the coordinates of $x$ are sufficiently small.
Then for any $a\in\GL_c$, $y\in Y_i$ and $x\in X_i$, $1\leq i\leq rk-1$,
\begin{align*}
&\lambda_{-1}((\langle\diag(I_{(2rk-1)c},a),1\rangle\langle [y],\varsigma^-([y])\rangle
\langle w',1\rangle)\cdot f)
\\&=\lambda_{-1}((\langle\diag(I_{(2rk-1)c},a),1\rangle\langle [y],\varsigma^-([y])\rangle
\langle x,1\rangle\langle w',1\rangle)\cdot f)
\\&=\psi^{-1}(\tr xy)\lambda_{-1}((\langle\diag(I_{(2rk-1)c},a),1\rangle\langle [y],\varsigma^-([y])\rangle
\langle w',1\rangle)\cdot f).
\end{align*}
This implies the integrand is a Schwartz function of $[t,v]$ (see \cite[(4.58)--(4.61)]{me12} and also, e.g., \cite[Lemma~10]{CFK}).

Taking $\zeta=0$ we obtain, as meromorphic continuations,
\begin{align*}
\Psi(0,s,\omega,f)=&\int\limits_{\GL_c}\int\limits_{F^{rk-2}}\int\limits_{F}
\lambda_{-1}((\langle\diag(I_{(2rk-1)c},a),1\rangle\langle [t,v],\varsigma^-([t,v])\rangle
\langle w',1\rangle)\\&\nonumber\cdot M^*(s,c,\tau_0\otimes\tau_0^*,\psi)f)\omega(\langle a,1\rangle)|a|^{(rk-1)c}\,dt\,dv\,da.
\end{align*}
By Proposition~\ref{proposition:substitution},
\begin{align*}
\gamma(s,\widetilde{\pi}\times\tau_0,\psi)\pi(\mathfrak{i}_{\GL_c})^{rk-1}
\vartheta(s,\widetilde{\pi},\tau_0)^{-1}Z(s,\omega,f)
&=\Psi(0,s,\omega,f)\\&=\Psi(0,s,\omega,M^*(s,c,\tau_0\otimes\tau_0^*,\psi)f),
\end{align*}
and further applying the proposition to the r.h.s.~ we deduce
\begin{align*}
&\gamma(1-s,\widetilde{\pi}\times\widetilde{\tau}_0,\psi)^{-1}\gamma(s,\widetilde{\pi}\times\tau_0,\psi)
\vartheta(1-s,\widetilde{\pi},\widetilde{\tau}_0)\vartheta(s,\widetilde{\pi},\tau_0)^{-1}
Z(s,\omega,f)=Z^*(s,\omega,f).
\end{align*}
In addition according to \eqref{eq:RS dependence on psi} and \eqref{eq:RS functional equation} ($\psi_{-1}=\psi^{-1}$),
\begin{align*}
&\gamma(s,\pi\times\tau_0,\psi)\gamma(1-s,
\widetilde{\pi}\times\widetilde{\tau}_0,\psi)=
\pi(\langle (-1)^rI_c,1\rangle)^{k}\eta_{\tau_0}(\langle-1,1\rangle)^{c}=\pi(\mathfrak{i}_{\GL_c})^{rk}\tau_0(\mathfrak{i}_{\GL_k})^{rc}.
\end{align*}
For the second equality note that $-1\in F^{*r}$ (if $2|r$, $m=2r$) 
and $\eta_{\tau_0}(\langle a^r,1\rangle)=\rho_1(\tau_0)(\langle a^rI_{rk},1\rangle)=\tau_0(\langle a^{r}I_k,1\rangle)^{r}$ (for all $a\in F^*$).
The result follows.
\end{proof}

\subsection{Unramified factors}\label{unr}
Using \eqref{eq:multiplicativity I} and \eqref{eq:multiplicativity II}
we reduce to the case of $c=k=1$, i.e., $\GL_1^{(m)}\times \GL_1^{(m)}$ $\gamma$-factors, which follows from
\eqref{eq:GL factors} and \eqref{eq:RS unramified factors}. This proves \eqref{eq:unramified factors}.

\begin{corollary}\label{corollary:C factors}
When data are unramified,
\begin{align}\label{eq:C factor all unramified}
C(s,c,\tau,\psi)=
\frac{b(1-s,c,\widetilde{\tau})}{a(s,c,\tau)}\left[\frac{L(s,\tau)}{L(1-s,\widetilde{\tau})}\right].
\end{align}
Here the factor in square brackets appears only when $G=\Sp_{c}$ and $m$ is odd.
\end{corollary}
\begin{proof}
By \eqref{rep:rho c tau in general} and \eqref{eq:mult II of ops} we can already assume $\tau$ is tempered.
The condition \cite[(2.45)]{me12} holds by Lemma~\ref{lemma:unr unitary Satake}, then by \cite[Theorems~66, 87]{me12} when data are unramified
\begin{align}\label{eq:unramified integral}
Z(s,\omega,f)=\frac{L(s,\pi\times\tau)}{b(s,c,\tau)}.
\end{align}
Here if $G=\GL_c$, $\tau=\tau_0\otimes\tau_0^*$ and $L(s,\pi\times\tau)=L(s,\pi\times\tau_0)L(s,\widetilde{\pi}\times\tau_0)$.
Now \eqref{eq:C factor all unramified} follows by combining \eqref{eq:unramified integral} with \eqref{intertwining operator on unramified}, \eqref{eq:Gamma def} and \eqref{eq:unramified factors}, and when $G=\Sp_c$ and $m$ is odd we also use \eqref{eq:RS unramified factors} for the factor $\gamma(s,\tau,\psi)$ appearing in $\vartheta(s,c,\tau,\psi)$ (see \S~\ref{RS complete L factors}).
\end{proof}
\begin{remark}\label{remark:dont need tempered globally}
For our purposes here it suffices to know \eqref{eq:unramified integral} at almost all places in the setting of
\eqref{eq:global property}, and this holds without the need to assume $\tau$ is tempered.
\end{remark}

\subsection{The case $F=\C$}\label{Archimedean property}
Both $\pi$ and $\tau$ are quotients of principal series representations, say, with inducing data $\{\pi_j\}_j$ and $\{\tau_i\}_i$. Applying \eqref{eq:multiplicativity I}, \eqref{eq:multiplicativity II}, \eqref{eq:GL factors} and \eqref{eq:RS Archimedean property} we conclude $\gamma(s,\pi\times\tau,\psi)$ is the product of factors
$\gamma^{\mathrm{Tate}}(s,\tau_i^r\pi_j^r,\psi)\gamma^{\mathrm{Tate}}(s,\tau_i^r\pi_j^{-r},\psi)$, and an additional product of factors $\gamma^{\mathrm{Tate}}(s,\tau_i^r,\psi)$ when $G=\Sp_c$ and $m$ is odd.
By \cite[Theorem~27 (5.2), (5.3), (5.8)]{CFK}, $\gamma(s,\pi^r\times\tau^r,\psi)$ defined in \cite{CFK} equals 
$\gamma(s,\pi\times\tau,\psi)$ when $G\ne\Sp_c$ or $m$ is odd, and $\gamma(s,\tau,\psi)\gamma(s,\pi\times\tau,\psi)$ otherwise. 
Then when $G\ne\Sp_c$ or $m$ is odd,
\eqref{eq:Archimedean property} follows from \cite[Theorem~27 (5.9)]{CFK}, proved using the results of Shahidi \cite{Shahidi1985} (see \cite[\S~6.14]{CFK}).

Consider $G=\Sp_c$ and even $m$ (then $N=c$). We can already assume $\tau$ is tempered (see \cite{Shahidi1985}), i.e., $k=1$.
According to Gan and Savin \cite[Corollary~11.3(ii)]{WGS}, for $r=1$,
\begin{align*}
\gamma_{G^{(2)}}(s,\pi\times\tau,\psi)=\gamma_{\SO_{c+1}}(s,\theta(\pi)\times\tau,\psi),
\end{align*}
where
$\gamma_{G^{(2)}}$ is the $G^{(2)}\times\GL_1^{(2)}$ factor defined in \cite{Gan} and $\gamma_{\SO_{c+1}}$
is the $\SO_{c+1}\times\GL_1$ factor defined in \cite{LR,RS}. Over $\C$ the $\gamma$-factor is uniquely determined by \eqref{eq:multiplicativity I}, \eqref{eq:multiplicativity II} and \eqref{eq:GL factors}, hence $\gamma(s,\pi\times\tau,\psi)=\gamma_{G^{(2)}}(s,\pi^r\times\tau^r,\psi)$ for all $r$, and the $\gamma$-factor $\gamma(s,\theta(\pi^r)\times\tau^r,\psi)$ defined in \cite{CFK} (note that $\SO_{c+1}$ is split) is identical with $\gamma_{\SO_{c+1}}(s,\theta(\pi^r)\times\tau^r,\psi)$. Therefore
$\gamma(s,\pi\times\tau,\psi)=\gamma(s,\theta(\pi^r)\times\tau^r,\psi)$
and the result again follows from \cite[Theorem~27 (5.9)]{CFK}.

\subsection{The global integral}\label{The global integral}
The global integral for $G=\Sp_c$ was defined in \cite[\S~3.1]{me12} and for $\GL_c$ in \cite[Appendix~A]{me12}. We briefly recall the construction for $\Sp_c$. We use the notation from
\S~\ref{The integrals} and in particular \S~\ref{The doubling setup} in the global setting.
See \S~\ref{Global coverings} for the choice of a $2$-cocycle $\rho_{2rkc}$ on $H^{(m)}(\A)$, this defines the
splitting of $H(F)$ in $H^{(m)}(\A)$, and see \cite[\S~1.5]{me12} for the definitions of the $2$-cocycles on
each copy of $G^{(m)}(\A)$ and the identifications of $G(F)$ in each.

Let $\tau$ be a genuine cuspidal representation of $\GL_{k}^{(m)}(\A)$.
Let $f$ be a standard $\widetilde{K}_H$-finite section on the space of
$\Ind_{\widetilde{P}({\A})}^{H^{(m)}({\A})}(|\det|^{r^{-1}(s-1/2)}\mathcal{E}_{\tau})$, where $\mathcal{E}_{\tau}$ was defined in
\S~\ref{the representation rho_c(tau)}. We have the Eisenstein series $E(h;s,f)$, defined as a sum over
$P(F)\backslash H(F)$ (see \cite[(3.1)]{me12}).

Let $\pi$ be a genuine cuspidal representation of $G^{(m)}(\A)$ (the right copy of $G^{(m)}(\A)$, see \cite[\S~3.1]{me12}) and consider two cusp forms $\varphi_1$ and $\varphi_2$ in the space of $\pi$. The global integral is given by
\begin{align}\label{global1}
Z(s,\varphi_1,\varphi_2,f)=&\int\limits_{G(F)\times G(F)\backslash G({\A})\times G({\A})}\,
\int\limits_{U(F)\backslash U({\A})}\varphi_1^{(\eta^{\times})^{-1}}(\langle g_1,1\rangle)\,\overline{{}^{\iota}\varphi_2(\langle g_2,1\rangle)}
\\&\times E(\langle u,\eta_{2rkc}^{-1}(u)\rangle\nonumber
\langle \mathfrak{e}_1(g_1),1\rangle\langle \mathfrak{e}_2(g_2),1\rangle;s,f)\,\psi_U(u)\,du\,dg_1\,dg_2.
\end{align}
Here $\eta^{\times}$ is a global $1$-cochain defined in \cite[\S~1.5]{me12} to compensate for differences between the cocycles on the copies of $G^{(m)}(\A)$, so that $\varphi_1^{(\eta^{\times})^{-1}}$ is an automorphic form on the left copy (see \cite[Corollary~13]{me12}).
By \cite[Theorem~63]{me12} \eqref{global1} is well defined, absolutely convergent away from the poles of the series and admits meromorphic continuation to $\C$. For decomposable data, by \cite[Theorem~64 and (3.16)]{me12} (and Conjecture~\ref{conjecture:intertwining op Speh}),
in $\Real(s)\gg0$ \eqref{global1} is Eulerian. The similar results for $\GL_c$ were proved in
\cite[Theorems~A.2--A.3]{me12}. By \eqref{eq:unramified integral} (proved in \textit{loc. cit.} also when $\tau$ is non-tempered
under assumptions implied by Conjecture~\ref{Shimura conjecture}, see
Remark~\ref{remark:dont need tempered globally}),
\begin{align}\label{eq:global integral computation}
Z(s,\varphi_1,\varphi_2,f)=\frac{L^S(s,\pi\times\tau)}{b^S(s,c,\tau)}Z_S(s,\omega,f).
\end{align}
The superscript (resp., subscript) $S$ denotes the infinite (resp., finite) product of local factors over the places outside (resp., inside) $S$, with
$S$ as in \S~\ref{unr L functions},
e.g., $Z_S(s,\omega,f)=\prod_{\nu\in S}Z(s,\omega_{\nu},f_{\nu})$.

\subsection{Crude functional equation}\label{Crude functional equation}
We present the argument briefly, for more details see \cite[\S~6.9]{CFK}.
We use the definitions and notation of \S~\ref{The global integral}.
According to the functional equation of the Eisenstein series and \eqref{intertwining operator on unramified}, for
a standard $\widetilde{K}_H$-finite decomposable section $f$,
\begin{align}\label{eq:func eq of series}
E(\cdot;s,f)=E(\cdot;1-s,M(s,\mathcal{E}_{\tau},w_P)f)=\frac{a^S(s,c,\tau)}{b^S(s,c,\tau)}E(\cdot;1-s,f').
\end{align}
Here $M(s,\mathcal{E}_{\tau},w_P)$ is an intertwining operator --- the global analog of \eqref{int:intertwining operator}. By Conjecture~\ref{conjecture:intertwining op Speh}, $\mathcal{E}_{\tau}=\otimes_{\nu}'\rho_c(\tau_{\nu})$, and by
Lemma~\ref{lemma:rho c dual}, $f'_{\nu}(s,h)=M_{\nu}(s,\rho_c(\tau_{\nu}),w_P)f_{\nu}(s,h)$ is a section of $V(\rho_c(\tau_{\nu}^*))$.
Using \eqref{eq:func eq of series} and \eqref{eq:global integral computation} for $b^S(s,c,\tau)f$,
\begin{align}\label{eq:global func eq 0}
L^S(s,\pi\times\tau)Z_S(s,\omega,f)&=
\frac{a^S(s,c,\tau)}{b^S(1-s,c,\widetilde{\tau})}L^S(1-s,\widetilde{\pi}\times\widetilde{\tau})Z_S(1-s,\omega,f').
\end{align}

Let $\vartheta^{\circ}(s,c,\tau_{\nu},\psi_{\nu})$ be equal to $\vartheta(s,c,\tau_{\nu},\psi_{\nu})$ when $m$ is even or $G=\GL_c$, otherwise it is $\vartheta(s,c,\tau_{\nu},\psi_{\nu})\gamma(s,\tau_{\nu},\psi_{\nu})^{-1}$. In the latter case
because \eqref{eq:RS global property} implies $\gamma(s,\tau,\psi)=1$,
\begin{align}\label{eq:gamma vartheta circ reverse}
\gamma^S(s,\tau,\psi)^{-1}\vartheta^{\circ}_S(s,c,\tau,\psi)=\gamma^S(s,\tau,\psi)^{-1}\gamma(s,\tau,\psi)\vartheta^{\circ}_S(s,c,\tau,\psi)=
\vartheta_S(s,c,\tau,\psi).
\end{align}
Since $\prod_{\nu}\eta_{\tau_{\nu}}(\langle2,1\rangle)=\prod_{\nu\notin S}\eta_{\tau_{\nu}}(\langle2,1\rangle)=1$
(see the discussion following Theorem~\ref{theorem:RS ten commendments}),
we have $\vartheta^{\circ}(s,c,\tau,\psi)=(\vartheta^{\circ})^S(s,c,\tau,\psi)=1$. Also
$\pi(\mathfrak{i}_G)=\pi^S(\mathfrak{i}_G)=1$. Then by \eqref{eq:C factor all unramified}, \eqref{eq:RS unramified factors} and \eqref{eq:gamma vartheta circ reverse},
\begin{align}\label{eq:C^S}
C^S(s,c,\tau,\psi)=\frac{b^S(1-s,c,\widetilde{\tau})}{a^S(s,c,\tau)}[\gamma^S(s,\tau,\psi)^{-1}]=
\prod_{\nu\in S}\pi_{\nu}(\mathfrak{i}_G)^{rk}\vartheta_{\nu}(s,c,\tau_{\nu},\psi_{\nu})\frac{b^S(1-s,c,\widetilde{\tau})}{a^S(s,c,\tau)}.
\end{align}
The global analog of Theorem~\ref{theorem:uniqueness for normalizing intertwining} implies $C(s,c,\tau,\psi)=1$, and now
\eqref{eq:global property} follows from \eqref{eq:global func eq 0},
\eqref{eq:C^S} and \eqref{eq:Gamma def}.

\subsection{Duality}\label{Self-duality}
Identity~\eqref{eq:self-duality} is trivial for $\Sp_c$, and for $\GL_c$ follows immediately from \eqref{eq:GL factors}.
\subsection{Functional equation}
Assume $G=\Sp_c$. If $m$ is odd, by \eqref{eq:RS functional equation} and \eqref{eq:RS dependence on psi} and because
$\eta_{\varepsilon\otimes1}(\langle -1,1\rangle)=(\varepsilon\otimes1)(\langle -1,1\rangle)^r=1$,
we have $\gamma(s,\tau,\psi)\gamma(1-s,\widetilde{\tau},\psi)=\tau(\langle (-1)^r,1\rangle)$.
In addition $\eta_{\tau}\eta_{\tau^*}=1$ by \cite[Proposition~77]{me12}, thus
$\vartheta(s,c,\tau,\psi)\vartheta(1-s,c,\widetilde{\tau},\psi)=\tau(\langle (-1)^r,1\rangle)^N$.
Together with \eqref{eq:composition of intertwining operators} we obtain
$\gamma(s,\pi\times\tau,\psi)\gamma(1-s,\pi\times\widetilde{\tau},\psi)=\tau(\langle (-1)^r,1\rangle)^N$.
Then \eqref{eq:functional equation} follows from
\eqref{eq:self-duality} and \eqref{eq:dependence on psi}.
For $G=\GL_c$ apply \eqref{eq:GL factors} and \eqref{eq:RS functional equation}.

\section{$L$- and $\epsilon$-factors}\label{L and epsilon factors}
We define the local factors using Theorem~\ref{theorem:ten commendments}, as in \S~\ref{RS complete L factors}, following \cite{Sh3} (see also \cite{LR,CFK}).
Let $\pi$ and $\tau$ be as in \S~\ref{The local integrals}.
Assume $F$ is non-archimedean. For $G=\GL_c$ and when $\pi$ and $\tau$ are essentially tempered, define the
$L$- and $\epsilon$-factors as the products of $L$- and $\epsilon$-factors for $\pi\times\tau_0$ and $\widetilde{\pi}\times\tau_0$
defined in \S~\ref{RS complete L factors} (see \eqref{eq:GL factors}). For $G=\Sp_c$ and tempered $\pi$ and $\tau$, define
$L(s,\pi\times\tau)=1/P(q^{-s})$ for $P(X)\in\C[X]$ with $P(0)=1$ such that the zeroes of $P(q^{-s})$ coincide with those of $\gamma(s,\pi\times\tau,\psi)$.
By \eqref{eq:dependence on psi}, $L(s,\pi\times\tau)$ is independent of $\psi$. We then apply \eqref{eq:functional equation} to define
$\epsilon(s,\pi\times\tau,\psi)\in \C[q^{-s},q^s]^*$. The essentially tempered case (for $\tau$) is defined
via \eqref{eq:Unramified twisting}.

In general we utilize the Langlands classification:
Let $(\otimes_{i=1}^{d'}\sigma_i)\otimes\pi'$ be Langlands' data for $\pi$, where $\pi'$ is omitted when $G=\GL_c$ or $l=n$ using the notation of
\eqref{eq:multiplicativity I};
and let $\otimes_{j=1}^d\tau_j$ be Langlands' data for $\tau_0$ ($\tau=\tau_0\otimes\tau_0^*$ for $\GL_c$, $\tau=\tau_0$ for $\Sp_c$). Then
\begin{align}\label{eq:RS L function essentially tempered}
&L(s,\pi\times\tau)=\prod_{i,j}L(s,\sigma_i\times\tau_j)L(s,\widetilde{\sigma}_i\times\tau_j)\prod_jL(s,\pi'\times\tau_j).
\end{align}
Here the product $\prod_jL(s,\pi'\times\tau_j)$ is included only when $G=\Sp_c$ and either $\pi'$ appears in the inducing data of $\pi$,
or $\pi'$ does not appear but $m$ is odd in which case $L(s,\pi'\times\tau_j)$ is taken to be $L(s,\tau_j)$. The factor $\epsilon(s,\pi\times\tau,\psi)$ is defined using the same product of $\epsilon$-factors.

When data are unramified, by Lemma~\ref{lemma:unr unitary Satake}, \eqref{eq:unramified factors}
and the definitions we deduce that $L(s,\pi\times\tau)$ equals the
unramified $L$-function of \S~\ref{unr L functions} and  $\epsilon(s,\pi\times\tau,\psi)=1$.

For $F=\C$ define $L(s,\pi\times\tau)=L(s,\mathrm{st}\circ\varphi)$ and
$\epsilon(s,\pi\times\tau,\psi)=\epsilon(s,\mathrm{st}\circ\varphi,\psi)$ with the notation of Theorem~\ref{theorem:ten commendments}.

According to \eqref{eq:multiplicativity I}, \eqref{eq:multiplicativity II}, \eqref{eq:Unramified twisting} and \eqref{eq:Archimedean property}, over any local field,
\begin{align}\label{eq:final local functional equation}
\gamma(s,\pi\times\tau,\psi)=\epsilon(s,\pi\times\tau,\psi)\frac{L(1-s,\widetilde{\pi}\times\widetilde{\tau})}{L(s,\pi\times\tau)}.
\end{align}

Finally we define the complete $L$-function for $G=\Sp_c$. Let $\pi$ and $\tau$ be genuine cuspidal representations of $G^{(m)}(\A)$ and $\GL_k^{(m)}(\A)$. Define $L(s,\pi\times\tau)=\prod_{\nu}L(s,\pi_{\nu}\times\tau_{\nu})$.
\begin{theorem}\label{theorem:complete Sp L function}
The $L$-function $L(s,\pi\times\tau)$ is absolutely convergent for $\Real(s)\gg0$, admits meromorphic continuation to $\C$ and satisfies a functional equation $L(s,\pi\times\tau)=\epsilon(s,\pi\times\tau)L(1-s,\widetilde{\pi}\times\widetilde{\tau})$.
\end{theorem}
\begin{proof}
The proof is similar to the proof of Theorem~\ref{theorem:complete RS GL L function}.
Briefly, the convergence and meromorphic continuation follow from the properties of the global integral \eqref{global1}, \eqref{eq:global integral computation} and the local integrals (see \S~\ref{The local integrals}), note that $b^S(s,c,\tau)$ (in \eqref{eq:global integral computation}) admits meromorphic continuation by \cite{Gao2018}. The functional equation follows
from \eqref{eq:global property}, \eqref{eq:final local functional equation} and \eqref{eq:dependence on psi}.
\end{proof}

\section{Shimura lift from $\Sp_{c}^{(2)}(\A)$ to $\GL_{N}(\A)$}\label{functorial lift}
Let $G=\Sp_c$, $N=c$ ($c=2n$).
Let $F$ be a totally imaginary number field with a ring of adeles $\A$, and $S_{\infty}$ be the set of complex places.
Let $\pi$ be a genuine cuspidal representation of $G^{(2)}(\A)$.
\begin{theorem}\label{theo:globl functorial lift}
The representation $\pi$ of $G^{(2)}(\A)$ has a Shimura lift $\Pi$ to $\GL_{N}(\A)$.
\end{theorem}

To prove the theorem, we construct the candidate lift $\Pi$ as in \cite{CKPS,CFK}, then verify the
conditions of the Converse Theorem of \cite[\S~2]{CKPS2} (the twisted version of \cite{CPS3}).

Fix a nonempty finite set of finite places $S_{\pi}$ such that
for all finite $\nu\notin S_{\pi}$, $\pi_{\nu}$ is unramified.
We use $\vartheta=\gamma_{\psi}$ to parameterize $\pi$ (see \S~\ref{unr L functions}).
For a finite $\nu\notin S_{\pi}$, $\pi_{\nu}$ is a constituent of a genuine unramified representation $\Ind_{\widetilde{B}_{G}(F_{\nu})}^{G^{(2)}(F_{\nu})}(\otimes_{i=1}^n\varepsilon\otimes\vartheta\mu_i)$ and
$\Pi_{\nu}$ is the irreducible unramified constituent of $\Ind_{B_{\GL_N}(F_{\nu})}^{\GL_N(F_{\nu})}(\otimes_{i=1}^n\mu_i^2\otimes\mu_i^{-2})$.
For $\nu\in S_{\infty}$, let $\Pi_{\nu}$ be the local functorial lift (to $\GL_N(F_{\nu})$) of
the representation $\theta(\pi_{\nu})$ attached to $\pi_{\nu}$ by the theta correspondence
(for $\theta(\pi_{\nu})$ see \cite{Howe1989,AdamsBarbasch1995}; for the lift see \cite{Bo,La3} and \cite[\S~5.1]{CKPS}).
For $\nu\in S_{\pi}$, let $\Pi_{\nu}$ be an arbitrary irreducible representation of $\GL_N(F_{\nu})$ with a trivial central character.
Put $\Pi=\otimes_{\nu}'\Pi_{\nu}$. The central character of $\Pi$ is trivial.

The product $L(s,\Pi)$ is absolutely convergent for $\Real(s)\gg0$, because if $\tau_0$ is the ``standard" representation of $\GL_1^{(2)}(\A)$ defined in \S~\ref{RS complete L factors} (i.e., $\tau_{0,\nu}$ corresponds to
$\varepsilon\otimes\vartheta_{\nu}$), $L^{S}(s,\Pi)=L^{S}(s,\pi\times\tau_0)$ for $S=S_{\pi}\cup S_{\infty}$.

Let $\eta$ be an automorphic character of $\A^*$, which is sufficiently highly ramified over all $\nu\in S_{\pi}$, and let $\mathscr{A}_0(S_{\pi},\eta)$ be the set of cuspidal representations $\tau'$ of $\GL_k(\A)$ such that for $\nu\in S_{\pi}$,
$\tau'_{\nu}$ is the twist of an unramified representation by $\eta_{\nu}$, and $1\leq k<N$. For $\tau'\in\mathscr{A}_0(S_{\pi},\eta)$, let $\tau=\vartheta\otimes\tau'$ be the genuine cuspidal representation of $\GL_k^{(2)}(\A)$ corresponding to $\tau'$ (see \S~\ref{covering GL r=1}). For each $\nu\in S_{\pi}$, if $\tau'_{\nu}=\eta_{\nu}\Ind_{B_{\GL_k}(F_{\nu})}^{\GL_k(F_{\nu})}(\otimes_{i=1}^k\chi_{\nu,i})$ where $\chi_{\nu,i}$ are unramified, $\vartheta_{\nu}\otimes\tau'_{\nu}$ is either unramified if $|2|=1$, or ramified when $|2|<1$ but the ramification level of $\vartheta_{\nu}\eta_{\nu}\chi_{\nu,i}$ still depends only on $\eta_{\nu}$, because $\vartheta_{\nu}$ is fixed independently of $\tau'$.
\begin{proposition}\label{proposition:L funcs pi and Pi}
$L(s,\pi\times\tau)=L(s,\Pi\times\tau)$, $\epsilon(s,\pi\times\tau)=\epsilon(s,\Pi\times\tau)$ and
$L(1-s,\widetilde{\pi}\times\widetilde{\tau})=L(1-s,\Pi^{\vee}\times\tau^{\vee})$. Moreover,
the $L$-factors for $\nu\in S_{\pi}$ are trivial.
\end{proposition}
\begin{proof}
The equality between the local $L$- and $\epsilon$-factors for $\nu\notin S_{\pi}$ follows
from the proof of \cite[Lemma~48]{CFK}. The proof extends to $G^{(2)}$ using
\eqref{eq:Unramified twisting}--\eqref{eq:multiplicativity II}, \eqref{eq:final local functional equation},
the definitions in \S~\ref{L and epsilon factors} and in particular \eqref{eq:RS L function essentially tempered},
Lemma~\ref{lemma:unr unitary Satake}, and because the doubling $\gamma$-factor for $\GL_c^{(2)}(F_{\nu})\times\GL_k^{(2)}(F_{\nu})$
coincides with the linear one from \cite{CFK} (thereby implying \cite[(7.2)]{CFK})
since $\GL_c^{(2)}(F_{\nu})$ is split over $\GL_c(F_{\nu})$.
The similar equality for $\nu\in S_{\infty}$ holds by the definitions and
\eqref{eq:Archimedean property}.

Let $\nu\in S_{\pi}$. We omit $\nu$ and use local notation, i.e., $\pi=\pi_{\nu}$,
and recall $\vartheta=\gamma_{\psi}$. Let $\theta_{\psi}(\pi)$ be the irreducible representation
of $\SO_{c+1}$ given by the theta correspondence, where $\SO_{c+1}$ is either split or non-split
(\cite[Theorem~1.1]{WGS}, the assumption that the residue characteristic is odd was removed by \cite{GanTakeda2016}). Also let $\pi_0$ be an arbitrary irreducible representation of $\SO_{c+1}$.
Assume $k=1$. Then $\tau=\vartheta\otimes\eta\chi$ where $\chi$ is an unramified
quasi-character of $F^*$. According to \cite[Theorem~1.3]{WGS},
$\gamma(s,\pi\times\tau,\psi)$ coincides with the $\gamma$-factor $\gamma_{\SO_{c+1}}(s,\theta_{\psi}(\pi)\times\eta\chi,\psi)$ for
$\SO_{c+1}\times\GL_1$ defined in \cite{LR,RS}. For the latter, the stability results of \cite{RS,CKPS}
imply $\gamma(s,\theta_{\psi}(\pi)\times\eta\chi,\psi)=\gamma(s,\Pi\times\eta\chi,\psi)$.
Now we proceed exactly as in \cite[Lemma~51]{CFK} to obtain the result on the $L$- and $\epsilon$-factors for all $k\geq1$.
\end{proof}
\begin{remark}
The theta correspondence does not seem to play a crucial role here: One can extend the arguments of
\cite{RS} to obtain stability for the $\Sp_{c}^{(2)}\times\GL_1^{(2)}$ $\gamma$-factor. In the generic setting
Zhang \cite{Zhang2017} proved stability when $q$ is odd for the Rankin--Selberg $\gamma$-factor from
\cite{me6}.
\end{remark}
Thus $L(s,\Pi\times\tau)=\epsilon(s,\Pi\times\tau)L(1-s,\Pi^{\vee}\times\tau^{\vee})$ follows from Proposition~\ref{proposition:L funcs pi and Pi} and Theorem~\ref{theorem:complete Sp L function}.
\begin{proposition}
$L(s,\Pi\times\tau)$ and $L(1-s,\Pi^{\vee}\times\tau^{\vee})$ are entire.
\end{proposition}
\begin{proof}
As in \cite[Theorem~68]{CFK} consider \eqref{eq:global integral computation} and multiply both sides by $b^S(s,c,\tau)$ ($S$, not $S_{\pi}$).
The function $b^S(s,c,\tau)$ is entire, because $b^S(s,c,\tau')$ is entire by \cite[Proposition~2.1]{KimShahidi2002} ($\tau'$ is not self-dual).

Next we see that $Z(s,\varphi_1,\varphi_2,f)$ is holomorphic in $\Real(s)\geq1/2$. This follows from \cite[Theorem~66]{CFK}, where we proved
the Eisenstein series is holomorphic given the condition on $\tau'$. Since $r=1$, the section $f$ belongs to the space of
$\Ind_{\widetilde{P}({\A})}^{H^{(2)}({\A})}(|\det|^{s-1/2}\mathcal{E}_{\tau})$ where $M_P=\GL_{kc}$ as in the linear case. Hence the constant term computation from \cite[Proposition~2.3]{JiangLiuZhang2013} is still valid, and the (global and local) intertwining operators
(see \cite[(10.2)]{CFK}) only involve the representation $\tau$ of $\GL_k^{(2)}$, so that their analytic results can be deduced
from the linear case. The remaining poles to consider (when we argue as in the proof of \cite[Theorem~66]{CFK}) are those in the strip
$1/2<\Real(s)<1$. Let $\rho(h)$ be the leading coefficient of the Laurent expansion of $E(h;s,f)$ about $s$ in this strip.
The computation of the exponents of the series (\cite[Lemma~65]{CFK}) still applies, because $r=1$ and using
\cite[(10.1)--(10.3), (10.6)]{CFK}, so that $\rho$ is square-integrable. Now we apply \cite[\S~III.3.1]{MW2} which is valid for a large class of covering groups including $\Sp_c^{(m)}$ for all $m$, as in \cite[Theorem~66]{CFK}.

Specifically, let $N(H)\geq1$ be the integer defined in \cite[\S~III.3.2]{MW2}, and $A_{\nu_0}$ be the subgroup of $a\in \A^*$ such that
$a_{\nu}=1$ for all $\nu\ne\nu_0$ and $a_{\nu_0}\in\mathcal{O}_{\nu_0}^{*2}$ for some $\nu_0\in S_{\pi}$. 
Assume the isomorphism $H(\A)[\sigma^{\mathrm{Rao}}]\rightarrow H(\A)[\rho_{2kc}]$ is given by $\langle h,\epsilon\rangle\mapsto\langle h,\phi(h)\epsilon\rangle$ where
$\phi\in\mathrm{C}^1(H(\A),\mu_2)$ (see \S~\ref{Global coverings} and also \eqref{eq:RS rho top rk}). Put $d(a)=\diag(I_{(c-1)k},aI_k)$. Then 
$A'=\{\langle d(a),\phi(d(a))\rangle:a\in A_{\nu_0}\}$ is a subgroup of $C_{\widetilde{M}_{(k^c)}(\A)}\cap H^{(2)}(\A)$. We can now assume 
$\tau\otimes\ldots\otimes\tau$ is a representation of $M_{(k^c)}(\A)[\sigma^{\mathrm{Rao}}]$. 
The restriction of the $N(H)$-th power of the central character of $\tau\otimes\ldots\otimes\tau$ to $A'$ is given by 
$\tau_{\nu_0}^{N(H)}(\langle a_{\nu_0}I_k,1\rangle)=(\vartheta_{\nu_0}\eta_{\nu_0})^{N(H)}(a_{\nu_0}^k)$, which is not positive real 
if $\eta_{\nu_0}$ is sufficiently highly ramified independently of $\tau$ ($\eta$ may depend on $\phi$). 
Therefore $\rho$ can not be square-integrable.

It remains to show that if $\pi_{\nu}$ is unramified or $\nu\in S_{\infty}$, each pole of $L(s,\pi_{\nu}\times\tau_{\nu})$
at $\Real(s)\geq1/2$ occurs with the same multiplicity in
$Z(s,\omega_{\nu},f_{\nu})$ for an entire section $f_{\nu}$, which is also $\widetilde{K}_{H_{\nu}}$-finite if $\nu\in S_{\infty}$.
We obtain this by applying \cite[Lemma~58]{CFK}, whose proof is applicable here. In more detail, for $\nu\notin S_{\infty}$,
\cite[Corollary~49]{CFK} is valid using Lemma~\ref{lemma:unr unitary Satake}, then the poles we seek are poles of
Rankin--Selberg $L$-functions for (generic) representations of $\GL_l^{(2)}(F_{\nu})\times \GL_k^{(2)}(F_{\nu})$. Because these coverings are
trivial, we can obtain the poles using \cite[Theorem~C.12]{CFK}. Note that the manipulations of the integrals in the proof of \cite[Lemma~58]{CFK} are similar to those
applied for the proof of \eqref{eq:multiplicativity I} and mostly involve unipotent subgroups. (Alternatively one can use
\cite[Lemma~53]{CFK} and \cite[Claims~54--55]{CFK} which are applicable here because
$\GL_k^{(2)}(F_{\nu})$ is split over $\GL_k(F_{\nu})$; the arguments mimic the unramified computation).
The archimedean results are valid because $F_{\nu}=\C$ and $\rho_c(\tau_{\nu})$ is $(rk,c)=(k,c)$ for $r=1$.
\end{proof}

\begin{proposition}\label{proposition:BVS}
$L(s,\Pi\times\tau)$ and $L(1-s,\Pi^{\vee}\times\tau^{\vee})$ are bounded in vertical strips of finite width.
\end{proposition}
\begin{proof}
This follows as in \cite[Corollary~71]{CFK} from \cite[Theorem~70]{CFK}, which is based
on the arguments of \cite[Proposition~1]{GL}. The properties of $L_{S_{\infty}}(\pi\times\tau)$ needed for the proof still follow
from \cite[Lemma~50]{CFK}. It remains to know $L^S(s,\pi\times\tau)$ is a meromorphic function of finite order, which is
Theorem~\ref{theorem:Muller} in Appendix~\ref{Muller}.
\end{proof}

The existence of a Shimura lift $\Pi'$ now follows from the Converse Theorem \cite[\S~2]{CKPS2}, namely, there exists an irreducible automorphic representation $\Pi'$ of $\GL_N(\A)$ such that $\Pi'_{\nu}\cong\Pi_{\nu}$ for all $\nu\notin S_{\pi}$.
The proof of Theorem~\ref{theo:globl functorial lift} is complete.

When we combine this theorem with the characterization of the image of global functoriality of Ginzburg,
Rallis and Soudry \cite[\S~11]{RGS} using their descent method (for generic representations), and with the strong multiplicity one result for isobaric representations (\cite{JS2,JS1})
we conclude the analog of \cite[Corollary~7.1]{CKPS}, that no data is lost at the places of $S_{\pi}$:
\begin{corollary}\label{corollary:globl functorial lift}
If $\pi$ is globally $\psi$-generic such that its $\psi$-theta lift to $\SO_{N-1}(\A)$ vanishes, $\Pi$ is the unique lift of $\pi$.
\end{corollary}

\appendix
\section{The partial $L$-function $L^S(s,\pi\times\tau)$ is meromorphic of finite order}\label{Muller}

The purpose of this section is to prove, in the setup of general covering groups, that partial $L$-functions occurring in the constant term of Eisenstein series with cuspidal data are meromorphic functions of finite order. This is one of the main ingredients towards proving that entire complete $L$-functions are bounded in
vertical strips of finite width, a property which is one of the conditions of the Converse Theorem (e.g., \cite{CPS3}).

In the linear case, for globally generic representations boundedness in vertical strips was proved by Gelbart and Shahidi \cite{GS2}, using deep results of complex analysis. Gelbart and Lapid \cite{GL} generalized and sharpened some of the results of \cite{GS2}, while providing a proof
which underlines the role of M{\"u}ller's bounds \cite{Mu89,Mu00} on the growth of Eisenstein series (M{\"u}ller's bounds were also crucial for \cite{GS2}). To us, the advantage of the approach of \cite{GL} is that their proof, being independent of the Whittaker model, lends itself to ``transparent extensions" to covering groups (once the local theory is fully developed).

Let $G$ be a connected reductive group defined over a totally imaginary field $F$.
Assume $\widetilde{G}(\A)$ is a topological central extension of $G(\A)$ by $\mu_m$, satisfying the properties from \cite[\S~I.1]{MW2} (this includes the coverings of \cite{Moore,Stein,Mats,BD}), realized using a $2$-cocycle $\rho\in\mathrm{Z}^2(G(\A),\mu_m)$.
Fix a splitting $y\mapsto\langle y,\eta(y)\rangle$ of $G(F)$ in $\widetilde{G}(\A)$.
Recall $\widetilde{G}(\A)\cong \{(\epsilon_{\nu})_{\nu}\in\mu_m:\prod_{\nu}\epsilon_{\nu}=1\}\backslash \prod_{\nu}'\widetilde{G}_{\nu}(F_{\nu})$,
where $\prod_{\nu}'\widetilde{G}_{\nu}(F_{\nu})$ is a restricted direct product with respect to, say, $(K_{\nu},\eta_{\nu})_{\nu\notin S}$, $K_{\nu}=G(\mathcal{O}_{\nu})$,
$y_{\nu}\mapsto\langle y_{\nu},\eta_{\nu}(y_{\nu})\rangle$ is a splitting of $K_{\nu}$ in $\widetilde{G}(F_{\nu})$, and
$S$ is a sufficiently large finite set. Let $G(\A_{\mathrm{fin}})=\prod_{\nu\notin S_{\infty}}'G(F_{\nu})$ and
$G_{\infty}=\prod_{\nu\in S_{\infty}}G(F_{\nu})$. Our assumption on $F$ implies
$\widetilde{G}_{\infty}$ is split over $G_{\infty}$. Thus we can assume
$\rho$ and the $1$-cochain $\eta$ are trivial on $S_{\infty}$, i.e.,
$\rho=\rho^{S_{\infty}}(=\prod_{\nu\notin S_{\infty}}\rho_{\nu}$) and similarly
$\eta=\eta^{S_{\infty}}$. It also follows that for
$x_{\infty}\in G_{\infty}$ and $y\in G(\A_{\mathrm{fin}})$,
\begin{align*}
\langle x_{\infty},1\rangle
\langle y,1\rangle
=\langle x_{\infty}y,1\rangle=\langle y,1\rangle\langle x_{\infty},1\rangle.
\end{align*}
In particular $x_{\infty}$ and $y$ commute in $\widetilde{G}(\A)$.

M{\"u}ller \cite{Mu89} carried out his analysis on the space $L^2(\Gamma\backslash G_{\infty})$ where $\Gamma$ is an arithmetic group, then used the correspondence between automorphic forms on $G(\A)$ and $G_{\infty}$ given in \cite[\S~4.3(2)]{BJ1977}, to derive an adelic version of the Trace Class Conjecture proved in \cite[Corollary~0.2]{Mu89} (see \cite[pp.~527--528]{Mu89}). We start by extending (the formal part of) this
correspondence to $\widetilde{G}(\A)$.

Let $K_0<G(\A_{\mathrm{fin}})$ be a compact open subgroup such that $\widetilde{G}(\A)$ is split over $K_0$, and fix a splitting $y\mapsto\langle y,\varsigma(y)\rangle$ of $K_0$ in $\widetilde{G}(\A)$. Denote the projection of $G(F)$ into $G(\A_{\mathrm{fin}})$ by $G(F)_{\mathrm{fin}}$, and for $g\in G(\A)$ put $\Gamma_g=(G_{\infty}\times{}^gK_0)\cap G(F)$.

Since $\eta=\eta^{S_{\infty}}$, $y\mapsto\langle y,\eta(y)\rangle$ is a splitting of $G(F)_{\mathrm{fin}}$ in $\widetilde{G}(\A)$.
For $y\in K_0$, write ${}^g\langle y,\varsigma(y)\rangle=\langle {}^gy,\varsigma^g(y)\rangle$. If ${}^gy\in G(F)_{\mathrm{fin}}$,
${}^gy\mapsto \langle {}^gy,\varsigma^g(y)\rangle$ and ${}^gy\mapsto\langle {}^gy,\eta({}^gy)\rangle$ are both splittings of
${}^gK_0\cap G(F)_{\mathrm{fin}}$, thus $\mu^g(y)=\varsigma^g(y)\eta^{-1}({}^gy)$ is a character of ${}^gK_0\cap G(F)_{\mathrm{fin}}$. In addition (because $\eta=\eta^{S_{\infty}}$) the
image of the projection of $\{\langle x,\eta(x)\rangle:x\in\Gamma_g\}$ into $\widetilde{G}_{\infty}$ is contained in $\{\langle x_{\infty},1\rangle : x_{\infty}\in G_{\infty}\}$.

Consider a genuine function $f$ on $G(F)\backslash \widetilde{G}(\A)/K_0$, where $G(F)$ and $K_0$ are identified with subgroups of
$\widetilde{G}(\A)$ using $\eta$ and $\varsigma$ (resp.). Then
\begin{align*}
f(\langle g,\epsilon\rangle)=
f(\langle g,\epsilon\rangle\langle y,\varsigma(y)\rangle)=
\mu^g(y)f(\langle g,\epsilon\rangle),\qquad\forall\,y\in {}^gK_0\cap G(F)_{\mathrm{fin}}.
\end{align*}
Hence $f$ vanishes on the preimage of $gG_{\infty}$ in $\widetilde{G}(\A)$, unless
$\mu^g$ is trivial.

Write $G(\A)=G(F)\mathcal{G}G_{\infty}K_0$ where $\mathcal{G}\subset G(\A)$ is a finite set of elements whose components at $S_{\infty}$ are all trivial, and let $\mathcal{G}^{1}=\{g\in\mathcal{G}:\mu^g=1\}$. For $f$ as above and $g\in\mathcal{G}$, let $f_g$ be the (genuine) function on $\widetilde{G}_{\infty}$ defined by
$f_g(\langle x_{\infty},\epsilon\rangle)=f(\langle g,1\rangle\langle x_{\infty},\epsilon \rangle)$. If $z_{\infty}\in G_{\infty}$ and $y\in K_0$ satisfy $z_{\infty}{}^gy\in \Gamma_g$,
\begin{align*}
f_g(\langle z_{\infty},1\rangle \langle x_{\infty},1\rangle)=
f(\langle z_{\infty}{}^gy,\varsigma^g(y)\rangle\langle g,1\rangle\langle x_{\infty},1\rangle)=
\mu^g(y)f_g(\langle x_{\infty},1\rangle)=f_g(\langle x_{\infty},1\rangle),
\end{align*}
where the last equality holds either because $f_g$ vanishes on both sides, or $\mu^g(y)=1$.

In the opposite direction assume $(f_g)_{g\in\mathcal{G}^{1}}$ is a genuine function on $\coprod_{g\in\mathcal{G}^{1}}(\Gamma_g\backslash \widetilde{G}_{\infty})$. Define a genuine function $f$ on
$G(F)\backslash \widetilde{G}(\A)/K_0$ to be $0$ outside the preimage of $G(F)\mathcal{G}^{1}G_{\infty}K_0$, and
$f(h)=f_g(\langle x_{\infty},\epsilon\rangle)$ where
\begin{align*}
h=\langle z,\eta(z)\rangle\langle  g,\epsilon\rangle\langle  x_{\infty},1\rangle \langle y,\varsigma(y)\rangle
,\qquad z\in G(F),\quad g\in\mathcal{G}^{1},\quad x_{\infty}\in G_{\infty},\quad y\in K_0.
\end{align*}
The equivariance properties are satisfied, once we prove $f$ is well defined. To this end assume 
\begin{align*}
h=\langle z',\eta(z')\rangle\langle  g',\epsilon'\rangle\langle  x_{\infty}',1\rangle \langle y',\varsigma(y')\rangle
\end{align*}
with similar notation. Then $g=g'$ and $(x_{\infty}'x_{\infty}^{-1}){}^g(y'y^{-1})=z'^{-1}z\in G(F)$, whence 
$f_g(\langle x_{\infty},1\rangle)=f_g(\langle x_{\infty}',1\rangle)$. It remains to show $\epsilon=\epsilon'$. 
Comparing both representations of $h$ in $\widetilde{G}(\A)$ we have 
\begin{align*}
&\langle (x_{\infty}'x_{\infty}^{-1}){}^g(y'y^{-1}),\varsigma^g(y'y^{-1})\varsigma^{-1}(y'y^{-1})\rho(y',y^{-1})\rho(y,y^{-1})^{-1}\varsigma(y')\varsigma^{-1}(y)\epsilon'\epsilon^{-1}\rangle
\\&=\langle z'^{-1}z,\rho(z'^{-1},z)\rho(z',z'^{-1})^{-1}\eta^{-1}(z')\eta(z)\rangle.
\end{align*}
Hence
\begin{align*}
&\varsigma^g(y'y^{-1})\varsigma^{-1}(y'y^{-1})\rho(y',y^{-1})\rho(y,y^{-1})^{-1}\varsigma(y')\varsigma^{-1}(y)\epsilon'\epsilon^{-1}
=\rho(z'^{-1},z)\rho(z',z'^{-1})^{-1}\eta^{-1}(z')\eta(z).
\end{align*}
Because $g\in\mathcal{G}^{1}$, $\mu^g=1$ whence $\varsigma^g(y'y^{-1})=\eta({}^g(y'y^{-1}))$, and because 
$\varsigma(y_1)\varsigma(y_2)\rho(y_1,y_2)=\varsigma(y_1y_2)$ for all $y_1,y_2\in K_0$ and
$\varsigma$ is trivial on the identity,
\begin{align*}
&\eta({}^g(y'y^{-1}))
\epsilon'\epsilon^{-1}
=\rho(z'^{-1},z)\rho(z',z'^{-1})^{-1}\eta^{-1}(z')\eta(z).
\end{align*}
In addition since ${}^g(y'y^{-1})$ coincides with the projection of $z'^{-1}z$ into $G(\A_{\mathrm{fin}})$ and
$\eta=\eta^{S_{\infty}}$, $\eta({}^g(y'y^{-1}))=\eta(z'^{-1}z)=\rho(z'^{-1},z)\eta(z'^{-1})\eta(z)$, then using
$\rho(z',z'^{-1})^{-1}\eta^{-1}(z')=\eta(z'^{-1})$ we conclude $\epsilon=\epsilon'$.

Also observe that the space of genuine functions on
$\Gamma_g\backslash \widetilde{G}_{\infty}$ is isomorphic to the space of functions on
$\Gamma_g\backslash G_{\infty}$.
To summarize:
\begin{appcorollary}\label{corollary:Muller bijection}
The map $f\mapsto(f_g)_{g\in\mathcal{G}^{1}}$ is a bijection between genuine functions on
$G(F)\backslash \widetilde{G}(\A)/K_0$ and
$\coprod_{g\in\mathcal{G}^{1}}(\Gamma_g\backslash G_{\infty})$.
\end{appcorollary}

\begin{apptheorem}\label{theorem:Muller}
Assume $G$ is split over $F$, $M$ is the Levi part of a maximal parabolic subgroup $P=M\ltimes U$ of $G$ and $\pi$ is a genuine cuspidal representation of $\widetilde{M}(\A)$. The partial $L$-function $L^S(s,\pi)$ defined in \cite{Gao2018} is a meromorphic function of finite order.
In particular with the notation of \eqref{eq:global property}, $L^S(s,\pi\times\tau)$ is a meromorphic function of finite order.
\end{apptheorem}
\begin{proof}
Let $E(g;\varphi,s)$ be the Eisenstein series attached to the representation induced from $\widetilde{P}(\A)$ and $\pi$ to $\widetilde{G}(\A)$.  We use notation as in \cite[\S~2.2]{GL} except that $M(\A)$ is replaced with $\widetilde{M}(\A)$,
$M(F)$ is regarded as a subgroup of $\widetilde{M}(\A)$ by first embedding it in $G(F)$ then using the splitting fixed above, and the maximal compact subgroup is $\widetilde{K}$, where $K$ is a maximal compact subgroup of $G(\A)$ as in \cite[\S~I.1.1]{MW2} (see
\cite[\S~I.1.4]{MW2}). The section $\varphi$ is standard, i.e., its restriction to $\widetilde{K}$ is independent of $s$, and also $\widetilde{K}$-finite.

We argue as in the proof of \cite[Theorem~2]{GL}: the inductive step \cite[Proposition~4.1]{Sh3} was extended to the covering groups of \cite{BD} in \cite[\S~8.3]{Gao2018}, and it remains to prove:
\begin{itemize}[leftmargin=*]
\item There exists a nontrivial entire function $q(s)$ of finite order and constants $c_0,l>0$, such that for any compact set $\mathcal{C}\subset G(\A)$, there is $c_1>0$ satisfying
    $|q(s)E(g;s,\varphi)|\leq c_1e^{c_0|s|^l}$ for all $g\in \mathcal{C}$ and $s\in\C$.
\end{itemize}
In the linear setting this result was contained in \cite[Theorem~0.2]{Mu00}.
The proof was based on growth bounds for (normalized) intertwining operators, and on the
inner product formula of Langlands for truncated Eisenstein series (\cite{Langlands1966,A2}).
We describe the extension of these to covering groups. Note that the general properties of
Eisenstein series and global intertwining operators were developed simultaneously for
reductive groups and their coverings in \cite{MW2} (see \cite[\S~II.1.5--7, \S~IV]{MW2}).

For the intertwining operators, the main result \cite[Theorem~2.4]{Mu00} was that cuspidal intertwining operators between associated parabolic subgroups can be normalized by entire functions of finite order, so that the normalized operators are entire operator-valued functions of finite order. The rank-$1$ case \cite[Theorem~2.1]{Mu00} was obtained by reinterpreting the similar result of \cite{Mu89}
for automorphic forms on $G_{\infty}$, in the adelic setting, using the correspondence of \cite[\S~4.3(2)]{BJ1977}
(see \cite[pp.~1129--1131]{Mu00}). This correspondence extends to $\widetilde{G}(\A)$ when we apply the formal
Corollary~\ref{corollary:Muller bijection} to automorphic forms. While for us the rank-$1$ case is sufficient
(because $P$ is maximal), we note that the result for arbitrary
parabolic subgroups now follows as well, since the proof only involves the multiplicativity of
intertwining operators, \cite[Lemma~1.1]{Arthur1981} and \cite[(1.4), (1.5)]{Arthur1982AJM}, all of which still apply.

For the remaining part of the proof (\cite[\S~3]{Mu00}) we utilize the following observations.
\begin{enumerate}[leftmargin=*]
\item The truncation operator of Arthur \cite{A2} is also defined in the context of covering groups, and the truncated Eisenstein
series is square-integrable (\cite[\S~I.2.13, \S~IV.2.4]{MW2}).
\item Li \cite[\S~3]{Li2013} extended the inner product formula of Langlands for truncated Eisenstein series to coverings groups
(and actually treated the more general asymptotic formula of \cite{Arthur1982}).
\item The singularities of Eisenstein series and intertwining operators lie on root hyperplanes (\cite[\S~IV.1.11]{MW2}).
\item The general decompositions of spaces of automorphic forms (\cite[\S~1.6]{Mu00}) remain valid,
in particular the spaces $A^2_{\mathrm{cusp}}(P,\chi,\sigma)\cap A^2(P)(\pi)$ in the notation of \textit{loc. cit.} are finite dimensional (see \cite[\S~3]{Li2013}),
this is needed for the proof of \cite[Corollary~3.4]{Mu00}. \qedhere
\end{enumerate}
\end{proof}

\begin{remark}
For this work there is no need to consider an arbitrary number field. Moreover, in light of \cite{GanIchino2018} applications of this section will probably be to coverings of rank $m>2$, so that taking
a totally imaginary field is not a compromise.
\end{remark}

\def\cprime{$'$} \def\cprime{$'$} \def\cprime{$'$}

\end{document}